\documentclass[a4paper,11pt]{article}

\usepackage{amsmath}
\usepackage{amsthm}
\usepackage{amssymb}
\usepackage{bbm}
\usepackage{hyperref}
\usepackage{lscape}
\usepackage{stackengine}
\usepackage{xcolor}

\numberwithin{equation}{section}

\newtheorem{coro}{Corollary}
\newtheorem{thm}[coro]{Theorem}

\newtheorem*{nb}{Remark}

\newtheorem{lem}{Lemma}[section]
\newtheorem{defi}[lem]{Definition}
\newtheorem{prop}[lem]{Proposition}
\newtheorem{cl}[lem]{Claim}
\newtheorem{cor}[lem]{Corollary}

\newcommand{\ssun}{s_{10}}
\newcommand{\deltaun}{\delta_{10}}

\newcommand{\sinit}{s_{11}}

\newcommand{\AAzero}{A_{12}}
\newcommand{\Szero}{s_{12}}

\newcommand{\sszeroun}{s_{13}}

\newcommand{\sszerodeux}{s_{14}}

\newcommand{\Cdix}{C_{15}}
\newcommand{\sdix}{s_{15}}

\newcommand{\ssquatre}{s_{16}}

\newcommand{\scinco}{s_{17}}

\newcommand{\squatre}{s_{18}}

\newcommand{\bMsix}{M_{19}}
\newcommand{\bAsix}{A_{19}}
\newcommand{\etazerozero}{\eta_{19}}
\newcommand{\squatro}{s_{19}}

\newcommand{\bAun}{A_{20}}
\newcommand{\etaun}{\eta_{20}}
\newcommand{\bsun}{s_{20}}

\newcommand{\donzeb}{\delta_{21}}
\newcommand{\sonzeb}{s_{21}}

\newcommand{\squatrep}{s_{22}}

\newcommand{\bMsixp}{M_{23}}
\newcommand{\bAsixp}{A_{23}}
\newcommand{\etazerop}{\eta_{23}}
\newcommand{\sonzep}{s_{23}}

\newcommand{\bAunp}{A_{24}}
\newcommand{\etaunp}{\eta_{24}}
\newcommand{\bsunp}{s_{24}}

\newcommand{\Avingt}{A_{25}}
\newcommand{\deltavingt}{\delta_{25}}
\newcommand{\mvingt}{m_{25}}
\newcommand{\svingt}{s_{25}}

\newcommand{\Avingtetun}{A_{26}}
\newcommand{\svingtetun}{s_{26}}

\newcommand{\Czerop}{C_{27}}

\newcommand{\Chuit}{C_{28}}

\newcommand{\Msix}{M_{29}}
\newcommand{\bAsixs}{A_{29}}
\newcommand{\etazero}{\eta_{29}}
\newcommand{\ssix}{s_{29}}
\newcommand{\Msixp}{M_{29}'}
\newcommand{\Chuitp}{C_{29}}
\newcommand{\uMsix}{M_{29}''}

\newcommand\1{\mathbbm 1}
\newcommand\barbelow[1]{\stackunder[1.2pt]{$#1$}{\rule{.8ex}{.075ex}}}
\newcommand{\Id}{{\mathop{\rm Id}}}
\newcommand{\m}[1]{\mathbbm{#1}}
\newcommand{\q}[1]{\mathcal{#1}}

\newcommand{\suno}{0}
\newcommand{\sunos}{\sigma^*}
\newcommand{\sdue}{\sigma_1}
\newcommand{\stella}{s_*}

\newcommand{\Sun}{{S_0}}
\newcommand{\Sdeux}{{S_1}}
\newcommand{\Strois}{{S_2}}
\newcommand{\Aquatre}{{A_3}}
\newcommand{\equatre}{{\eta_3}}
\newcommand{\Squatre}{{S_3}}
\newcommand{\ecinq}{{\eta_4}}
\newcommand{\Asept}{{A_5}}
\newcommand{\esept}{{\eta_5}}
\newcommand{\Ssept}{{S_5}}
\newcommand{\ehuit}{{\eta_6}}

\title{\textbf{On degenerate blow-up profiles for the subcritical semilinear heat
    equation}}
\author{Frank Merle\\
{\it \small CY Cergy Paris Universit\'e and IHES}\\
Hatem Zaag\\
{\it \small 
  Universit\'e Sorbonne Paris Nord,  LAGA, CNRS (UMR 7539)},\\
 {\it \small F-93430, Villetaneuse, France}
}

\begin{document}

\maketitle

\begin{abstract}
We consider the semilinear heat equation with a superlinear power
nonlinearity in the Sobolev subcritical range. We construct a 
solution which blows up in finite time only at the origin, with a
completely new blow-up profile, which is cross-shaped. Our method
is general and extends to the construction of other 
solutions blowing up only at the origin, with a large variety of
blow-up profiles, degenerate or not.
\end{abstract}

\medskip



\medskip


\section{Introduction}

We consider the following subcritical semilinear heat equation
\begin{equation}\label{equ}
\partial_t u = \Delta u +|u|^{p-1}u,
\end{equation}
where $u:(x,t)\in \m R^N \times[0,T)\to \m R$, $T>0$,
\begin{equation}\label{condp}
p>1\mbox{ and } (N-2)p< N+2.
\end{equation}
We consider a solution  $u(x,t)$ blowing up in finite time $T>0$:
\[
\|u(t)\|_{L^\infty}\to \infty\mbox{ as }t\to T,
\]
and $a\in \m R^N$ a blow-up point of $u(x,t)$: $|u(a,t)|\to \infty$ as
$t\to T$.\\
From Giga and Kohn \cite{GKiumj87}
and Giga, Matsui and Sasayama \cite{GMSiumj04}, we know that all
blow-up solutions are Type 1 in the subcritical case:
\begin{equation}\label{type1}
\forall t\in [0,T),\;\;\;\|u(t)\|_{L^\infty}\le C(T-t)^{-\frac 1{p-1}}
\mbox{ for some }C>0.
\end{equation}

In order to simplify the exposition, we assume
\[
N=2
\]
and focus on the simplest case of the new profiles we handle in this
paper (see \eqref{A4} below).
Other examples for $N=2$ and some extensions to the
case $N\ge 3$ are given in Section \ref{secext}.
Following  Giga and Kohn in \cite{GKcpam85}, we introduce
\begin{equation}\label{defw}
w_a(y,s) = (T-t)^{\frac 1{p-1}}u(x,t) \mbox{ where } y = \frac{x-a}{\sqrt{T-t}}\mbox{ and }s=-\log(T-t).
\end{equation}
From Giga and Kohn \cite{GKcpam89}, we know that up to replacing $u$
by $-u$, we have
\begin{equation}\label{conv}
w_a(y,s) \to \kappa\equiv(p-1)^{-\frac 1{p-1}}\mbox{ as }s\to \infty,
\end{equation}
uniformly on compact sets.
According to Vel\'azquez \cite{Vtams93} (see also Filippas and Kohn
\cite{FKcpam92} together with Filippas and Liu \cite{FLihp93}), we may refine that
convergence and obtain the following ``blow-up profile'' $Q(y)$ such that
\begin{equation}\label{defQ0}
  w_a(y,s) -\kappa\sim Q(y,s) \mbox{ as }s\to \infty
\end{equation}
uniformly on compact sets,  with:\\
- either 
\begin{equation}\label{A2}
  Q(y,s) = -\frac \kappa{4ps}\sum_{i=1}^l h_2(y_l),
 \end{equation}
where $l=1$ or $2$, after a rotation of coordinates; keeping only the
leading terms in the polynomials involved in $Q(y,s)$, we obtain the following
quadratic form 
\begin{equation}\label{defmult2}
B(y) = \frac \kappa{4p} \sum_{i=1}^l y_i^2
\end{equation}
which is non zero and nonnegative;\\
- or 
\begin{equation}\label{Am}
  Q(y,s) =  -e^{-(\frac m2-1)s}\sum_{j=0}^m C_{m,j}h_{m-j}(y_1)h_j(y_2)
\end{equation}
as $s\to \infty$, for some even integer $m=m(a)\ge 4$, where $y=(y_1,y_2)$, $h_j(\xi)$ is the rescaled Hermite polynomial defined by 
\begin{equation}\label{defhj}
h_j(\xi) = \sum_{i = 0}^{\big[j/2\big]} \frac{j!}{i!(j - 2i)!}(-1)^i\xi^{j - 2i},
\end{equation}
and the multilinear form (obtained by keeping only the leading terms
of the polynomials of $Q(y)$)
\begin{equation}\label{defmult}
B(y) = \sum_{j=0}^m C_{m,j} y_1^{m-j}y_2^j
\end{equation}
is also non zero and nonnegative.

\medskip

If the origin is the only zero for the
multilinear form $B(y)$ defined in \eqref{defmult2}
and \eqref{defmult}, we are in the non-degenerate case. If not, we
are in the degenerate case. Accordingly, the corresponding blow-up
profile $Q(y)$ given in \eqref{A2} or \eqref{Am} will be said to be
non-degenerate or degenerate.

\medskip

Following the classification in \eqref{defQ0}, \eqref{A2} and
\eqref{Am}, a natural question arises: is $a$ an isolated blow-up
point or not?

\medskip

In the non-degenerate case, we know from Theorem 2 page 1570 in
Vel\'azquez \cite{Vcpde92} that $a$ is isolated. In particular, for any $b$ in a small ball
centered at $a$ with $b\neq a$, $u(b,t)$ has a finite limit denoted by
$u(b,T)$, as $t\to T$, with the following equivalents as $b\to a$:
\begin{align*}
u(b,T) &\sim \left[\frac{(p-1)^2}{8p}
  \frac{|x-a|^2}{|\log|x-a||}\right]^{-\frac 1{p-1}}&&\mbox{ if \eqref{A2}
  holds  with }l=2;\\
  u(b,T) & \sim \left[ \frac{(p-1)^2}\kappa B(x-a)\right]^{-\frac
           1{p-1}} &&
 \mbox{ if \eqref{Am} holds with } B(y)>0 \mbox{ for }y\neq 0.
\end{align*}

In the degenerate case, the situation is less clear. In fact, the only
known examples are ``artificial'', in the sense that the
solution depends only on one space variable, say $\omega\cdot x$ where
$\omega \in \m S^1$. A radial solution blowing-up outside the origin
in 2 dimensions gives rise to a degenerate situation too.
n these two examples, the blow-up point is not
isolated. Apart from these trivial examples, no more solutions with a
degenerate profile are known.

\medskip

Following this, we wonder whether there exists a solution obeying
\eqref{defQ0} such that \eqref{Am} holds for some $m\ge 4$ with a
degenerate multilinear form in \eqref{defmult} and an isolated blow-up
point. 
In this paper, we provide such an example, which is the first ever in
the subcritical range (see Theorem \ref{th1} below). Note that in the
supercritical case, Merle, Rapha\"el and Szeftel have already provided
in \cite{MRSimrn20}  an example of a signle-point blow-up solution with a
degenerate anisotropic profile, strongly relying on the existence of a stationary
solution to equation \eqref{eqw} below which decays to zero at infinity
(such a solution doesn't exist in the subcritical range).

\subsection{The existence question: state of the art and difficulties
  in the degenerate case}\label{secmatano}

We review here the question of the existence of blow-up solutions
obeying \eqref{A2} and \eqref{Am}.
Let us first mention that in the one dimensional case, the question
was positively answered by Bricmont and Kupiainen in \cite{BKnonl94} (see
also Herrero and Vel\'azquez \cite{HVdie92} for the case \eqref{Am}
with $m=4$).

\medskip

Let us go back to the two dimensional case and first focus on the non-degenerate case. The only examples we know
concern the case \eqref{A2} with $l=2$, 
thanks to Bricmont and Kupiainen 
\cite{BKnonl94} together with Merle and Zaag \cite{MZdmj97}. Such a
behavior is known to be stable with respect to perturbations in initial
data from \cite{MZdmj97} together with Fermanian, Merle and Zaag
\cite{FMZma00} and \cite{FZnonl00}.
Note that Herrero and Vel\'azquez showed the genericity of such a
behavior in \cite{HVasns92} and \cite{HVcras92} dedicated to the one
dimensional case, and in a non published document in higher space dimensions. 
We would like to mention also the solutions constructed by Nguyen
and Zaag in \cite{NZasens17}, showing a refinement of \eqref{A2} with
$l=2$.
As for the case \eqref{Am} with a non degenerate multilinear
form in \eqref{defmult}, no example is available, up to our knowledge,
not even for the symmetric cases with $B(x)=|x|^m$ or
$B(x)=x_1^m+x_2^m$ with an even $m\ge 4$.

\medskip

Concerning the degenerate case in \eqref{Am} with a degenerate
multilinear form in \eqref{defmult},
the only cases we know are the one-dimensional trivial cases we have
just mentioned above, showing a non-isolated blow-up point.
Apart from these trivial examples, no more solutions with a
degenerate profile are known.

\medskip

As stated earlier, the main goal of the paper is to provide an example
of a blow-up solution obeying \eqref{Am} in the degenerate case with
an isolated blow-up point.

\medskip

Let us mention that the question of having non trivial solutions with  degenerate
profiles was mentioned by Hiroshi Matano, 
because of the failure of  the formal computation.
As a matter of fact, the strategy used in the non degenerate case is
ineffective in the degenerate case, as we will explain below.

\medskip

Indeed, that strategy consists in
working
in the similarity variables setting
\eqref{defw}, where equation \eqref{equ} is transformed into the
following equation satisfied by $w_a$ (or $w$ for short): for all
$y\in \m R^2$ and $s\ge -\log T$,
\begin{equation}\label{eqw}
\partial_s w = \Delta w-\frac 12 y \cdot \nabla w -\frac w{p-1} +|w|^{p-1}w.
\end{equation}
For example, the idea used by Bricmont and Kupiainen in \cite{BKnonl94}
to construct their example in one space dimension with \eqref{Am}
which holds with $m=4$ 
consists in linearizing equation \eqref{eqw} around the following profile:
\begin{equation}\label{profy4}
(p-1+e^{-s}|y|^4)^{-\frac 1{p-1}}.
\end{equation}
Accordingly, if we intend to construct a solution obeying the
following degenerate estimate: 
\begin{equation}\label{A4}
w_0(y,s) -\kappa \sim - e^{-s}h_2(y_1)h_2(y_2) \mbox{ as }s\to \infty,
\end{equation}
uniformly on compact sets, a naive idea would be to linearize equation
\eqref{eqw} around the following profile:
\begin{equation}\label{profnaif}
\left(p-1+\frac{(p-1)^2}\kappa e^{-s}y_1^2 y_2^2\right)^{-\frac 1{p-1}},
\end{equation}
which already has the same expansion \eqref{A4} as the solution we intend to
construct.

\medskip

Unfortunately, a big problem arises with this profile, since it
doesn't decay to $0$ as $|y|\to \infty$, unlike the profile in
\eqref{profy4}. In fact, this smallness of the profile \eqref{profy4}
at infinity combined with the stability of the zero solution of
\eqref{eqw} is essential
in the control of the solution at infinity in space. In other words, with
the profile \eqref{profnaif}, we can't get such a control, and the
naive idea collapses, unless we can manage to get this decaying property.
Note that with the naive profile \eqref{profnaif},
the corresponding (approximate) solution is given by
$u(x,t)=(T-t)^{-\frac 1{p-1}}\left(p-1+\frac{(p-1)^2}\kappa
  \frac{x_1^2 x_2^2}{T-t}\right)^{-\frac 1{p-1}}$, which blows up
everywhere on the axes  $x_1=0$ and $x_2=0$.

\subsection{Main result of the paper}
\label{secstrat} 

In order to construct a solution obeying \eqref{A4}, following the previous subsection, a natural idea would be to refine
the expansion \eqref{A4} in order to get higher order terms ensuring
some decaying along the axes $y_i=0$, which are the degenerate
directions of the naive profile \eqref{profnaif}. Such a refinement is
given below in Lemma \ref{lemtay}. It shows that the term
$C_{6,0}(h_6(y_1)+h_6(y_2))$ may ensure that decaying property,
provided that we take $C_{6,0}=-\delta$ with $\delta>0$ and
large
(note that the parameter $C_{6,0}$ is free in the expansion of Lemma \ref{lemtay}).
Keeping only the leading terms of the polynomials, this leads
to the following refinement of the profile \eqref{profnaif}:
\begin{equation}\label{defPhi}
  \Phi(y,s)=    \left [p-1+\frac{(p-1)^2}\kappa\left( e^{-s}y_1^2y_2^2 + \delta
        e^{-2s}(y_1^6+y_2^6)\right)\right]^{-\frac 1{p-1}},
  \end{equation}
  as we will explain more in detail in Section \ref{secfor} below.
This refined version clearly has the decaying property at infinity
in space, which enables us to rigorously prove the existence of a
solution obeying the degenerate profile \eqref{A4} and blowing up only
at the origin, answering Matano's question: 
\begin{thm}[A blow-up solution for equation \eqref{equ} with a
  cross-shaped blow-up profile] \label{th1} When $N=2$, there exists
  $\delta_0>0$ such that for any $\delta\ge \delta_0$, there exists a solution $u(x,t)$ to
  equation \eqref{equ} which blows up in finite time $T$ only at the
  origin, with:\\
  (i) (Inner profile)
\begin{equation}\label{expw}
w_0(y,s) -\kappa \sim - e^{-s}h_2(y_1) h_2(y_2)\mbox{ as }s\to \infty,
\end{equation}
in $L^2_\rho(\m R^2)$ with
\begin{equation}\label{defro}
\rho(y) = \exp\left(-\frac{|y|^2}4\right)/(4\pi).
\end{equation}
and uniformly on compact sets.\\
(ii) (Intermediate profile):  For
  any $K>0$, it holds that
   \[
    \sup_{e^{-s}y_1^2y_2^2+\delta e^{-2s}(y_1^6+y_2^6) <K}
    \left|w(y,s)-\Phi(y,s)\right|\to 0 \mbox{ as } t\to T,
    \]
    where $\Phi(y,s)$ is defined in \eqref{defPhi}.\\
(iii) (Final profile): For any $x\neq 0$, $u(x,t)$ converges to a
finite limit $u(x,T)$ uniformly on compact sets of $\m
R^2\backslash\{0\}$ as $t\to T$, with
\begin{equation}\label{finalprof}
u(x,T) \sim\left[\frac{(p-1)^2}\kappa (x_1^2x_2^2+\delta(x_1^6+x_2^6))\right]^{-\frac 1{p-1}}\mbox{ as }x\to 0.
\end{equation}
\end{thm}
\begin{nb}
One may convince himself that the profiles in \eqref{defPhi} and
\eqref{finalprof} are cross-shaped, which justifies the title of the
theorem.
\end{nb}
\begin{nb}
 This is the first example of a blow-up solution in the subcritical
 range showing a degenerate
multilinear form in \eqref{defmult} with $m\ge 4$ and an isolated blow-up
point. Note that in the literature, many examples of blow-up solutions with only one
blow-up point are available (see Weissler \cite{Wjde84}, Bricmont and Kupiainen \cite{BKnonl94},
Merle and Zaag \cite{MZdmj97}, Nguyen and Zaag \cite{NZasens17}). Strikingly enough, almost all the known examples
are either radial or asymptotically radial, in the sense that they
approach a radially symmetric blow-up profile. Up to our knowledge,
the only exceptions 
hold in the Sobolev supercritical case, with
the Type 1 solution constructed by Merle, Rapha\"el and Szeftel
\cite{MRSimrn20}
(with a $\frac 1s$ rate in one direction and an exponentially decaying
rate in the other),
and also the Type 2 (i.e. non Type 1 \eqref{type1})
blow-up solution due to Collot, Merle and
Rapha\"el \cite{CMRjams20},
where the profile is anisotropic for both examples. No example is
available in the subcritical range.   
In Theorem \ref{th1},
we provide such an example,
with a non-radial solution blowing up only at the
origin, obeying the behavior \eqref{Am} with a degenerate multilinear
form in \eqref{defmult}, and an anisotropic blow-up profile.
\end{nb}
\begin{nb}
 Note that our method allows also to derive \textbf{new} solutions for
$N=2$ and also $N\ge 3$,
blowing up only at the origin, both in the degenerate and
the non-degenerate cases, showing a non necessarily radial profile in \eqref{Am} and
\eqref{defmult} (see below in Section \ref{secext}).
\end{nb}
\begin{nb}
  As we have just written before the statement of Theorem \ref{th1},
  our strategy relies on the refinement of the goal \eqref{A4} given
  in Lemma \ref{lemtay} below. 
Taking $C_{6,0}=-\delta$ where $\delta>0$ (and large) is
crucial in our argument. As a matter of fact, using our analysis and
the blow-up criterion we proved in \cite{MZcpam98} for solutions of
equation \eqref{eqw}, we show that the construction is impossible in
the case $C_{6,0}>0$. More precisely, if $u(x,t)$ is a symmetric solution (with
respect to the axes and the bissectrices) of equation
\eqref{equ}  blowing up at some time $T>0$ such that $u(0)\in
L^\infty(\m R^2)$ and estimate \eqref{expw} holds, then, estimate
\eqref{dlw0} holds with $C_{6,0}\le 0$ (note that a similar statement
holds without the symmetry assumption).
\end{nb}

\subsection{Extensions of the main result}\label{secext}

Following Theorem \ref{th1}, we would like to mention that our
strategy extends with no difficulty to the construction of solutions
to equation \eqref{equ} blowing up only at the origin with other types of profiles
(as defined in \eqref{defQ0}, \eqref{A2} and \eqref{Am})
summarized
in the following tables, respectively dedicated to the non-degenerate
and the degenerate cases, where $C_0=\frac{(p-1)^2}\kappa$ and
$\delta_0>0$ is large.

\medskip

Note that in the degenerate case, the idea behind the design of the
profiles we are presenting below is simple:\\
- First, considering the classification given in \eqref{defQ0} (and
its natural extension in higher dimensions), we choose the case
\eqref{Am} with some $Q(y)$ and a degenerate multilinear form $B(y)$ in
\eqref{defmult}; in Theorem \ref{th1}, we choose $Q(y)=e^{-s}h_2(y_1)h_2(y_2)$
and $B(y)=y_1^2y_2^2$.\\
- Then, we refine estimate \eqref{defQ0} by exhibiting higher order
terms, and select among them the polynomials which live on the
degenerate directions of $B(y)$; in Theorem \ref{th1}, the degenerate directions are
$y_1=0$ and $y_2=0$, the refinement of \eqref{defQ0} is given below in
\eqref{dlw0} and the polynomials we select in that refinement
of \eqref{defQ0} are $e^{-2s}h_6(y_1)$ and $e^{-2s}h_6(y_2)$.\\
- Finally, we design a profile similar to \eqref{defPhi}, with 2
terms, one corresponding to the degenerate multilinear form $B(y)$ defined in
 \eqref{defmult} (which is
$e^{-s}y_1^2y_2^2$ in Theorem \ref{th1}), and the other to the
polynomials we selected on the degenerate directions of $B(y)$ (which
is $e^{-2s}(y_1^6+y_2^6)$ in Theorem \ref{th1}). The key property of that profile is that
it is decaying to zero in all directions, though with different scales
according to whether we are in the degenerate or the non
degenerate directions of the multilinear form $B(y)$.

\medskip

Here are the 2 tables:
\begin{landscape}
\begin{center}
  \begin{table}
  \begin{tabular}{|c|c|l|}
    \hline
 &&\\    
   $Q(y,s)$
   & Equivalent of $u(x,T)$ as $x\to 0$&Conditions\\
 &&\\    
   \hline
            &&\\
    $-e^{(1-\frac k2)s} \displaystyle\sum_{i=0}^kC_{k,i}h_{k-i}(y_1)h_i(y_2)$
     & $\left[C_0B(x)\right]^{-\frac 1{p-1}}$ with
       $B(x) = \displaystyle\sum_{i=0}^k C_{k,i}x_1^{k-i}x_2^i$
   &
      \begin{tabular}{l}
        $N=2$, $\delta \ge \delta_0$,\\
      $k\ge 4$ is even and\\
$B(x)>0$
       for $x\neq 0$
       \end{tabular}\\
 &&\\
    \hline
          &&\\

$- e^{-(\frac k2-1)s}
\displaystyle\sum_{j_1+\dots+j_N=k}C_{k,j_2,,\dots,j_N}h_{j_1}(y_1)\dots h_{j_N}(y_N)$
     & $\left[C_0B(x)\right]^{-\frac 1{p-1}}$ with
$B(x)= \displaystyle\sum_{j_1+\dots+j_N=k}C_{k,j_2,\dots,j_N}x_1^{j_1}\dots x_N^{j_N}$
   &
      \begin{tabular}{l}
        $N\ge 3$, $\delta \ge \delta_0$,\\
      $k\ge 4$ is even and\\
$B(x)>0$
       for $x\neq 0$
       \end{tabular}\\
 &&\\
    \hline
   \end{tabular}
  \caption{Extensions in the non-degenerate case.}
  \end{table}
\end{center}
\end{landscape}
\begin{landscape}
   \begin{center}
\begin{tabular}{|c|c|l|}
    \hline
 &&\\    
   $Q(y,s)$
   & Equivalent of $u(x,T)$ as $x\to 0$&Conditions\\
 &&\\    
   \hline
 &&\\    
    $-e^{-s}h_2(y_1)h_2(y_2+a_0y_1)$
    & $\left[C_0\left(x_1^2(x_1+a_0x_2)^2+\delta(x_1^6+x_2^6)\right)\right]^{-\frac 1{p-1}}$
    & $N=2$, $\delta \ge \delta_0$, $a_0\in \m R$.\\
    &&\\    
    \hline
    &&\\
    $-e^{(1-\frac k2)s} h_{k_1}(y_1)\dots h_{k_l}(a_l y_1 +b_l y_2)$
    & $\left[C_0\left(x_1^{k_1}\dots (a_lx_1+b_lx_2)^{k_l}+\delta(x_1^{k+2}+x_2^{k+2})\right)\right]^{-\frac 1{p-1}}$
    &
      \begin{tabular}{l}
        $N=2$, $\delta \ge \delta_0$, $k=k_1+\dots+k_l$,\\
      $k_i\ge 2$ is even, $(a_i,b_i) \neq (0,0)$,\\
        the straight lines\\
        $\{y_1=0\}$,...., $\{a_iy_1+b_i y_2=0\}$\\
        are distinct.
      \end{tabular}
      \\
    &&\\
       \hline
    &&\\
    $-e^{(1-\frac k2)s}\displaystyle
    \prod_{i=1}^m\left(\sum_{j\in I_i}h_{2\theta_i}(y_j)\right)$
    &  $\left[C_0\left(
      \displaystyle 
      \prod_{i=1}^m\left(\sum_{j\in I_i}|x_j|^{2\theta_i}\right)
      + \delta(x_1^{k+2} + \dots +x_N^{k+2})
      \right) \right]^{-\frac 1{p-1}}$
 &
       \begin{tabular}{l}
         $N\ge 3$, $\delta \ge \delta_0$,
         $k=\displaystyle\sum_{i=1}^m 2 \theta_i$,\\
         and the set $I_i$ make a partition\\
         of $\{1,\dots,N\}$.
       \end{tabular}\\
    \\
    \hline
 \end{tabular}

  \bigskip
  
  Table 2 : Extensions in the degenerate case.
  \end{center}
\end{landscape}

Following these examples, we would to
 make some comments:\\
- Up to our knowledge, the only known examples such that \eqref{Am}
holds with some even $m\ge 4$ and a non-degenerate form in
\eqref{defmult} hold only in one space dimension. Those examples are
due to Bricmont and Kupiainen \cite{BKnonl94} (see also Herrero and
Vel\'azquez \cite{HVdie92} when $m=4$). Strikingly enough, we couldn't
find a proof in the literature for the existence of a solution with a
multilinear form $B(x) = |x|^4$ or $B(x) = x_1^4+\cdot+x_N^4$ when
$N\ge 2$. We are happy to say that our strategy provides such
examples, as stated in Table 1. By the way, the non-degenerate
multilinear forms $B(x)$ shown in that table need not be symmetric.\\
- The first example in Table 2 shows that the degeneracy directions of
the multilinear form $B(x)$ need not be orthogonal, unlike one may assume
from the constructed example in Theorem \ref{th1}.

 \bigskip

We proceed in several sections to prove Theorem \ref{th1}:\\
- first, in Section \ref{secfor},  we formally derive the profile in similarity variables;\\
- then, in Section \ref{secformu}, we explain our strategy for the
proof;\\
- in Section \ref{secdynw0}, we explain the dynamics of the equation
and suggest the general form of initial data;\\
- in Section \ref{seccontrolwa}, we prove a crucial $L^\infty$ bound
on the solution in similarity variables;\\
- in Section \ref{secmainres}, we conclude the proof of Theorem \ref{th1};\\
- finally, we prove the technical details in Section \ref{sectech}.

\section{Formal determination of the profile in similarity variables} \label{secfor}
In this paper, we ask whether we can have a solution $u(x,t)$ to equation
\eqref{equ} which blows up in finite time $T>0$ at the origin, such that
estimate \eqref{A4} holds, namely with 
\begin{equation*}
w_0(y,s) -\kappa \sim - e^{-s}h_2(y_1)h_2(y_2) \mbox{ as }s\to \infty,
\end{equation*}
uniformly on compact sets.
In order to simplify the calculations, we
will assume that $u(x,t)$ is symmetric with respect to the axes and
the bisectrices, for any
$t\in[0,T)$.

\medskip

As we have already discussed in Section \ref{secmatano}, applying
 the strategy of Bricmont and Kupiainen \cite{BKnonl94} is
ineffective, because the naive profile given in \eqref{profnaif} is
not decaying to zero along the axes $y_i=0$. Following what we wrote
in Section \ref{secstrat}, the idea to refine the naive guess in
\eqref{profnaif} goes through the refinement of the target behavior in
\eqref{A4}, which we give in the following:
\begin{lem}[Second order Taylor expansion]\label{lemtay}
  Following \eqref{A4} and assuming the solution is symmetric with
  respect to the axes and bissectrices, it holds that
  \begin{align}
    w_0(y,s) =&\kappa  - e^{-s} h_2h_2
    +e^{-2s}\left\{- \frac{32p}{3\kappa} h_0h_0
   - \frac{16p}\kappa (h_2h_0+h_0h_2)\right.\nonumber\\
&    -\frac{4p}\kappa (h_4h_0+h_0h_4)-\frac{32 p}\kappa h_2h_2
                                                +C_{6,0}(h_6h_0+h_0h_6) \label{dlw0}\\
              &\left. +(\frac{4p}\kappa s+C_{6,2})(h_4h_2+h_2h_4)
                +\frac p{2\kappa} h_4h_4\right\}+O\left(s^2
                e^{-3s}\right)\nonumber                
  \end{align}
  as $s\to \infty$, uniformly on compact sets, for some constants
  $C_{6,0}$ and $C_{6,2}$, where the notation $h_ih_j$ stands for $h_i(y_1)h_j(y_2)$.
  \end{lem}
  \begin{proof} The proof is omitted since it is straightforward from the method we
    explained in Proposition 1 
of our paper \cite{MZimrn21}.
  \end{proof}
  From this expansion, we remark that the term $C_{6,0}e^{-2s} (
  h_6(y_1)+h_6(y_2))$ will ensure the good decaying property, if
  $C_{6,0}<0$, especially on the axes, where the main term
  $e^{-s}y_1^2y_2^2$ is $0$. More precisely, taking $C_{6,0}=-\delta$ for
  some $\delta>0$ to be fixed large enough, then, keeping only the leading term
  in the polynomials, we naturally suggest the following modified
  version of \eqref{profnaif}
  \begin{equation}\label{goodprof}
\varphi(y,s) =\left[\frac ED\right]^{\frac 1{p-1}}
\end{equation}
with
\begin{align}
  E&=1+e^{-s}P(y)+e^{-2s}Q(y),\label{defE}\\
  D&=p-1+\frac{(p-1)^2}\kappa\left( e^{-s}y_1^2y_2^2 + \delta
     e^{-2s}(y_1^6+y_2^6)\right),\label{defD}
     \end{align}
where $P(y)$ and $Q(y)$ are polynomials which will be chosen so  that the
numerator in \eqref{defE} is positive (hence, $\varphi$ is well
defined), and 
\begin{align}
\varphi(y,s) &=\kappa  - e^{-s} h_2h_2 \label{target}\\
                                 & +e^{-2s}\left\{-\delta (h_6h_0+h_0h_6)
                                    +\gamma (h_4h_2+h_2h_4)
                +\frac p{2\kappa} h_4h_4\right\}
                 +O\left( e^{-3s}\right)\nonumber
            \end{align}
            as $s\to \infty$, uniformly on compact sets,  for some $\gamma$ to be fixed later.
            Note that the aimed expansion agrees with the prediction in
            Lemma \ref{lemtay} at the order $e^{-s}$ and most of the
            order $e^{-2s}$.
 Most importantly, we will also require  
that $\varphi(\cdot,s)\in L^\infty(\m R^2)$  and $\varphi(y,s) \to 0$ as $|y|\to
  \infty$, for any $s\ge 0$, and the choice of $\gamma$ will be
            crucial for that. 

  \medskip
  
            Let us explicit the choice of $\delta$, $\gamma$, $P(y)$ and $Q(y)$. Since \eqref{goodprof} directly implies that
            \begin{align}
              &\varphi(y,s)=\kappa +e^{-s}\left(\frac\kappa{p-1} P(y)-y_1^2y_2^2\right)\label{dlphi}\\
             & +e^{-2s}\left(\frac\kappa{p-1} Q(y) +\frac{\kappa(2-p)}{2(p-1)^2}P(y)^2
- \frac{P(y)}{p-1} y_1^2y_2^2-\delta y_1^6-\delta y_2^6+\frac
               p{2\kappa}y_1^4y_2^4\right)
               +O(e^{-3s})\nonumber
\end{align}
as $s\to \infty$, uniformly on compact sets, we see by identification
with \eqref{target} that
we must have
\begin{align}
  P(y) = &\frac {p-1} \kappa\left(y_1^2y_2^2-h_2h_2\right),\label{defP}\\
  Q(y) =&\frac{p-1}
  \kappa\left(\frac{P(y)}{p-1} y_1^2y_2^2
+\frac{\kappa(p-2)}{2(p-1)^2}P(y)^2
  +\delta (y_1^6-h_6(y_1))+\delta (y_2^6-h_6(y_2))\right.\nonumber\\
  &\left. +\frac p{2\kappa}(h_4h_4-y_1^4y_2^4)
          +\gamma(h_4h_2+h_2h_4)\right).  \label{defQ}
\end{align}
Now,
we claim the following:
\begin{lem}[Good definition, boundedness and decaying at infinity of
  $\varphi(y,s)$ \eqref{goodprof} - \eqref{defP} - \eqref{defQ}]\label{lemphi}
  Take  $\gamma =\frac{6p-2}\kappa$ and consider
$\delta\ge \deltaun$
and $s\ge \ssun(\delta)$
 for some large enough
$\deltaun\ge 1$
and $\ssun(\delta)\ge 0$.
 Then:\\
(i) $E\ge \frac 12$, hence
$\varphi(y,s)$ is well defined and positive.\\
(ii)  $\|\varphi(\cdot,s)\|_{L^\infty}\le \kappa +C_0e^{-\frac
  s3}$, for some $C_0>0$.\\
(iii) $\|\nabla \varphi(\cdot,s)\|_{L^\infty}\le C_0 e^{-\frac s 6}$.\\
(iv) $\varphi(y,s) \to 0$ as $|y|\to\infty$.
\end{lem}
   \begin{proof}
    Take $\gamma =\frac{6p-2}\kappa$ and consider $\delta\ge 1$ and $s\ge
    0$ to be taken large enough.\\
    (i) It's enough to show that both $P(y)$ and $Q(y)$ are bounded
    from below. Since
    \[
      h_2(\xi) = \xi^2-2,\;\;
     h_4(\xi) = \xi^4 - 12 \xi^2+12 \mbox{ and }
      h_6(\xi) = \xi^6-30\xi^4+180\xi^2-120
\]
from \eqref{defhj}, it follows from \eqref{defP} and \eqref{defQ} that
\begin{equation}\label{lowP}
  P(y) = \frac{2(p-1)}\kappa\left(y_1^2+y_2^2-2\right)
  \ge -\frac{4(p-1)}\kappa
\end{equation}
and
\begin{align}
  Q(y) =\frac{p-1}\kappa &\left(\frac 2
    \kappa\left[y_1^4y_2^2+y_1^2y_2^4\right]
                         +30\delta[y_1^4+ y_2^4]
                            +\frac
                           {6p}\kappa[-y_1^4y_2^2-y_1^2y_2^4 ]
                           +\gamma[y_1^4y_2^2+y_1^2y_4^2]
                            \right)
                            \nonumber\\
   &+O\left(\delta[1+|y|^2]+[1+|\gamma|][1+|y|^4]\right) \nonumber\\
 =\frac{p-1}\kappa &\left( \left(\frac 2\kappa-\frac{6p}\kappa+\gamma\right)\left[y_1^4y_2^2+y_1^2y_2^4\right]
                     +30\delta[y_1^4+ y_2^4] \right) \nonumber\\
                         &+O\left(\delta[1+|y|^2]+[1+|\gamma|][1+|y|^4]\right). \nonumber
\end{align}
Choosing $\gamma = \frac{6p-2}\kappa$,
we see that
\begin{equation}\label{expQ}
  Q(y) = \frac{30\delta(p-1)}\kappa [y_1^4+ y_2^4]
  +O\left(\delta[1+|y|^2]+[1+|y|^4]\right)
  \mbox{ as }|y|\to \infty,
\end{equation}
Taking $\delta$ large enough, then $|y|$ large enough, we wee that
\begin{equation}\label{lowQ}
Q(y) \ge \frac{15\delta(p-1)}\kappa [y_1^4+ y_2^4],
\end{equation}
which means that $Q(y)$ is bounded from below. Since \eqref{lowP}
implies that $P(y)$ is bounded from below too, taking $s$ large
enough, we see from \eqref{defE} that $E\ge \frac 12$, hence,
$\varphi$ defined in \eqref{goodprof} is well defined.\\
(ii) Using \eqref{lowP} and \eqref{expQ}, 
    we see from \eqref{goodprof} that
    \begin{equation}\label{total}
|\varphi(y,s)|^{p-1}\le C\frac{(N_1+N_2+N_3)}D
    \end{equation}
    where
    \begin{align}
      N_1&= 1+e^{-s},\;
      N_2 = \delta e^{-2s}|y|^4,\;
      N_3= e^{-s}|y|^2\label{defNi}
    \end{align}
    and $D$ is introduced in \eqref{defD}.
    Since $D\ge p-1$, it follows that
    \begin{equation}\label{tot1}
\frac{N_1}D\le \frac {1 +e^{-s}}{p-1}.
    \end{equation}
    Introducing $z=e^{-\frac s4}y$, we write
    \begin{equation}\label{tot2}
      \frac{N_2}D\le \frac{C\delta e^{-s}|z|^4}{1+z_1^2z_2^2+e^{-\frac
          s2}|z|^6}
      \le \frac{C\delta e^{-s}|z|^4}{e^{-\frac s2}+e^{-\frac s2}|z|^6}
      = \frac{C\delta e^{-\frac s2}|z|^4}{1+|z|^6}
      \le C \delta e^{-\frac s2}.
    \end{equation}
    Introducing $X=|z|^2$, we write
    \begin{equation}\label{n3d}
\frac{N_3}D\le \frac{C  e^{-\frac s2}|z|^2}{1+e^{-\frac
    s2}|z|^6}=C g_\epsilon(X),
\end{equation}
where
\begin{equation}\label{defgex}
\epsilon = e^{-\frac s2}, \;X=|z|^2 \mbox{ and }g_\epsilon(X)=\frac {\epsilon X}{1+\epsilon X^3}.
\end{equation}
Since $g_\epsilon'(X) = \frac{\epsilon(1-2\epsilon X^3)}{(1+\epsilon
  X^3)^2}$ which changes from negative to positive at $X=X_\epsilon
\equiv (2\epsilon)^{-\frac  13}$, it follows that
\[
g_\epsilon(X) \le g_\epsilon(X_\epsilon)=\frac{2\epsilon
  X_\epsilon}3=\frac{(2\epsilon)^{\frac 23}}3.
\]
Using \eqref{n3d} and \eqref{defgex}, we see that
\begin{equation}\label{tot3}
\frac{N_3}D \le Ce^{-\frac s3}.
\end{equation}
Collecting the estimates in \eqref{total}, \eqref{tot1}, \eqref{tot2}
and \eqref{tot3}, then recalling the definition
\eqref{conv}
of $\kappa$, we get the desired conclusion.\\
(iii)
By definition \eqref{goodprof} of $\varphi(y,s)$, we write for
all $s\ge 1$ and $y\in \m R^2$,
\[
(p-1)\log \varphi = \log E- \log D
\]
where $E$ and $D$ are defined in \eqref{defE} and \eqref{defD} (note
that $D>0$ by definition, and that $E>0$ from item (i) of this
lemma). Taking the gradient, we see that
\begin{equation}\label{defgrad}
\nabla \varphi =\frac{\varphi}{p-1}\left(\frac{\nabla E}{E}-\frac{\nabla D}{D}\right).
\end{equation}
Since
\begin{equation}\label{defD1}
 |\nabla D|\le C\delta e^{-2s}|y|^5+ Ce^{-s}|y_1|y_2^2+ Ce^{-s}|y_2|y_1^2\equiv D_1+D_2+D_3
\end{equation}
from \eqref{lowP} and \eqref{expQ}, we write from item (ii) of this lemma
\begin{equation*}
|\nabla \varphi|\le C\left[\frac{|\nabla E|}E+\frac {D_1}D+\frac {D_2}D+\frac {D_3}D\right].
\end{equation*}
Noting that 
\begin{align}
  E\ge &\frac {E_0}C\mbox{ where }E_0=1+ e^{-s}(y_1^2+y_2^2)+e^{-2s}(y_1^4+y_2^4),\label{lowE}\\
  |\nabla E|\le& C e^{-s}|y|+C\delta e^{-2s}(1+|y|^3),\nonumber
 \end{align}
from \eqref{lowP}, \eqref{lowQ} and \eqref{expQ},
using the definitions \eqref{defD} and \eqref{defD1} of $D$ and $D_1$,
then proceeding as for the
proof of item (ii) of this lemma, we show that
\begin{equation}\label{boundED}
\frac{|\nabla E|}E\le \frac{|\nabla E|}{E_0}\le C(\delta)e^{- \frac s3}\mbox{ and }\frac {D_1}D \le C(\delta)e^{- \frac s3}.
\end{equation}
By symmetry, it remains only to bound the term $\frac{D_2}D$, when
$y_2\neq 0$. Since
$\delta\ge 1$, using the
definitions \eqref{defD} and \eqref{defD1} of $D$ and $D_2$, we write
\begin{equation*}
\frac{D_2}D \le C\frac{e^{-s}|y_1|y_2^2}{1+e^{-s}y_1^2y_2^2+e^{-2s}y_2^6}.
\end{equation*}
As  a function in the variable $|y_1|$, the left-hand side realizes its maximum for
\[
|y_1|= \sqrt\frac{1+ e^{-2s}y_2^6}{e^{-s}y_2^2}.
\]
Therefore, it follows that
\[
\frac{D_2}D\le C \sqrt\frac{e^{-s}y_2^2}{1+ e^{-2s}y_2^6}\le C
e^{-\frac s6},
\]
in particular
\begin{equation}\label{boundD}
\frac{|\nabla D|}D\le Ce^{-\frac s6},
\end{equation}
and the estimate on $\nabla \varphi$ in item (iii) follows.\\
(iv) From \eqref{lowP}, \eqref{defQ} and \eqref{expQ}, we see that the numerator of
the fraction in \eqref{goodprof} is a polynomial of degree $4$,
whereas the denominator is of degree $6$. Thus, the conclusion
follows.\\
This concludes the proof of Lemma \ref{lemphi}.
\end{proof}

\section{Strategy of the proof}
\label{secformu}

From Section \ref{secfor}, we recall that
our goal is to construct $u(x,t)$, a solution of equation \eqref{equ}
defined for all $(x,t)\in \m R^2 \times [0,T)$ for some small enough
$T>0$, such that 
\begin{equation}\label{primary}
w_0(y,s) -\kappa \sim - e^{-s}h_2(y_1)h_2(y_2) \mbox{ as }s\to \infty,
 \end{equation}
uniformly on compact sets, 
where $w_0(y,s)$ is the similarity variables' version defined in
\eqref{defw}. We also recall our wish to construct a solution which is
symmetric with respect to the axes and the bisectrices.

\medskip

Consider $\varphi(y,s)$ our candidate for the profile,
defined in \eqref{goodprof}, \eqref{defP} and \eqref{defQ}, with
$\gamma=\frac{6p-2}\kappa$ and $\delta\ge 1$ fixed large enough, so that
Lemma \ref{lemphi} holds.
From the target expansion of $w_0$ given in Lemma \ref{lemtay} (with
$C_{6,0}=-\delta$ and $C_{6,2}=\gamma$) and the
expansion of the profile $\varphi$ given in \eqref{target}, we further
specify our goal by requiring that $q(y,s)$ is small in some sense
that will shortly given in \eqref{smallVA}, where 
\begin{equation}\label{defq}
q(y,s)=w_0(y,s) - \varphi(y,s).
\end{equation}
In order to achieve this goal, we need to write then understand the
dynamics of the equation satisfied by $q(y,s)$ near $0$. This is done
in the following section. 

\subsection{Dynamics for $q(y,s)$ defined in \eqref{defq}}

Roughly speaking, our first task is to perform local estimates for
bounded $y$, through a spectral analysis in $L^2_\rho$ near $\kappa$ \eqref{conv}, the 
constant solution of equation \eqref{eqw}, assuming a uniform bound
on the solution.

\medskip

More precisely, from \eqref{eqw}, we see that $q(y,s)$ satisfies the following
equation, for all $(y,s)\in \m R^2 \times [s_0, \infty)$, where
$s_0=-\log T$:
\begin{equation}\label{eqq}
\partial_s q = (\q L +V(y,s)) q +B(y,s,q) + R(y,s),
\end{equation}
with
\begin{align}
  \q L q = &\Delta q - \frac 12 y \cdot \nabla q +q,\;\;
 V(y,s) = p\varphi(y,s)^{p-1}-\frac p{p-1}, \label{defR}\\
 B(y,s,q) =& |\varphi(y,s)+q|^{p-1}(\varphi(y,s)+q)-\varphi(y,s)^p -
             p \varphi(y,s)^{p-1}q,\nonumber\\
  R(y,s)=&-\partial_s \varphi(y,s)+ (\q L-1) \varphi(y,s)-\frac {\varphi(y,s)}{p-1} +\varphi(y,s)^p. \nonumber
 \end{align}
Let us give in the following some useful properties of the terms
involved in equation \eqref{eqq}:\\
- \textit{The linear term}: 
Note that $\q L$ is a self-adjoint operator in $L^2_\rho$, the $L^2$
space with respect to the measure $\rho dy$ defined in \eqref{defro}.
 The spectrum of $\q L$ consists only in eigenvalues:
\begin{equation}\label{specL}
{\rm Spec }\;\q L = \{1-\frac m2\;|\;m\in \m N\}
\end{equation}
with the following set of eigenfunctions:
\begin{equation}\label{eigenL}
\{h_ih_j \equiv h_i(y_1)h_j(y_2)\;|\;i,j\in \m N\},
\end{equation}
which spans the space $L^2_\rho$, where we see the rescaled Hermite
polynomials \eqref{defhj}. Note also that
\begin{equation}\label{Lhihj}
\q L (h_ih_j) = \left(1-\frac{i+j}2\right)h_ih_j.
\end{equation}
- \textit{The potential}: It is bounded in $L^\infty$ and small in
$L^r_\rho$ for any $r\ge 1$, in the sense that
\begin{equation}\label{boundV}
  \|V(\cdot,s)\|_{L^\infty}\le \frac p{p-1}\mbox{ and }
  \|V(\cdot,s)\|_{L^r_\rho}\le C(r)e^{-s},
\end{equation}
for $s$ large enough, thanks to
Lemma \ref{lemphi} and a small Taylor expansion.\\
- \textit{The nonlinear term}: Since $\varphi$ is uniformly bounded in
space and time, thanks to item (iii) of Lemma \ref{lemphi}, assuming
the following a priori estimate, 
\begin{equation}\label{apriori}
\|q\|_{L^\infty(\m R^2\times [s_0, \infty))}\le M,
\end{equation}
we easily see that $B(y,s,q)$ is superlinear, in the sense that
\begin{equation}\label{boundB}
  |B(y,s,q)|\le C(M)|q|^{\bar p}\mbox{ where }\bar p =\min(p,2)>1,
\end{equation}
for all $y\in \m R^2$, $q\in \m R$ and large $s$.\\
- \textit{The remainder term}: By comparing the expression
\eqref{defR} of $R(y,s)$, with the expression of equation \eqref{eqw},
we see that $R(y,s)$ measures the quality of $\varphi(y,s)$ as an
approximate solution of \eqref{eqw}. In fact, through a straightforward calculation
(see below in Section \ref{secR}),
one can show that 
\begin{equation}\label{boundR}
\forall r\ge 2,\;\; \forall s\ge 0,\;\;\|R(s)\|_{L^r_\rho}\le C(r)e^{-2s},
\end{equation}
which is consistent with our approach in Section \ref{secfor}
(particularly Lemma \ref{lemtay} and estimate \eqref{target}),
where we constructed $\varphi$ as an approximate solution for equation \eqref{eqw}.

\medskip

Following these properties, constructing a small solution $q(y,s)$  to
equation \eqref{eqq} seems to be
reasonable, except for a serious issue: the control of the nonlinear
term $B(y,s,q)$, since the power function is not continuous in
$L^2_\rho$. The control of the potential term $Vq$ is delicate too. In
the following section, we explain our idea to gain those controls. 

\subsection{Control of the potential and the nonlinear terms in equation \eqref{eqq}}\label{secnonl}
In this section, we explain how we will achieve the control of the
potential and the nonlinear terms $Vq$ and $B$ in equation
\eqref{eqq}. 
For simplicity, we present the argument only for the nonlinear term
$B$, and restrict to the case when $p\ge 2$, hence $\bar p=2$ (the
actual proof for $Vq$ and $B$ will be done for any $p>1$).\\
Assuming the a priori estimate \eqref{apriori},
we may reduce equation \eqref{eqq} to the following linear equation
with a source term
\begin{equation}\label{eqlin}
  \partial_s q = (\q L+ \bar V)q +R(y,s)
\end{equation}
where
\begin{equation}\label{boundVb}
  |\bar V(y,s)|= \left |V(y,s) + \frac{B(y,s,q)}q\right|\le
C+C(M)|q|^{\bar p-1}\le   C+C(M)M^{\bar p-1}\equiv \bar C(M). 
\end{equation}
Thanks to the regularizing effect of the operator $\q L$ (see Lemma \ref{lemVel} below),
we may use equation \eqref{eqlin} together with \eqref{boundVb} and \eqref{boundR} to control the $L^4_\rho$ norm of the solution with the $L^2_\rho$ of the solution, up to some time delay:
\begin{equation}\label{delay}
\|q(s)\|_{L^4_\rho}\le e^{(1+\bar C)s^*}\|q(s-s^*)\|_{L^2_\rho}+Ce^{-2s}.
\end{equation}
Clearly, this is equivalent to writing that
\[
\|q(s)^2\|_{L^2_\rho}\le 2e^{2(1+\bar C)s^*}\|q(s-s^*)\|^2_{L^2_\rho}+2Ce^{-4s}, 
\]
which allows to control the nonlinear term $B(y,s,q)$ in $L^2_\rho$, thanks to \eqref{boundB}.

\bigskip

Note that this control is possible thanks to the a priori estimate \eqref{apriori}, which needs to be checked. We will explain that in the following section.

\subsection{Proof of the a priori bound \eqref{apriori}}

This is the key part of the argument, dedicated to the proof of the a
priori bound \eqref{apriori}, under which the local estimates presented
in the previous subsections hold. In fact, as we will shortly see from
the geometrical transform \eqref{wawb}, we reduce the uniform bound
for $w_0$ to a local estimate for $w_a$, where $a$ is arbitrary. Later
in Section \ref{seccontrolwa}, we will see that the control of $w_a$
will follow from the spectral analysis of equation \eqref{eqw}, in
particular the stability of its zero solution and its heteroclinic
orbit connecting $\kappa$ \eqref{conv} and $0$ and given below in \eqref{defpsi}.

\medskip

More precisely, 
since $\varphi$ is uniformly bounded by Lemma \ref{lemphi}, using the definition \eqref{defq} of $q$, it is enough to control $w_0(a,s)$, for any $a\in \m R^2$ and $s\ge s_0$, in order to prove \eqref{apriori}. Using the similarity variables' definition \eqref{defw}, we remark that
\begin{equation}\label{wawb}
  w_b(y,s) = w_0(y+be^{\frac s2},s).
\end{equation}
This way, we reduce the question to the control of $w_b(0,s)$, for any $b\in \m R^2$ and $s\ge s_0$. Since $w_b$ satisfies equation \eqref{eqw}, which is parabolic, we further reduce the question to the control of $\|w_b(s)\|_{L^2_\rho}$, for any $b\in \m R^2$ and $s\ge s_0$. This will be done through a careful choice of initial data (at $s=s_0$) for $w_0$, which completely determines initial data for $w_b$. Then, starting from these initial data, and integrating equation \eqref{eqw} for $s\ge s_0$, we will show that $w_b$ will decrease, providing the control on $\|w_b(s)\|_{L^2_\rho}$, hence on $w_b(0,s)=w_0(be^{\frac s2},s)$ (thanks to \eqref{wawb}), and finally on $\|w_0(s)\|_{L^\infty}$.

\subsection{Definition of a shrinking set to make $q(s)\to 0$}

From our formal analysis in Section \ref{secfor}, in particular the
target expansion of $w_0$ in Lemma \ref{lemtay}  (with
$C_{6,0}=-\delta$ and $C_{6,2}=\gamma$) and the expansion of
$\varphi$ in \eqref{target}, we may write the following target for
$q(y,s)$ defined in \eqref{defq}:
\begin{align}
q(y,s) =\frac p\kappa e^{-2s}
  \{&-32 h_0h_0-16(h_2h_0+h_0h_2)-4(h_4h_0+h_0h_4)
  -32h_2h_2\nonumber\\
  &+4s(h_4h_2+h_2h_4)\}+O(s^2e^{-3s})\mbox{ as } s\to \infty. \label{secondary}
  \end{align}
 In fact, we will not require such a sharp expansion.   We will instead
  require that the components of $q(y,s)$ are bounded by the same rate
  as the one
  in front of the corresponding eigenfunctions shown on the right-hand side of
  \eqref{secondary}. More precisely, 
expanding any function $v\in L^2_\rho$ 
on the eigenfunctions of $\q L$ given in \eqref{eigenL}:
\begin{equation}\label{decomp}
v = \sum_{i=0}^7\sum_{j=0}^iv_{i,j}h_{i-j}h_j+v_-,
\end{equation}
where
\begin{equation}\label{defvij}
v_{i,j} = \int v(y) k_{i-j}(y_1) k_j(y_2) \rho(y)dy
\end{equation}
is the coordinate of $v$ along the eigenfunction $h_{i-j}h_j$ (which
corresponds to the eigenvalue $\lambda =1-\frac i2$ by \eqref{Lhihj}), and
\begin{equation}\label{defkn}
k_n=h_n/\|h_n\|_{L^2_\rho}^2,
\end{equation}
our target in this paper is to construct a solution to
equation \eqref{eqq} defined on some interval $[s_0,\infty)$ where
$s_0=-\log T$ such that
\begin{equation}\label{third}
\forall s\ge s_0,\;\;q(s)\in  V_A(s),
 \end{equation}
where $V_A(s)$ is defined as follows:
\begin{defi}[A shrinking set to trap the solution]\label{defvas}
  For any $A>0$ and $s\ge 1$, 
  $V_A(s)$ is the set of all $v\in L^\infty$
  such that $\|v+\varphi(\cdot,s)\|_{L^\infty}\le 2\kappa$,
  $v_{i,j}=0$ if $i$ or $j$ is odd, 
  $|v_{i,j}|\le A e^{-2s}$ if $i\le 4$,
$|v_{6,0}|=|v_{6,6}|\le A s^2e^{-3s}$,
$|v_{6,2}|=|v_{6,4}|\le A se^{-2s}$ and
$\|v_-\|_{L^2_\rho} \le A^2s^2e^{-3s}$,
where $v$ is decomposed as in \eqref{decomp}, and the profile $\varphi$
defined in \eqref{goodprof} and Lemma \ref{lemphi}.
\end{defi}
\begin{nb}
The $L^\infty$ bound in this definition is crucial to control the nonlinear term (see
\eqref{apriori} and Section \ref{secnonl}).
\end{nb}
Clearly, we see that
\begin{equation}\label{smallVA}
  \mbox{if }v\in V_A(s),\mbox{ then }
  \|v\|_{L^2_\rho}\le CAs e^{-2s},
 \end{equation}
for $s$ large enough,
which shows that our goal in \eqref{third} implies indeed that $q(s)\to 0$ as
$s\to \infty$ in $L^2_\rho(\m R^2)$.
The next part of the paper is devoted to the rigorous proof
of our goal in \eqref{third}, and to the fact that it implies our main
result stated in Theorem \ref{th1}.
\section{Dynamics of the equation and general form of initial data} \label{secdynw0}

In this section, we study the dynamics of equation \eqref{eqq} and
suggest a general form of initial data, depending on some
parameters, which will be fine-tuned in some further step to provide a
suitable solution so that Theorem \ref{th1} holds. 

\medskip

We proceed in 3 steps, each presented in a subsection, starting by the
projections of equation \eqref{eqq} on the different components of the
decomposition \eqref{decomp}, then, we study the behavior of the flow
on the boundary of the shrinking set $V_A(s)$ defined in
\eqref{defvas}. Finally, we suggest a general form for initial
data. Note that technical details will be postponed to Section
\ref{sectech} below, hoping to make our exposition clearer.
\subsection{Dynamics  of equation \eqref{eqq} in the shrinking set $V_A(s)$} 

Aiming at proving
\eqref{third}, we assume $q(s)\in
V_A(s)$ for all $s\in [s_0,s_1]$ for some $s_1\ge s_0$ and derive
differential equations satisfied by $q_{i,j}$ and $q_-$ in this
section. In fact, we will first derive the size of the components of the remainder term
$R(y,s)$ \eqref{defR} in the following:
\begin{lem}[Expansion of $R(y,s)$ \eqref{defR}]\label{lemR}
  For all $r\ge 2$, it follows that
  \begin{align*}
R(y,s) =\frac {p e^{-2s}}\kappa &\left\{ 32 h_0h_0+32(h_2h_0+h_0h_2)
          +4(h_4h_0+h_0h_4) \right.\\
    &+\left. 32 h_2h_2+4(h_4h_2+h_2h_4) \right\}+O(e^{-3s})
  \end{align*}
  in $L^r_\rho(\m R^2)$, as $s\to \infty$.
\end{lem}
\begin{proof}
Se Section \ref{secR} below.
\end{proof}

Now, we project equation \eqref{eqq} in the following:
\begin{prop}[Dynamics of equation \eqref{eqq} in
  $V_A(s)$] \label{propdyn}
  For any $A\ge 1$, there exists $\sinit(A)\ge 1$ such that the
following holds.
  Assume that
\begin{equation}\label{qs0}
 q(s_0)\in L^\infty,\;\; \nabla q(s_0)\in L^\infty\mbox{ and }
\forall r\ge 2,\;\;
\|q(s_0)\|_{L^r_\rho}\le C(r) A s_0e^{-2s_0},
  \end{equation}
and  that $q(s)\in V_A(s)$ satisfies equation \eqref{eqq}, for all $s\in
    [s_0,s_1]$, for some  $s_1\ge s_0\ge \sinit(A)$.  Then, for all $s\in [s_0,s_1]$:\\
    (i) For all $i\in \m N$ and $0\le j \le i$, $|q_{i,j}'(s)-(1-\frac
    i2)q_{i,j}(s)|\le C_iAse^{-3s}+R_{i,j}(s)$ ,\\
    (ii) $\frac d{ds}\|q_-(s)\|_{L^2_\rho} \le - 3
    \|q_-(s)\|_{L^2_\rho} + CAse^{-3s}+\|R_-(s) \|_{L^2_\rho}$,\\
   where $q(\cdot,s)$ and $R(\cdot,s)$ are decomposed as in \eqref{decomp}.
  \end{prop}
  \begin{proof}
    The proof is straightforward, except for the control of the
    potential term $Vq$ and the superlinear term $B$ in equation \eqref{eqq}, which both need
    a regularizing delay estimate, already mentioned informally in
    \eqref{delay}. For that reason, we only state that delay estimate
    here, and postpone the proof to Section \ref{sectechdyn} below.
    \end{proof}
  \begin{nb}
    The fact that $q$ satisfies the PDE \eqref{eqq} on the interval $[s_0,s]$ and
    not just at a particular time $s$ is important for
      the delay estimate \eqref{delay}.
 \end{nb}
  As we have just mentioned, let us precisely state that delay
  estimate, which is crucial for the proof of Proposition \ref{propdyn}:
  \begin{prop}[A delay regularizing estimate for equation \eqref{eqq}]\label{propdelay}
    Under the hypothesis of Proposition \ref{propdyn} and for any
    $r\ge 2$ and $s\in [s_0,s_1]$, it holds that
    \[
\|q(s)\|_{L^r_\rho}\le C(r)A s e^{-2s}.
    \]
  \end{prop}
  \begin{proof}
See below in Section \ref{sectechdelay}.
    \end{proof}

\subsection{Behavior of the flow on the boundary of $V_A(s)$}
Recall that our aim is to suitably choose $q(s_0)$ so that $q(s)\in
V_A(s)$ for all $s\ge s_0$.\\
As in the previous section, we assume that $q(s) \in V_A(s)$ for all
$s\in [s_0,s_1]$ for some $s_1 \ge s_0$. This time, we will assume in
addition that at $s=s_1$,
one component of $q(s_1)$ (as defined in \eqref{decomp}) ``touches'' the
corresponding part of the boundary of its bound defined in Definition
\ref{defvas}. 
We will then derive the position of the flow of that
component with respect to the boundary, and see whether it is inward
(which leads to a contradiction) or outward. This way, only the
components with an outward flow may touch. Hopefully, those components
will be in a finite dimensional space, leading the way to the
application of a consequence of Brouwer's lemma linked to some fine-tuning in initial
data, in order to guarantee that the solution will stay in $V_A(s)$
for all $s\ge s_0$. More precisely, this is our statement:
\begin{prop}[Position of the flow on the boundary of
  $V_A(s)$]\label{propflow}
  There exists
$\AAzero\ge 1$
such that for any
$A\ge \AAzero$,
there exists
 $\Szero(A)\ge 1$ such that  if the hypotheses of Proposition \ref
 {propdyn} hold with $s_0\ge \Szero$, then:\\
(i) If $q_{i,j}(s_1) = \theta A e^{-2s_1}$ with $i\le 4$, $i$ and $j$
even, and $\theta=\pm 1$, then, $\theta q_{i,j}'(s_1) > \frac d{ds}
 A{e^{-2s}}_{|s=s_1}$.\\
(ii) If $q_{6,j}(s_1)=\theta A s^2 e^{-3s}$ with $j=0$ or $j=6$, and $\theta=\pm 1$, then, $\theta q_{i,j}'(s_1) > \frac d{ds} As^2{e^{-3s}}_{|s=s_1}$.\\
 (iii) If $q_{6,j(s_1)}=\theta A s e^{-2s}$ with $j=2$ or $j=4$, and $\theta=\pm 1$, then, $\theta q_{i,j}'(s_1) < \frac d{ds} As{e^{-2s}}_{|s=s_1}$.\\
(iv) If $\|q_-(s_1)\|_{L^2_\rho} = A^2s_1^2 e^{-3s_1}$, then
$\frac d{ds}\|q_-(s_1)\|_{L^2_\rho}< \frac d{ds} {A^2s^2 e^{-3s}}_{|s=s_1}$.
\end{prop}
\begin{nb}
The flow in items (i) and (ii) is outward, while in items (iii) and
(iv), it is inward.
\end{nb}

\begin{proof}
  See Section \ref{sectechflow} below.
\end{proof}
Following this statement, we immediately see that the constraints on
$q_{6,0}=q_{6,6}$ and 
$q_-$ in Definition \ref{defvas} of $V_A(s)$ can never achieve equality:
\begin{cor}[$q_{6,2}=q_{6,4}$ and $q_-$ never quit] \label{corq-}
Following Proposition \ref{propflow} and assuming that $s_1>s_0$, it
holds that $|q_{6,2}(s_1)|=|q_{6,4}(s_1)|< As_1e^{-2s_1}$ and $\|q_-(s_1)\|_{L^2_\rho} < A^2s_1^2 e^{-3s_1}$.
 \end{cor}
\begin{proof}
  We will only prove the estimate on $q_-$, since the other follows in
  the same way.
  Proceeding by contradiction, we assume that
  $\|q_-(s_1)\|_{L^2_\rho}\ge A^2s_1^2 e^{-3s_1}$. Since $q(s)\in
  V_A(s)$ for all $s\in[s_0,s_1]$ by hypothesis, it follows that
  \[
    \forall s\in (s_0,s_1],\;\; \|q_-(s)\|_{L^2_\rho} \le A^2s^2 e^{-3s}
\mbox{ and }
    \|q_-(s_1)\|_{L^2_\rho} = A^2s_1^2 e^{-3s_1}.
  \]
  In particular, this translates intro the following estimate between
  the derivatives of both curves:
  \[
  \frac d{ds} \|q_-(s_1)\|_{L^2_\rho} \ge \frac d{ds} A^2s^2 e^{-3s},
\]
which is a contradiction by item (iv) of Proposition \ref{propflow}.
\end{proof}

\subsection{Choice of initial data}
Following Corollary \ref{corq-}, we see that the control of $q(s)$ in
$V_A(s)$ reduces to the control of 5 components:
$q_{0,0}$, $q_{2,0}$, $q_{4,0}$, $q_{4,2}$ and $q_{6,0}$ 
(together with the control of $\|q(s)+\varphi(s)\|_{L^\infty}$,
which is a major novelty of our
paper). It happens that item (i) in Proposition \ref{propflow}
indicates that the flow of those components is ``transverse outgoing''
on the boundary of the constraint introduced in Definition
\ref{defvas}. As in our various papers where we construct solutions to
PDEs with a prescribed behavior (see for example \cite{MZdmj97}), the
control of this finite-dimensional part will be done through a
Brouwer-type lemma, involving some fine-tuning of parameters in
initial data. Namely, the idea would be to introduce the following
family of initial data (for $w_0(y,s_0)$ linked to $q(y,s_0)$ by
definition \eqref{defq}) depending on parameters
$d_{0,0}$, $d_{2,0}$, $d_{4,0}$, $d_{4,2}$ and $d_{6,0}$:
\[
w_0(y,s_0) = \varphi(y,s_0)+\left[Ae^{-2s_0} S(y)+As^2e^{-3s} \bar S(y)\right]\chi(y)
\]
where 
\begin{align}
S(y)=& d_{0,0}+d_{2,0}[h_2(y_1)+h_2(y_2)] +
  d_{4,0}[h_4(y_1)+h_4(y_2)] + d_{4,2}h_2(y_1)h_2(y_2), \nonumber\\
  \bar S(y) =&d_{6,0}(h_6(y_1)+h_6(y_2)),\label{defS}
\end{align}
for some sufficiently decaying function $\chi$. In fact, in order to
stick to the shape of the profile $\varphi$ \eqref{goodprof}, we will
in fact take the following initial data for equation \eqref{eqw}:
\begin{equation}\label{defw0}
w_0(y,s_0) = \left[ \frac ED + \frac{p-1}{\kappa D^2}\left(Ae^{-2s_0}S(y) +As^2e^{-3s} \bar S(y) \right)\right]^{\frac 1{p-1}},
\end{equation}
where $E$, $D$, $S$ and $\bar S$ are introduced in \eqref{defE},
\eqref{defD} and \eqref{defS},
which yields the following initial data for equation
\eqref{eqq}:
\begin{equation}\label{q0}
q(y,s_0)=w_0(y,s_0)-\varphi(y,s_0).
\end{equation}
n the following, we exhibit a set for the parameters so that $q(s_0)$
is well-defined, $q(s_0)\in V_A(s_0)$ with other smallness and decay
properties, inherited from the profile $\varphi(y,s_0)$ given in
\eqref{goodprof}.
More precisely, this is our statement:
\begin{prop}[Initialization]\label{propinit} For any $A\ge 1$, there
  exists
$\sszeroun(A)\ge 1$
 such that for all
$s_0\ge \sszeroun(A)$,
there exists a set
$\q D(A,s_0)\subset[-2,2]^5$ such that for
  all parameter $d\equiv (d_{0,0}, d_{2,0}, d_{4,0}, d_{4,2}, d_{6,0})\in \q
  D$, we have the following 2 properties:\\
 (i) $E+\frac {p-1}{\kappa D}\left(Ae^{-2s_0}S(y) +As^2e^{-3s} \bar S(y)\right)\ge   \frac 14$,
  hence $w_0(y,s_0)$ \eqref{defw0}
  and $q(y,s_0)$ \eqref{q0} are well-defined.\\
  (ii) $q(s_0)\in  V_A(s_0)$ introduced in Definition \ref{defvas},
  estimate \eqref{qs0} holds true,
  $|q_{6,2}(s_0)|=|q_{6,4}(s_0)|\le  CAe^{-3s_0}<A s_0e^{-2s_0}$, $\|q_-(s_0)\|_{L^2_\rho}\le CAe^{-3s_0}<A^2 s_0^2e^{-3s_0}$,
  $\|w_0(s_0)\|_{L^\infty}\le \kappa+ Ce^{-\frac {s_0}3}$ and
  $\|\nabla w_0(s_0)\|_{L^\infty}\le Ce^{-\frac {s_0}6}$. Moreover,
   the function
  \begin{equation}\label{121}
  \begin{array}{rcl}
    \q D &\longrightarrow &[-A e^{-2s_0}, A e^{-2s_0}]^4\times [-As_0^2 e^{-3s_0}, As_0^2 e^{-3s_0}],\\
   d&\longmapsto& (q_{0,0}(s_0),
  q_{2,0}(s_0), q_{4,0}(s_0), q_{4,2}(s_0),q_{6,0}(s_0))
  \end{array}
  \end{equation}
is one-to-one.
  \end{prop}
  \begin{proof}
See Section \ref{sectechinit}.
    \end{proof}

\section{Control of $w_a(y,s)$ for any $a$}\label{seccontrolwa}
Starting from this section, our goal is to show that the $L^\infty$ bound on $q(y,s)+\varphi(y,s)$ in 
    Definition \ref{defvas} of $V_A(s)$ never enjoys an equality
    case.
    In other words, if $q(s_0)$ is given by \eqref{q0} for some
    parameter  $(d_{0,0},  d_{2,0}, d_{4,0}, d_{4,2}, d_{6,0})\in \q D$ defined
    in Proposition \ref{propinit} and $q(y,s)$ satisfies equation
    \eqref{eqq} with $q(s)\in V_A(s)$ given in
    Definition \ref{defvas}, for all $s\in [s_0,s_1]$, for some $s_1\ge
    s_0$, with $s_0$ large enough, then
    \begin{equation}\label{g0}
      \forall s\in [s_0,s_1],\;\;\forall y\in \m R^2,\;\;
      |q(y,s)+\varphi(y,s)|=|w_0(y,s)|<2 \kappa,
    \end{equation}
   where we have used the definition
\eqref{defq} of $q(y,s)$. Using the
relation \eqref{wawb}, this is equivalent to showing that
\begin{equation}\label{g1}
  \forall a\in \m R^2,\;\; \forall s\in [s_0,s_1],\;\; |w_a(0,s)|< 2\kappa.
 \end{equation}
 We will proceed in several steps in order to control $w_a(y,s)$ for
 any $a\in \m R^2$: we first give a uniform control of
 the gradient, then we explain our strategy, depending on the region
 where $a$ belongs to. The next three subsections are dedicated to the proof
 of the estimate in each region. Finally, we give in the last subsection a concluding
 statement for the whole section.
 
\subsection{Control of the gradient}
In this section,
we use the Liouville theorem we proved in \cite{MZcpam98}
and \cite{MZma00} for equation \eqref{equ} in order to show that
$\|\nabla q(s)\|_{L^\infty}$ is small, provided that $s_0$ is large enough.
Let us first recall our version of the Liouville theorem, stated for
equation \eqref{eqw}:
\begin{prop}[A Liouville theorem for equation \eqref{eqw}]\label{propliou}
Under condition \eqref{condp}, consider $W(y,s)$ a solution of
equation \eqref{eqw} defined and uniformly bounded for all $(y,s)\in
\m R^N\times (-\infty,\bar s)$ for some $\bar s\le +\infty$. Then, either $W\equiv 0$, or $W\equiv \pm \kappa$
or $W(y,s)=\pm \kappa(1\pm e^{s-s^*})^{-\frac 1{p-1}}$ for all $(y,s)\in
\m R^N\times (-\infty, \bar s]$ and for some $s^*\in \m R$. In all cases, it
holds that $\nabla W \equiv 0$.
\end{prop}
\begin{nb}
If $\bar s=+\infty$, then the unbounded solution $W(y,s)=\pm
\kappa(1-e^{s-s^*})^{-\frac 1{p-1}}$ never occurs. If $\bar s
<+\infty$, then that solution may occur with some $s^*$ satisfying $
e^{\bar s-s^*}<1$.
\end{nb}
\begin{proof}
If $\bar s=+\infty$, see Theorem 1.4 page 143 in \cite{MZcpam98} for the
nonnegative case and Theorem 1 page 106 in \cite{MZma00} for the
unsigned case.\\
If $\bar s<+\infty$, then the statement follows from a small
adaptation of the previous case. See Corollary 1.5 page 144 in
\cite{MZcpam98} where a similar adaptation is carried out.
\end{proof}
Let us now state our result for the gradient:
\begin{prop}[Smallness of the gradient]\label{propgrad}
  For all $A\ge 1$ and $\delta_0>0$, there exists
  $\sszerodeux(A,\delta_0)\ge 1$ 
  such that for all
 $s_0\ge \sszerodeux(A,\delta_0)$,
  if $q(s_0)$ is given by \eqref{q0} for some $(d_{0,0},
  d_{2,0}, d_{4,0}, d_{4,2},d_{6,0})\in \q D$ defined in Proposition
  \ref{propinit}, and $q(s)\in
  V_A(s)$ for all $s\in [s_0,s_1]$ for some $s_1\ge s_0$
  satisfies equation \eqref{eqq}, then, for all $s\in [s_0,s_1]$,
  $\|\nabla q(s)+\nabla \varphi(s)\|_{L^\infty}\le \delta_0$.
  \end{prop}
\begin{proof}
See Section \ref{secgrad}.
\end{proof}

\subsection{Strategy for the control of $w_a(y,s)$}

With this estimate, we can make a reduction of our goal \eqref{g1} in
the following:
\begin{cl}[Reduction] \label{clreduc} Under the hypotheses of Proposition
  \ref{propgrad}, assuming that
  \begin{equation}\label{condgrad}
    \delta_0 \le \frac \kappa 4 
    \mbox{ and }
 s_0\ge \sszerodeux(A,\delta_0),
\end{equation}
 we see that estimate
  \eqref{g1} follows from the following:
\begin{equation}\label{g2}
    \forall a\in \m R^2,\;\;\forall s\in [s_0,s_1],\;\;
    \|w_a(s)\|_{L^2_\rho} \le \frac 32 \kappa.
  \end{equation}
\end{cl}
\begin{proof}
Under the hypotheses of Proposition \ref{propgrad}, assume that \eqref{g2} holds. Noting that $\|\nabla
  w_a(s)\|_{L^\infty} \le \delta_0$ from Proposition \ref{propgrad} together with
  the relations \eqref{defq} and \eqref{wawb}, we use a Taylor
  expansion to write
\begin{equation}\label{lipwa}
  |w_a(y,s) - w_a(0,s)| \le |y| \cdot \|\nabla w_a(s)\|_{L^\infty}
  \le  {\delta_0} |y|.
\end{equation}
Therefore, since
\begin{equation}\label{calculrho}
\int\rho(y)dy=1\mbox{ and }\int|y|\rho(y) dy=1
\end{equation}
by
definition \eqref{defro} of $\rho$, it follows that
\[
  |w_a(0,s)| \le \int|w_a(y,s)|\rho(y)dy +\delta_0
  \le \|w_a(s)\|_{L^2_\rho}+ \delta_0
  \le \frac 32 \kappa +\frac \kappa 4 < 2\kappa,
\]
and the claim follows.
\end{proof}
In the following subsections, assuming that \eqref{condgrad} holds, we
will prove either \eqref{g0}, \eqref{g1} or \eqref{g2}, according to
the context. Using the sharper gradient estimate at initial time $s_0$
given in item (ii) of Proposition \ref{propinit}, let us remark that
at $s=s_0$, $w_a(y,s_0)$ is ``flat'' in $L^2_\rho$, in the sense that
it is close to some constant independent from space, as we prove in
the following: 
\begin{lem}[Flatness of $w_a(y,s_0)$] \label{lemflat}
  For any $A\ge 1$,
  $s_0\ge \sszeroun(A)$
and parameter
  $d=(d_{0,0}, d_{2,0},  d_{4,0}, d_{4,2}, d_{6,0})\in \q D(A,s_0)$, for any
  $a\in \m R^2$,
  \begin{equation}\label{estflat}
\|w_a(\cdot,s_0)-w_0(ae^{\frac{s_0}2},s_0)\|_{L^2_\rho}\le Ce^{-\frac{s_0}6},
  \end{equation}
  where
$\sszeroun(A)$
   and $\q D(A,s_0)$ are defined in Proposition
  \ref{propinit}. 
\end{lem}  
\begin{proof}
The proof follows by the same argument as  Claim \ref{clreduc}, in
particular, estimate \eqref{lipwa}.
\end{proof}
Following Lemma \ref{lemflat}, let us remark that equation \eqref{eqw}
satisfied by $w_a(y,s)$ has 3 bounded and nonnegative explicit ``flat'' solutions:
$0$, $\kappa$, and
\begin{equation}\label{defpsi}
\psi(s) = \kappa(1+e^s)^{-\frac 1{p-1}}
\end{equation}
(note that $\psi$ is a heteroclinic orbit connecting $\kappa$ to $0$,
and that all its time shifts are also solutions). In fact, 
it happens that $w_a(0,s_0)=w_0(ae^{\frac {s_0}2},s_0)$ belongs to the
interval $[0, \kappa+Ce^{-\frac{s_0}3}]$, from \eqref{wawb} and Proposition
\ref{propinit}. In other words, from Lemma \ref{lemflat}, we are in
the vicinity of one of those 3 explicit solutions, whose stability
properties are known! Indeed, if $0$ and $\psi$ are stable, as we will show
below (see Propositions \ref{proptrap} and \ref{proptrap2}), $\kappa$ has both
stable and unstable directions, as one may see from the linearization
of equation \eqref{eqw} around $\kappa$ (see below in \eqref{eqva}),
where the linearized operator appears to be $\q L$ introduced in
\eqref{defR}, whose spectrum is given in \eqref{specL}. In particular,
the existence of the heteroclinic orbit $\psi$ \eqref{defpsi} shows
the instability of $\kappa$. In other
words, if $w_0(ae^{\frac{s_0}2},s_0)$ is close to $\kappa$, we need to
refine estimate \eqref{estflat} in order to better understand the
dynamics of $w_a(y,s)$ for $s\ge s_0$.

\medskip

Accordingly, we will decompose the space into 3 regions, where
$w_a(y,s_0)$ will be in the vicinity of one of the 3 above-mentioned
explicit solutions, leading to 3 different scenarios for the behavior
of $w_a(y,s)$ for $s\ge s_0$. More precisely, given $m<M$, we
introduce 3 regions $\q R_i(m,M,s_0)$ for $i=1,2,3$ as follows:
\begin{align}
            \q R_1 &= \{a\in \m R^2\;|\;  Me^{-s_0}\le G_0(a)\},\nonumber\\
\q R_2 &= \{a\in \m R^2\;|\; me^{-s_0}\le G_0(a)\le Me^{-s_0}\},\nonumber\\
  \q R_3&= \{a\in \m R^2\;|\; G_0(a)\le me^{-s_0}\},
           \label{defqR}
\end{align}
where
\begin{equation}\label{defG0}
  G_0(a)
   = \frac{p-1}\kappa\left[a_1^2a_2^2+\delta(a_1^6+a_2^6)\right].
\end{equation}
In the following lemma, we will see that $G_0(a)$ is a kind of norm
which measures the size of $w_a(0,s_0)=w_0(ae^{-\frac{s_0}2},s_0)$:
\begin{lem}[Size of initial data in the three regions]\label{lemregions}
  For any $M\ge 1$, there exists $\Cdix(M)>0$ such that for any $A\ge
   1$, there exists $\sdix(A,M)$ such that for any $s_0\ge \sdix(A,M)$
   and $d\in \q D(A,s_0)$ defined in Proposition \ref{propinit}, for
   any $m\in (0,1)$, the following holds:\\
- If $a\in \q R_1$, then $0\le w_0(ae^{\frac{s_0}2},s_0)\le \kappa
  (1+M)^{-\frac 1{p-1}}+ \Cdix(M) e^{-\frac{s_0}3}$.\\
- If $a\in \q R_2$, then
  $\kappa (1+M)^{-\frac 1{p-1}} -  \Cdix(M)e^{-\frac{s_0}3}\le
  w_0(ae^{\frac{s_0}2},s_0)\le \kappa (1+m)^{-\frac 1{p-1}}+  C e^{-\frac{s_0}3}$.\\
- If $a\in \q R_3$, then
$\kappa (1+m)^{-\frac 1{p-1}} -  C e^{-\frac{s_0}3} \le w_0(ae^{\frac{s_0}2},s_0)\le \kappa +Ce^{-\frac{s_0}3}$.
\end{lem}
\begin{proof}
 See Section \ref{secregions}.
\end{proof}
Following the 2 previous lemmas and what we have mentioned concerning
the stability of the 3 explicit solutions of equation \eqref{eqw}
mentioned in \eqref{defpsi} and the line before, 3 scenarios become clear
for the proof of estimate \eqref{g2}:\\
- \textbf{Scenario 1}: if $a\in \q R_1$, provided that $M$ is large
enough, we are in the vicinity of the zero solution. Thanks to its
stability, $w_a(s)$ will remain small and \eqref{g2} will follow.\\
- \textbf{Scenario 2}: if $a\in \q R_2$, we are in the vicinity of
$\psi(s)$ \eqref{defpsi}. Thanks to its stability, $w_a(s)$ will
remain close to $\psi(s)$.
Since $\psi(s) \le
\kappa$, estimate \eqref{g2} holds.\\
- \textbf{Scenario 3}: if $a\in \q R_3$, we are in the vicinity of
$\kappa$, which has both stable and unstable directions as mentioned earlier. For that
reason, the bounds on $w_a(0,s_0)=w_0(ae^{\frac{s_0}2},s_0)$ given in Lemma
\ref{lemregions} are not enough, and
we need a more refined expansion of $w_a(y,s_0)$, followed by an
integration of PDE \eqref{eqw} satisfied by $w_a$. This step is the key
point of our argument. It is inspired by our techniques in
\cite{MZimrn21} for the control of $w_a$ where $a$ is a blow-up point
located near some given non-isolated blow-up point. In fact, from our careful
design of initial data in \eqref{defw0}, the integration of the PDE
will show that for $a$ not ``very small'' (in a sense that will naturally
appear in the proof), $w_a(s)$ will be attracted to
the vicinity of the heteroclinic orbit $\psi$ \eqref{defpsi}, leading 
us to \textit{Scenario 2}, where the stability of $\psi$ will imply
estimate \eqref{g2}. If $a$ is ``very small'', then, thanks to the gradient estimate of Proposition \ref{propgrad}, the boundedness of
$\|w_a(s)\|_{L^2_\rho}$ will follow from the boundedness of
$\|w_0(s)\|_{L^2_\rho}$, itself a consequence of the fact that
$q(s)\in V_A(s)$ given in Definition \ref{defvas} (see \eqref{smallVA}).

\medskip

In the following, we give details for those 3 scenarios in 3 subsections.

\subsection{Control of $w_a(y,s)$ in Region $\q R_1$}

The stability of the zero solution for equation \eqref{eqw} (under some
$L^\infty$ a priori bound) is crucial for the argument, as we wrote
above in \textit{Scenario 1}. Let us state it in the following:
\begin{prop}[Stability of the zero solution for equation \eqref{eqw}
  under an $L^\infty$ a priori bound]
  \label{proptrap}
  There exists $\epsilon_0>0$ and $M_0\ge 1$ such that
  if $w$ solves equation \eqref{eqw} with $|w(y,s)|\le 2 \kappa$ for all $(y,s)\in \m R^2\times
  [\suno,\sdue]$ for some $\sdue\ge \suno$, with
  $\|w(\suno)\|_{L^2_\rho}\le \epsilon_0$ and
  $\nabla w(0)(1+|y|)^{-k}\in L^\infty$ for some $k\in \m N$, then
\[
  \forall s\in [\suno,\sdue],\;\; \|w(s)\|_{L^2_\rho}
  \le M_0 \|w(\suno)\|_{L^2_\rho}
e^{-\frac{s}{p-1}}.
\]
  \end{prop}
 \begin{proof}
The proof is somehow classical, apart from a delay estimate
to control the nonlinear estimate we have already presented in Section
\ref{secnonl}. In order to focus only on the main arguments, we
postpone the proof to Section \ref{sectrap} in the appendix.
\end{proof}
Fixing $M\ge 1$ such that
\begin{equation}\label{defM}
\kappa(1+M)^{-\frac 1{p-1}}\le \max\left(\frac{\epsilon_0}2, \frac \kappa{2M_0}\right)
\end{equation}
where $M_0$ and $\epsilon_0$ are given in Proposition \ref{proptrap},
then taking $s_0$ large enough, we see from Proposition \ref{propinit}
that $\nabla w(s_0)\in L^\infty$, and from Lemmas \ref{lemregions}
and \ref{lemflat} that the smallness condition required in Proposition
\ref{proptrap} holds in Region $\q R_1$, leading to the trapping of
$w_a$ near $0$, proving the bound \eqref{g2}. More precisely, this is
our statement:
\begin{cor}[Exponential decay of $w_a(s)$ in Region $\q R_1$]\label{corr1}
  For all $A\ge 1$, there exists
  $\ssquatre(A)\ge 1$
such that if
 $s_0 \ge \ssquatre(A)$,
  $d\in \q D(A,s_0)$ and $a\in \q R_1$ defined in \eqref{defqR},
 then $\nabla w(s_0)\in L^\infty$ and 
 $\|w_a(s_0)\|_{L^2_\rho}\le 2\kappa(1+M)^{-\frac 1{p-1}}\le \epsilon_0$ introduced in Proposition
\ref{proptrap}.
If in addition we have $|w_a(y,s)|\le 2\kappa$, for all $(y,s)\in \m R^2 \times[s_0,s_2]$,
for some $s_2\ge s_0$,
then,
\[
  \forall s\in [s_0,s_2],\;\; \|w_a(s)\|_{L^2_\rho}
  \le M_0 \|w_a(s_0)\|_{L^2_\rho}e^{-\frac{s-s_0}{p-1}}
  \le 2\kappa(1+M)^{-\frac 1{p-1}} M_0 e^{-\frac{s-s_0}{p-1}}\le \kappa,
\]
where $M_0$ is also introduced in Proposition \ref{proptrap}. In
particular, \eqref{g2} holds.
\end{cor}

\subsection{Control of $w_a(y,s)$ in Region $\q R_2$}
As we explained above in \textit{Scenario 2}, the stability of the
solution $\psi$ \eqref{defpsi} is the key argument. Let us first state
that stability result:
\begin{prop}[Stability of the heteroclinic orbit for equation \eqref{eqw}
  under an $L^\infty$ a priori bound] \label{proptrap2}
   There exists $M_1\ge 1$ such that
 if $w$ solves equation \eqref{eqw} with $|w(y,s)|\le 2 \kappa$
  for all $(y,s)\in \m R^2\times  [0,\sdue]$ for some $\sdue\ge 0$,
  and
  \begin{equation}\label{cond0}
    \nabla w(\suno)(1+|y|)^{-k}\in L^\infty,\;\;
  \|w(\suno)-\psi(\sunos)\|_{L^2_\rho}\le 
\frac {|{\psi'}(\sunos)|}{M_1}
 \end{equation}
for some $k\in \m N$ and $\sunos\in \m R$,
where $\psi$ is defined in \eqref{defpsi}, then
\begin{equation}\label{samir}
  \forall s\in [\suno,\sdue],\;\; \|w(s)-\psi(s+\sunos)\|_{L^2_\rho}
  \le M_1 \|w(\suno) -\psi(\sunos)\|_{L^2_\rho}
  \frac{|{\psi'}(s+\sunos)|}{|{\psi'}(\sunos)|}.
 \end{equation}
\end{prop}
  \begin{nb}
    By definition \eqref{defpsi} of $\psi$, we have exponential decay
    in \eqref{samir}. Moreover,
    since $\kappa$ and $\psi(s)$ \eqref{defpsi} are both solutions to
    equation \eqref{eqw}, $\kappa$ should never satisfy condition
    \eqref{cond0}. This is clear, except when $\sunos\to -\infty$,
    since in that case $\psi(\sunos) \to \kappa$. More precisely,
    we see from \eqref{defpsi} that
    $\kappa-\psi(\sunos)\sim \frac{\kappa e^{\sunos}}{p-1}\sim
    |{\psi'}(\sunos)|$, which shows that \eqref{cond0} is indeed
    sharp, up to a multiplying (small) factor, $1/M_1$.
\end{nb}
\begin{proof}
The proof is much more involved than the proof of Proposition
\ref{proptrap},  since by definition \eqref{defpsi}, $\psi(\sunos)$ may be close to $\kappa$, the
unstable equilibrium of equation \eqref{eqw}. As for the previous
proposition, we postpone the proof to Section \ref{sectrap} in the
appendix. 
\end{proof}
In the following corollary, using Proposition \ref{propinit}, Lemmas \ref{lemflat} and \ref{lemregions}, we show that
$w_a(s)$ is trapped near the heteroclinic orbit $\psi$ \eqref{defpsi}
whenever $a$ is in Region $\q R_2$ \eqref{defqR}. More precisely, this
is our statement:
\begin{cor}[Trapping of $w_a(s)$ near $\psi$ in Region $\q R_2$]\label{corr2}
  There exists $\bar \sigma \in \m R$ such that for all $m\in (0,1)$,
  there exists $\barbelow \sigma(m)\le \bar \sigma$ such that for all
  $A\ge 1$,  there exists $\scinco(A,m)\ge 1$
such that for all $s_0 \ge \scinco(A,m)$, $d\in \q D(A,s_0)$ and $a\in
\q R_2$ defined in \eqref{defqR} (with $M$ defined in \eqref{defM}),
$\nabla w_a(s_0)\in L^\infty$ and 
$\|w_a(s_0)-\psi(\sunos)\|_{L^2_\rho}\le Ce^{-\frac{s_0}6}
\le \frac {|{\psi'}(\sunos)|}{M_1}$ for some 
$\sunos \in [\barbelow\sigma(m), \bar \sigma]$, where $M_1$ was introduced
in Proposition \ref{proptrap2}. 
If in addition we have $|w_a(y,s)|\le 2\kappa$, for all $(y,s)\in \m R^2 \times[s_0,s_2]$,
for some $s_2\ge s_0$,
then,
\[
  \forall s\in [s_0,s_2],\;\;
  \|w_a(s)\|_{L^2_\rho} \le \psi(s+\sunos-s_0)
  + \frac{CM_1}{C(m)}e^{-\frac{s_0}6} \le \kappa + \frac \kappa 4\le
  \frac 32\kappa.
 \]
In particular, \eqref{g2} holds.
\end{cor}
\begin{proof}
The proof is omitted, since it is a direct consequence of Proposition
\ref{proptrap2}, thanks to Proposition \ref{propinit}, Lemmas \ref{lemflat} and \ref{lemregions}.
  \end{proof}

\subsection{Control of $w_a(y,s)$ in Region $\q R_3$}\label{secR3}

Given an initial time $s_0$, some $m\in(0,1)$ and $a\in \q R_3$
defined in \eqref{defqR}, we write $a$ as follows:
\begin{equation}\label{defa}
a=(Ke^{-\frac {s_0} 2}, Le^{-\frac {s_0} 2}),
\end{equation}
for some real numbers $K$ and $L$. Since $w_0(\cdot,s_0)$ is
symmetric with respect to the axes and the bissectrices (see \eqref{defw0}), so is
$w_0(\cdot,s)$ for any later time $s$. For that reason, we only
consider the case where
\begin{equation}\label{condkl}
  0\le K\le L.
\end{equation}
Since
\begin{equation}\label{waw0}
  w_a(y,s_0) = w_0(y+ae^{\frac {s_0} 2},s_0) = w_0(y_1+K, y_2+L,s_0)
  \end{equation}
from \eqref{wawb}, we will use the explicit expression \eqref{defw0}
of $w_0$  in order to estimate the components of $w_a(s_0)$ in
$L^2_\rho$. Taking this as initial data, we will integrate equation
\eqref{eqw} (which is satisfied by $w_a$) in order to estimate
$w_a(s)$ for later times $s\ge s_0$.

\medskip

It happens that the outcome of the integration depends on the size of 
  $a$, which can be measured in
terms of the position of $K+L$ and $A$. For that reason, we
distinguish two cases in the following.

\subsubsection{Case where $K+L\ge A$}\label{seclargekl}
From the decomposition \eqref{defa} of $a=(a_1,a_2)$, this is the case of
``large'' $a$, where $a_1+a_2 \ge Ae^{-\frac{s_0}2}$. 
Let us first estimate $w_a(s_0)$.
In order to be consistent with the definition
\eqref{goodprof} of our profile $\varphi$ and the decomposition in
regions we suggest in \eqref{defqR}, we will give the expansion of
$w_a(s_0)$ in $L^2_\rho$, uniformly with respect to the small variable 
\begin{equation}\label{defiota}
\iota = e^{-s_0}K^2L^2+ \delta e^{-2s_0}[K^6+L^6]
\end{equation}
and the large parameter $A$.
This is our statement:
\begin{lem}[Initial value of $w_a(y,s)$ for large $a$]
  \label{propdeco0}
  For any $A\ge 1$, there exists $\squatre(A)\ge 1$ such that for any
  $s_0\ge \squatre(A)$ and any parameter
  $(d_{0,0}, d_{2,0}, d_{4,0}, d_{4,2}, d_{6,0})\in \q D$ defined in
  Proposition \ref{propinit},  if  $w_0(y,s_0)$ is given by
  \eqref{defw0}, then,  for any $m\in(0,1)$ and $a\in \q R_3$ defined
  in \eqref{defqR} with $a$ decomposed as in \eqref{defa} for some
  $L\ge K\ge 0$, with $K+L\ge A$,  
   the following expansion holds in
 $L^r_\rho(\m R^2)$ for any $r\ge 2$:
  \begin{align*}
    w_a(y,s_0)=\kappa-\iota
     -e^{-s_0}&\left\{2KL^2h_1h_0+2K^2Lh_1h_0+L^2h_2h_0+4KLh_1h_1
      +K^2h_0h_2\right.\\
    &+\left.2Lh_2h_1+2Kh_1h_2+h_2h_2\right\}
        +O\left(\frac \iota A\right)+O(\iota^2),
  \end{align*}
  where $\iota$ is defined in \eqref{defiota}.
 \end{lem}
\begin{proof}
See Section \ref{secwas}.
\end{proof}
 Noting that
\begin{equation}\label{boundiota}
  \iota \le \frac{m\kappa}{p-1} \le \frac \kappa{p-1}
\end{equation}
  from \eqref{defiota} and \eqref{defqR}, it follows from this lemma that $\|w_a(s_0)-\kappa\|_{L^2_\rho} \le C\iota+CJ$ where
\begin{equation}\label{boundJ}
  J\equiv e^{-s_0}(K^2+L^2)\le 2 e^{-s_0}(K^6+L^6)^{\frac 13}
  \le  2 e^{-s_0} \left(\frac{e^{2s_0}\iota}\delta\right)^{\frac 13}
  = 2 e^{-\frac{s_0}3} \left(\frac \iota\delta\right)^{\frac 13}
  \le 2 \iota^{\frac 13}
\end{equation}
from \eqref{defiota} and the fact that $s_0\ge 0$ and $\delta \ge 1$
(see the beginning of Section \ref{secformu}). Using again
\eqref{boundiota},  it follows that
$w_a-\kappa$ is small at $s=s_0$, whenever $m$ is small.
It is natural then to make a linear approximation for equation \eqref{eqw}
around $\kappa$ in order to obtain the expansion of $w_a(y,s)$ for later times, as
long as $w_a(s)-\kappa$ remains small. In fact, we will see that the
projection on $h_0h_0=1$ will dominate in $w_a(y,s)$. For that reason,
given some small $\eta^*$ such that 
\begin{equation}\label{condeta*}
\eta^* \ge \frac{m\kappa}{p-1},
\end{equation}
where $m\in (0,1)$ is the constant appearing in the definition
\eqref{defqR} of $\q R_3$, our
integration will be valid only on the interval $[s_0,s^*]$ where
$s^*=s^*(s_0,a,\eta^*)$ is defined by 
\begin{equation}\label{defs*}
  e^{s^*-s_0} \iota = \eta^*
 \end{equation}
and $\iota$ is given in \eqref{defiota}. Note that $s^*$ is well
defined since $L+K\ge A>0$, hence $\iota>0$. Note also that  $s^*\ge
s_0$ thanks to condition \eqref{condeta*}, as we will see in item (i) of
Lemma \ref{cordyn} below. More precisely, this is our statement:
\begin{lem}[Decreasing from $\kappa$ to $\kappa - \eta^*$ for ``large'' $a$
  under some a priori $L^\infty$ bound]\label{cordyn}
  There exists
$\bMsix>0$,
$\bAsix\ge 1$
and $\etazerozero>0$
such that for all
$A\ge \bAsix$,
there exists $\squatro(A)$ such that
for any $s_0\ge \squatro(A)$ and any parameter
$(d_{0,0}, d_{2,0}, d_{4,0}, d_{4,2}, d_{6,0})\in \q D$ defined in Proposition
\ref{propinit},  if $w_0(y,s_0)$ is given by \eqref{defw0}, then,
for any $\eta^*\in (0, \etazerozero]$
and $m\in (0,\min(1,\frac{\eta^*(p-1)}{\kappa}))$, for
any  $a\in \q R_3$ defined in \eqref{defqR} where $a$ is given by
\eqref{defa} for some $L\ge K\ge 0$, with $K+L\ge A$, the following
holds:\\ 
(i) $s^*\ge s_0$,  where $s^*$ is introduced in \eqref{defs*}.\\
  (ii) If we assume in addition that $\|w_a(s)\|_{L^\infty}\le 2
  \kappa$ for all $s\in [s_0,s_1]$ for some $s_1\ge s_0$, then,
  for all $s\in [s_0, \min(s^*,s_1)]$,
\begin{equation*}
\|w_a(\cdot,s) - \left(\kappa
 - e^{s-s_0}\iota\right)\|_{L^2_\rho}
\le
\bMsix
\left(\eta^*  +A^{-1}\right)e^{s-s_0}\iota
+\bMsix e^{-\frac{s_0}3}.
\end{equation*}
  \end{lem}
  \begin{proof}
See Section \ref{secwas}.
  \end{proof}
  With this result, we are able to prove estimate \eqref{g2}:
  \begin{cor}[Proof of the $L^2_\rho$ bound when $a\in \q R_3$ is
    ``large'' under some $L^\infty$ a priori bound]\label{corbR3} 
    There exists
$\bAun>0$,
$\etaun>0$
and
$\bsun(\eta^*)\ge 1$,
such that under the
hypotheses of Lemma \ref{cordyn}, if in addition
$A\ge \bAun$,
$\eta^*\le \etaun$
and
$s_0\ge \bsun(\eta^*)$,
then, for all
    $s\in[s_0,s_1]$, $\|w_a(s)\|_{L^2_\rho} \le \frac 32 \kappa$.
  \end{cor}
  \begin{proof}
    Following Lemma \ref{cordyn}, we further assume that
    $s_0\ge \sszerodeux(A,1)$
  defined in Proposition \ref{propgrad}, so we can apply that
    proposition. Consider then $s\in [s_0,s_1]$. We distinguish 2
    cases:\\ 
    \textit{Case 1: $s\le s^*$}.
   We write from Lemma \ref{cordyn} and
    \eqref{calculrho}, together with the definitions \eqref{defiota}
    and \eqref{defs*} of $\iota$ and $s^*$
     \begin{align*}
        \|w_a(s)\|_{L^2_\rho}
       &\;\le \kappa + e^{s^*-s_0}\iota
         +\bMsix(\eta^*+A^{-1})  e^{s^*-s_0}\iota
         +\bMsix e^{-\frac{s_0}3}\\
&\;         \le \kappa + \eta^*
         +\bMsix(\eta^*+A^{-1})  \eta^*+\bMsix e^{-\frac{s_0}3}.
     \end{align*}
Since $A\ge 1$, taking $\eta^*$ small enough and $s_0$ large enough, we get the result.\\
      \textit{Case 2: $s\ge s^*$}. In this case, we have $s_0\le
      s^*\le s \le s_1$. 
 We will show that starting at time $s^*$, $w$
      will be trapped near the heteroclinic orbit $\psi$
      \eqref{defpsi}, thanks to Proposition \ref{proptrap2}. In
      particular, its $L^2_\rho$ will remain bounded by $\frac 32 \kappa$.
       More precisely, from item (ii) of Lemma \ref{cordyn}
      and \eqref{defs*}, we see that 
        \[
      \|w_a(s^*) - (\kappa -\eta^*)\|_{L^2_\rho} \le
    \bMsix
    (\eta^*+A^{-1})  \eta^*
 +\bMsix e^{-\frac{s_0}3}.
    \]
Assuming that $\eta^*<\kappa$, we may introduce $\sigma^*\in \m R$ such that $\psi(\sigma^*) = \kappa-\eta^*$,
where $\psi$ is defined in \eqref{defpsi}. Noting that
\begin{equation}\label{dlpsi}
|{\psi'}(\sigma^*)|\sim \frac {\kappa e^{\sigma^*}}{p-1} \sim \kappa-\psi(\sigma^*)=\eta^*\mbox{ as } \eta^*\to 0, 
\end{equation}
we see that taking $\eta^*$
small enough, then $A$ and $s_0$ large enough, we have
\[
  \|w_a(s^*)-\psi(\sunos)\|_{L^2_\rho}
  \le \frac {|{\psi'}(\sunos)|}{M_1\left[1+\frac 2\kappa \|{\psi'}\|_{L^\infty}\right]}, 
 \]
where $M_1$ is introduced in Proposition \ref{proptrap2}. Since
$\nabla w_a(s^*)\in L^\infty$, thanks to Proposition \ref{propgrad},
together with the definition \eqref{defq} of $q$ and the
transformation \eqref{wawb}, Proposition \ref{proptrap2} applies and
we see that at time $s$, we have 
\begin{align*}
\|w_a(s)-\psi(s+\sunos-s^*)\|_{L^2_\rho}
 \le &\; M_1 \|w_a(s^*) -\psi(\sunos)\|_{L^2_\rho}
  \frac{|{\psi'}(s+\sunos)|}{|{\psi'}(\sunos)|}\\
  \le \frac {|{\psi'}(s+\sunos-s^*)|}{1+\frac 2\kappa \|{\psi'}\|_{L^\infty}} \le \frac \kappa 2.
\end{align*}
Since $\psi\le \kappa$ by definition \eqref{defpsi}, using
\eqref{calculrho} we see that
\[
  \|w_a(s) \|_{L^2_\rho}
  \le \psi(s+\sunos-s^*) + \frac \kappa 2 \le \frac 32 \kappa.
 \]
This concludes the proof of Corollary \ref{corbR3}.
  \end{proof}

\subsubsection{Case where $K+L\le A$}

Given $m\in(0,1)$ and $a\in \q R_3$ \eqref{defqR}, we aim in this
section to handle the case where $a$ is ``small'', namely when it
belongs to the triangle $\q T_0$ defined by 
\[
  \q T_0= \{(Ke^{-\frac{s_0}2},Le^{-\frac{s_0}2})\;|\;0\le K\le L\mbox{ and }K+L\le A\}.
\]
Let us recall from the beginning of Section \ref{seccontrolwa} that
$q(s)\in V_A(s)$ given in definition \ref{defvas}, for all $s\in
[s_0,s_1]$ for some $s_1\ge s_0$, hence,
\[
\forall s\in [s_0, s_1],\;\; \|w_0(s)\|_{L^\infty} \le 2 \kappa.
\]
Introducing the segment
\begin{equation}\label{defgs}
\q  G_\sigma=\{(K' e^{-\frac \sigma 2}, L' e^{-\frac \sigma 2})\;|\;
  \;0\le K'\le L'\mbox{ and }K'+L'= A\}
\end{equation}
and the triangle
\[
\q T_1= \{(K'e^{-\frac{s_1}2},L'e^{-\frac{s_1}2})\;|\;0\le K'\le L'\mbox{ and }K'+L'\le A\},
\]
we see that
\[
\q T_0 = \cup_{s_0\le \sigma < s_1} \q G_\sigma \cup \q T_1.
\]
We then proceed in 2 steps:\\
- In Step 1, we handle the case of ``small'' $s$, namely when
$a\in \q G_\sigma$ with $s_0\le s\le \sigma \le s_1$,
and also the case where  $a\in \q T_1$ with $s_0\le s\le s_1$.\\
- In Step 2, we handle the case of ``large'' $s$, namely when
$a\in \q G_\sigma$ with $s_0\le \sigma \le s\le s_1$.

\bigskip

\textbf{Step 1: Case of ``small'' $a$ and ``small'' $s$} \label{ps1}

Consider $a\in \q G_\sigma$ with $s_0\le s\le \sigma < s_1$, or
$a\in \q T_1$ with $s_0\le s\le s_1$. The conclusion will follow from
the $L^2_\rho$ estimate on $w_0$ together with the gradient estimate
of Proposition \ref{propgrad}. Indeed, using a Taylor expansion
together with \eqref{waw0} and \eqref{calculrho}, we write
\begin{align*}
  w_a(0,s)= w_0(a e^{\frac s2},s)
 & = w_0(0,s) +O\left(a e^{\frac s2} \|\nabla w_0(s)\|_{L^\infty}\right),\\
  \int w_0(z,s) \rho(z) dz &= w_0(0,s) +O\left(\|\nabla w_0(s)\|_{L^\infty}\right),
\end{align*}
on the one hand. On the other hand, since $w_0 = \varphi+q$ by definition \eqref{defq},
recalling that $q(s)\in V_A(s)$, we write by definitions
\eqref{goodprof} and \eqref{defvij} of $\varphi$ and $q_{0,0}(s)$,
together with Definition \ref{defvas} of $V_A(s)$,   
\begin{align*}
  \int w_0(z,s) \rho(z) dz
 & = \int \varphi (z,s) \rho(z) dz + \int q(z,s) \rho(z) dz
  = \int \varphi (z,s) \rho(z) dz + q_{0,0}(s)\\
 & = \kappa + O\left(e^{-s}\right) +O\left(Ae^{-2s}\right).
\end{align*}
Introducing $(K',L')$ such that
\begin{equation}\label{defk'l'}
a=(K' e^{-\frac \sigma 2}, L' e^{-\frac \sigma 2})
\end{equation}
with $\sigma =s_1$ if $a\in \q T_1$,
we see that
\[
0\le K'\le L'\mbox{ and }K'+L'\le A.
\]
Recalling that $s\le \sigma$, we write
\[
  |a e^{\frac s2}|=\sqrt{{K'}^2+{L'}^2} e^{\frac{s-\sigma}2}
  \le \sqrt 2 A.
\]
Taking $A\ge 1$ and recalling that $\|\nabla w_0(s)\|_{L^\infty}\le
\delta_0$ from Proposition \ref{propgrad} provided that
$s_0\ge \sszerodeux(A, \delta_0)$,
we write from the previous estimates that 
\[
  |w_a(0,s)|\le \kappa + C\delta_0 +CA \delta_0+ Ce^{-s}+CA e^{-2s}
  \le \frac 32 \kappa,
\]
whenever $\delta_0\le \donzeb(A)$ and $s_0 \ge \sonzeb(A, \delta_0)$ for some
$\sonzeb(A,\delta_0)\ge 1$ and $\donzeb(A)>0$. In particular, estimate
\eqref{g1} holds. 

\bigskip

\textbf{Step 2: Case of ``small'' $a$ and ``large'' $s$}

Now, we consider $a\in \q G_\sigma$ with $s_0\le \sigma \le s\le
s_1$. As we will shortly see, the conclusion follows here from the case
$K+L\ge A$ treated in Section \ref{seclargekl}, if one replaces there
$K$, $L$ and $s_0$ by $K'$, $L'$ and $\sigma$. Indeed, by definition
\eqref{defgs} of $\q G_\sigma$, we have
\begin{equation}\label{k'l'}
  K'+L'=A,
\end{equation}
where $K'$ and $L'$ are defined in \eqref{defk'l'}. Following our
strategy in Section \ref{seclargekl}, we first start by expanding
$w_a(y,\sigma)$ as in Lemma \ref{propdeco0}:
\begin{lem}[Expansion of $w_a(y,\sigma)$]\label{propdeco0'}
  For any $A\ge 1$, there exists
$\squatrep(A)\ge 1$
  such that
  for any
$s_0\ge \squatrep(A)$,
  the following holds:\\
  Assume that  $q(s)\in V_A(s)$ satisfies
  equation \eqref{eqq} for any $s\in [s_0, \sigma]$ for some $\sigma
  \ge s_0$, such that \eqref{qs0} holds and $\nabla q(s_0)\in
  L^\infty(\m R^2)$. Assume in addition that
  $a=(K',L')e^{-\frac \sigma 2}$ such that \eqref{k'l'} holds.
 Then,
   \begin{align}
    w_a(y,\sigma)=\kappa-\iota' & +O\left(\frac {\iota'} A\right)+O({\iota'}^2) \label{expwas}\\
     -e^{-\sigma}&\left\{2K'{L'}^2h_1h_0+2{K'}^2L'h_1h_0+{L'}^2h_2h_0+4K'L'h_1h_1
      +{K'}^2h_0h_2\right. \nonumber\\
                 &+\left.2L'h_2h_1+2K'h_1h_2+h_2h_2\right\}\nonumber\\
                 +q_{6,2}(\sigma)& \left\{
                   {L'}^4h_2h_0 +{K'}^4 h_0h_2
                   +4{L'}^3 h_2h_1 +4{K'}^3 h_1h_2
                   +6({K'}^2+{L'}^2)h_2h_2\right.\nonumber\\
                   &\left.+4K'h_3h_2+4L'h_2h_3
                   +h_4h_2+h_2h_4\right\}, \nonumber 
  \end{align}
in $L^r_\rho$ for any $r\ge 2$, with $|q_{6,2}(\sigma)|\le A \sigma
e^{-2\sigma}$, where 
  \begin{equation}\label{defiota'}
\iota' = e^{-\sigma}{K'}^2{L'}^2+ \delta e^{-2\sigma}[{K'}^6+{L'}^6].
\end{equation}
\end{lem}
\begin{proof}
See Section \ref{secwas}.
\end{proof}
Now, arguing as for Lemma \ref{cordyn}, we see $w_a(\sigma)$ as
initial data then integrate equation \eqref{eqw} to get an expansion of
$w_a(s)$ for later times:
\begin{lem}[Decreasing $w_a(\sigma)$ from $\kappa$ to $\kappa -
  \eta^*$] \label{cordyn'}
  There exists
$\bMsixp>0$, $\bAsixp\ge 1$ and $\etazerop>0$
such that for all
$A\ge \bAsixp$ and $\eta^*\in(0,\etazerop]$,
there exists
$\sonzep(A,\eta^*)$
such that
for any
$\sigma \ge \sonzep$
and $s_1\ge \sigma$, if $q(s)\in V_A(s)$ given
in Definition \ref{defvas} satisfies equation \eqref{eqq}
on $[\sigma,s_1]$ and $\nabla q(\sigma)\in L^\infty$, if $a=(K'e^{-\frac \sigma 2},
L' e^{-\frac \sigma 2})$ with  $K'+L'=A$, then:\\ 
(i) $s^*\ge \sigma$,  where $s^*$ is such that $e^{s^*-\sigma} \iota' =
\eta^*$ where $\iota'$ is defined in \eqref{defiota'}.\\
  (ii) For all $s\in [\sigma, \min(s^*,s_1)]$,
\begin{equation*}
\|w_a(s) - \left(\kappa
  - e^{s-\sigma}\iota'\right)\|_{L^2_\rho}
\le \bMsixp\left(\eta^*+A^{-1}\right)e^{s-\sigma}\iota'+\bMsixp
e^{-\frac\sigma 3}.
\end{equation*}
\end{lem}
\begin{proof}
The proof follows from a straightforward adaptation of the proof of
Lemma \ref{cordyn}, given in Section \ref{secwas}. For that reason,
the proof is omitted.
  \end{proof}
As in the case $K+L\ge A$, we derive from this result the following
corollary where we prove estimate \eqref{g2}:
  \begin{cor}[Proof of the $L^2_\rho$ bound when $a\in \q R_3$ is
    ``small'' and $s$ is ``large'']
   \label{corbR3'}
    There exists
    $\bAunp>0$, $\etaunp>0$ and $\bsunp(\eta^*)\ge 1$,
  such that under the hypotheses of Lemma \ref{cordyn'}, if in
    addition
    $A\ge \bAunp$,  $\eta^*\le \etaunp$ and $s_0\ge \bsunp(\eta^*)$,  
  then, for all
    $s\in[\sigma,s_1]$, $\|w_a(s)\|_{L^2_\rho} \le \frac 32 \kappa$.
  \end{cor}
  \begin{proof}
The proof is omitted since this statement follows from Lemma
\ref{cordyn'} exactly in the same way Corollary \ref{corbR3} follows
from Lemma \ref{cordyn}. 
\end{proof}
  
\subsubsection{Concluding statement for the control of $w_a(s)$ when
  $a\in \q R_3$}
Combining the previous statements given when $a\in  \q R_3$ defined in
\eqref{defqR}, we obtain the following statement:
\begin{lem}[Control of $w_a(s)$ when $a\in \q R_3$ if $w(s_0)$ is
  given by \eqref{defw0} and $q(s)\in V_A(s)$]
  \label{lemr3}

  There exist $\Avingt\ge 1$ and $\mvingt\in(0,1)$
  such that for all $A\ge \Avingt$, there exists $\deltavingt(A)>0$ such
  that for all $\delta_0\in(0,\deltavingt(A))$, there exists
  $\svingt(A,\delta_0) \ge 1$ such that for all $s_0\ge
  \svingt(A,\delta_0)$, $d\in \q D(A,s_0)$ defined in Proposition
  \ref{propinit} and $s_1\ge s_0$, the following holds:\\
Assume that $w$ is the solution of equation \eqref{eqw} with initial data
$w_0(y,s_0)$ defined in \eqref{defw0}, such that for all
$s\in[s_0,s_1]$, $q(s) \in V_A(s)$ given in  Definition \ref{defvas},
where $q(s)$ is defined in \eqref{defq}.
Then, for all $s\in [s_0,s_1]$:\\
(i) $\|\nabla w(s)\|_{L^\infty}\le \delta_0$.\\
(ii) For all $m\in (0, \mvingt]$ and $a\in \q R_3$ defined in
\eqref{defqR}, $|w_a(0,s)|<2\kappa$.
\end{lem}
\begin{proof}
The proof is omitted since it is straightforward from Claim
\ref{clreduc}, Corollary \ref{corbR3}, Step 1 given on page \pageref{ps1}
and Corollary \ref{corbR3'}. 
\end{proof}

\subsection{Concluding statement for the control of $w_a(s)$ for any
  $a\in \m R^2$}
Fixing $M$ as in \eqref{defM} and taking $m=\mvingt$ introduced in
Lemma \ref{lemr3}, we see that the 3 regions in \eqref{defR} are
properly defined. Then, combining the previous statements given for the
different regions (namely, Corollaries \ref{corr1} and \ref{corr2},
together with Lemma \ref{lemr3}), we derive the
following:
\begin{prop}[Control of $\|w(s)\|_{L^\infty}$ if $w(s_0)$ is given by
  \eqref{defw0} and $q(s)\in V_A(s)$] \label{propwa}
 There exist $\Avingtetun\ge 1$ 
  such that for all $A\ge \Avingtetun$, there exists
  $\svingtetun(A) \ge 1$ such that for all $s_0\ge
  \svingtetun(A)$, $d\in \q D(A,s_0)$ defined in Proposition
  \ref{propinit} and $s_1\ge s_0$, the following holds:\\
Assume that $w_0$ is the solution of equation \eqref{eqw} with initial data
$w_0(y,s_0)$ defined in \eqref{defw0}, such that for all
$s\in[s_0,s_1]$, $q(s) \in V_A(s)$ given in  Definition \ref{defvas},
where $q(s)$ is defined in \eqref{defq}.
Then:\\
(i) For all $s\in [s_0,s_1]$, $\|w_0(s)\|_{L^\infty}<2\kappa$.\\
(ii) For all $a\in \m R^2$ such that
\begin{equation}\label{alarge}
  |a_1|+|a_2|\ge A e^{-\frac{s_1}2},
  \end{equation}
there exist $\bar s(a)\ge s_0$ and $\bar
M(a)\ge 0$ such that if $\bar s(a)\ge s_1$, then for all $s\in [\bar s(a), s_1]$,
$\|w_a(s)\|_{L^2_\rho} \le \bar M(a) e^{-\frac s{p-1}}$.
\end{prop}
\begin{proof}$ $\\
(i) The proof is omitted since it is straightforward from the
above-mentioned statements.\\
(ii) Take  $a\in \m R^2$ such that \eqref{alarge} holds.
The conclusion follows according to the position
of $a$ in the 3 regions $\q R_i$ defined in in \eqref{defqR}, with the
constants $M$ and $m$ fixed in the beginning of the current subsection.\\
If $a \in \q R_1$, then the conclusion follows from Corollary
\ref{corr1}.\\
If $a\in \q R_2$, then we see from Corollary \ref{corr2} and its proof
that Proposition \ref{proptrap2} applies. In particular, $w_a(s)$ is
trapped near the heteroclinic orbit $\psi$ \eqref{defpsi}, and the
exponential bound follows from \eqref{samir} and \eqref{defpsi}.\\
Finally, if $a\in \q R_3$, then, following Section \ref{secR3}, we
write
\begin{equation}\label{defa'}
  a=(K,L) e^{-\frac {s_0}2}
\end{equation}
  as in \eqref{defa} and reduce to the
case where $0\le K\le L$ thanks to the symmetries of initial data
$w_0(\cdot, s_0)$ defined in \eqref{defw0}.\\
If $K+L\ge A$, then we see from Corollary \ref{corbR3} and its proof
that $w_a(s)$ is trapped near the heteroclinic orbit $\psi(s)$
\eqref{defpsi} and the exponential bound follows from estimate
\eqref{samir} given in Proposition \ref{proptrap2}.\\
If $K+L\le A$, using \eqref{alarge} and \eqref{defa'},
we may write $a=(K',L')e^{-\frac \sigma 2}$ for
some $\sigma \in [s_0,s_1]$ with $K'+L'=A$, as we did in
\eqref{defk'l'} and \eqref{k'l'}. Using Corollary \ref{corbR3'}, we
see  that $w_a(s)$ is trapped near the heteroclinic orbit $\psi(s)$,
and the conclusion follows from estimate \eqref{samir} again. This concludes
the proof of Proposition \ref{propwa}.
\end{proof}

\section{Proof of the main result} \label{secmainres}
In this section, we collect all the arguments to derive Theorem
\ref{th1}. We proceed in 3 subsections:\\
- We first prove the existence of a solution of equation \eqref{eqq}
trapped in the set $V_A(s)$ defined in Definition \ref{defvas}.\\
- Then, we derive a better description on larger sets, yielding the
so-called \textit{intermediate profile}.\\
- Finally, we prove that the origin is the only blow-up point and
derive the \textit{final profile} given in \eqref{finalprof}.

\subsection{Existence of a solution to equation \eqref{eqq} trapped in
  the set $V_A(s)$}
This is our statement in this section:
\begin{prop}[Existence of a solution $q(y,s)$ in the shrinking set
  $V_A(s)$]\label{proptrap0} $ $\\
  (i)  There exist
$\Czerop>0$,
   $A\ge 1$, $s_0\ge 1$ and a parameter
  $d=(d_{0,0}, d_{2,0}, d_{4,0}, d_{4,2}, d_{6,0})\in \q D(A,s_0)$
defined in Proposition \ref{propinit} such that the solution $w_0(y,s)$
of equation \eqref{eqw} with initial data at $s=s_0$ given by
\eqref{defw0} satisfies $q(s)\in V_A(s)$ and
$\|\nabla q(s)\|_{L^\infty}\le \Czerop$
for any $s\ge s_0$, where $q$ is defined in \eqref{defq} and the
set $V_A(s)$ is given in Definition \ref{defvas}. \\
(ii) If $u(x,t) = (T-t)^{-\frac 1{p-1}}w_0\left(\frac x{\sqrt{T-t}},-\log(T-t)\right)$
with $T=e^{-s_0}$, then $u$ is a solution of equation \eqref{equ}
which blows up only at the origin.
\end{prop}
\begin{nb}
Since initial data in \eqref{defw0} is symmetric with respect to the
axes and the bissectrices, the same holds for $w_0(y,s)$.
\end{nb}
Before proving Proposition \ref{proptrap0}, let us
recall
the following consequence of the Liouville Theorem stated in
Proposition \ref{propliou}:
\begin{cor}[Behavior of the gradient]\label{corgra}
  Following Proposition \ref{proptrap0}, it holds that
 $\|\nabla w_0(s)\|_{L^\infty}\to 0$ as $s\to \infty$.
\end{cor}
\begin{proof}
 This is a classical consequence of the  Liouville Theorem stated in Proposition
\ref{propliou}, which follows similarly as does
Proposition \ref{propgrad}. For a statement and a proof, see Theorem 1.1 page 141
in Merle and Zaag \cite{MZcpam98}.
\end{proof}
Let us now prove Proposition \ref{proptrap0}:
\begin{proof}[Proof of Proposition \ref{proptrap0}]
  $ $\\
(i)  This is a classical combination of our previous analysis, namely Lemma
  \ref{lemphi}, Propositions \ref{propflow}, \ref{propinit}, \ref{propgrad} and
  \ref{propwa}.
  The argument relies on a reduction of the
problem in finite dimensions, then the proof of the reduced problem
thanks to the degree theory. We give the argument only for the sake of
completeness. Expert readers may skip this proof.

\medskip

In order to apply the above-mentioned statements, let us fix
$A=\max(\AAzero, 1, \Avingtetun)$
and
$s_0=\max(\ssun(\delta), \Szero(A),\sszeroun(A), \sszerodeux(A,1),\svingtetun(A))$,
where $\delta$ is the constant
appearing in the definition of the profile $\varphi(y,s)$ and fixed at
the beginning of Section \ref{secformu}.

\medskip

We proceed by contradiction and assume that for all
$d=(d_{0,0}, d_{2,0}, d_{4,0}, d_{4,2}, d_{6,0})\in\q D(A,s_0)$
defined in Proposition \ref{propinit}, $q(d,y,s)$ doesn't
remain in the set $V_A(s)$ for all $s\ge s_0$, where $q(d,y,s) =
w_0(d,y,s) - \varphi(y,s)$, $\varphi(y,s)$ is defined in
\eqref{goodprof}, with the constants $\gamma$ and $\delta$ are fixed
in Lemma \ref{lemphi} and the beginning of Section \ref{secformu}, and
$w_0(d,y,s)$ is the solution of equation \eqref{eqw} with initial data
at $s=s_0$ given by \eqref{defw0}.

\medskip

Since the function given in
\eqref{121} is one-to-one from item (ii) in Proposition \ref{propinit}, we may take $\bar d =(\bar d_{0,0}, \bar
d_{2,0},  \bar d_{4,0},  \bar d_{4,2},  \bar d_{6,0}) \in [-1,1]^5$ as a new
parameter, where
\begin{equation}\label{defbd}
  \bar d_{i,j} = \frac{e^{2s_0}}Aq_{i,j}(d,s_0)\mbox{ if }(i,j) \in I_0,\;\;
  \bar d_{6,0}= \frac{e^{3s_0}}{As_0^2}q_{6,0}(d,s_0)
\end{equation}
and $I_0\equiv\{(0,0), (2,0), (4,0), (4,2)\}$. Accordingly, we
introduce the notation
\[
\bar q(\bar d,y,s) = q(d,y,s),\;\; \bar w_0(\bar d,y,s) = w_0(d,y,s)
\]
and so on. Since $q(d,y,s_0)\in V_A(s_0)$ by
Proposition \ref{propinit}, from continuity, we may introduce
$s^*(d)=\bar s^*(\bar d)$
as the minimal time such that $q(d,y,s)\in V_A(s)$ for all $s\in
[s_0,s^*]$ and an equality case occurs at time $s=s^*$ in one of the
$\le$ defining $V_A(s^*)$ in Definition \ref{defvas}.

\medskip

Applying
Proposition \ref{propwa}, we see that no equality case occurs for
$\|w_0(d,s^*)\|_{L^\infty}$.  We claim that no equality case occurs for
$|q_{6,2}(d,s^*)|=|q_{6,4}(d,s^*)|$ neither $\|q_-(d,s_0)\|_{L^2_\rho}$.\\
Indeed, if $s^*(d)=s_0$, this follows from Proposition
\ref{propinit}. If $s^*(d)>s_0$, then, this follows from Corollary
\ref{corq-}, which can be applied thanks to Proposition \ref{propinit}.
In other words, it holds that
\[
  \mbox{either }|q_{i,j}(d,s^*)|= A e^{-2s^*}\mbox{ for some }(i,j) \in
  I_0
  \mbox{ or }|q_{6,0}(d,s^*)|=A{s^*}^2e^{-3s^*}.
\]
This way, the following rescaled flow is well defined:
\[
\begin{array}{rcl}
 \Phi : [-1,1]^5 &\longrightarrow &\partial [-1,1]^5\\
  \bar d& \longmapsto &\frac{e^{2s^*}}A(q_{0,0}, q_{2,0}, q_{4,0}, q_{4,2}, \frac{e^{s^*}}{{s^*}^2}q_{6,0})(d,s^*(d)).
\end{array}
\]
We claim that:\\
(a) $\Phi$ is continuous;\\
(b) $\Phi_{|[-1,1]^5} = \Id$,
which will imply a contradiction, by the
degree theory and finish the proof. Let us prove (a) and (b).\\
\textit{Proof of (a)}: Using items (i) and (ii) in Proposition \ref{propflow}, we see that
the flow of the five components $|q_{i,j}|$ with $(i,j)\in
I_0\cup\{(6,0)\}$ is transverse outgoing on the corresponding boundary
(i.e. $Ae^{-2s}$ if $(i,j) \in I_0$ or $As^2e^{-3s}$ if
$(i,j)=(6,0)$). This implies the continuity of $\Phi$.\\
\textit{Proof of (b)}: If $\bar d \in \partial [-1,1]^5$ and $d$ is
the corresponding original parameter in $\q D$, then we see
by definition \eqref{defbd} that 
\begin{equation}\label{edges0}
  \mbox{either }|q_{i,j}(d,s_0)|= A e^{-2s_0}\mbox{ for some }(i,j) \in
  I_0
  \mbox{ or }|q_{6,0}(d,s_0)|=A{s_0}^2e^{-3s_0}.
\end{equation}
Since $q(d,s_0)\in V_A(s_0)$ by Definition \ref{defvas}, this implies
that $\bar s^*(\bar d) = s^*(d)=s_0$. From \eqref{edges0} and
\eqref{defbd}, we see that $\Phi(\bar d) = \bar d$ and (b) holds.\\
Since we know from the degree theory that there is no continuous function
from $[-1,1]^5$ to its boundary which is equal to the identity on the
boundary, a contradiction follows. Thus, there exists a parameter
$d=(d_{0,0}, d_{2,0}, d_{4,0}, d_{4,2}, d_{6,0})\in\q D(A,s_0)$
defined in Proposition \ref{propinit} such that $q(d,y,s)\in V_A(s)$
for all $s\ge s_0$. Applying Proposition \ref{propgrad} and Lemma \ref{lemphi}, we see that
$\|\nabla q(d,s)\|_{L^\infty}\le 1+C_0$, for all $s\ge s_0$.

\medskip

\noindent (ii) From the similarity variables definition \eqref{defw},
$u$ is indeed a solution of equation \eqref{equ} defined for all
$(x,t)\in \m R^2 \times [0,T)$. From the classical theory by Giga and
Kohn \cite{GKcpam89}, we know that when $s\to \infty$, $w_b(s) \to
\kappa$ when $b$ is a blow-up point, and $w_b(s) \to 0$ if not, in
$L^2_\rho$. Therefore, it's enough to prove that $w_0(s)\to\kappa$ and
$w_a(s)\to 0$ as $s\to \infty$ in order to conclude.\\  
First, since $q(s)\in V_A(s)$ for all $s\ge s_0$ from item (i), we see
from \eqref{smallVA}, the relation \eqref{defq} and the definition \eqref{goodprof} of
$\varphi$ that $w_0(s)\to\kappa$ as $s\to \infty$.\\
Second, if $a\neq 0$, we see that for any $s_1 \ge 2 \log\frac
A{|a_1|+|a_2|}$, 
\eqref{alarge} holds. Since Proposition \ref{propwa} holds here,
applying its item (ii), 
we see that for some $\bar s(a) \ge s_0$ and $\bar M(a)\ge 0$, for any
$s_1\ge \max(2 \log\frac A{|a_1|+|a_2|}, \bar s(a))$, for any $s\in [\bar s(a), s_1]$,
$\|w_a(s)\|_{L^2_\rho} \le \bar M(a) e^{-\frac s{p-1}}$. Making
$s_1\to \infty$, we see that $w_a(s)\to 0$. This concludes the proof
of Proposition \ref{proptrap0}. 
\end{proof}

\subsection{Intermediate profile}
From the construction given in Proposition \ref{proptrap0} together
with \eqref{smallVA} and the expansion \eqref{target} of the profile
$\varphi$ defined in \eqref{goodprof}, we have a solution $w_0(y,s)$ of equation
\eqref{eqw} such that our goal in \eqref{A4} holds, namely
\[
w_0(y,s) =\kappa  - e^{-s}h_2(y_1)h_2(y_2) +o(e^{-s})\mbox{ as }s\to \infty,
\]
in $L^2_\rho$ and uniformly on compact sets, by parabolic
regularity.

\medskip

For large $s$, the second term of this expansion will be
small with respect to the constant $\kappa$, which means that we only
see a flat shape for $w_0$. Extending the convergence beyond compact
sets, say, for $|y|\sim e^{\frac s4}$, would allow to see $w$
escape that constant. This is indeed what Vel\'azquez did in Theorem 1
page 1570 of \cite{Vcpde92}, proving that $w_0$ has the degenerate
profile \eqref{profnaif}, in the sense that
\[
  \sup_{|y|<Ke^{\frac s4}}
  \left|w_0(y,s) - \left(p-1+\frac{(p-1)^2}\kappa e^{-s}y_1^2y_2^2\right)^{-\frac 1{p-1}}\right| \to 0
  \mbox{ as } s\to \infty,
\]
for any $K>0$.

\medskip

If we see in this expansion that $w_0$ departs from the constant
$\kappa$ \eqref{conv}
for $y_1=K_1 e^{\frac s4}$ and
$y_2=K_2 e^{\frac s4}$ for some $K_1>0$ and $K_2>0$, this is not the
case on the axes, namely when $y_1=0$ or $y_2=0$. In accordance with
our idea in deriving the profile $\varphi$ in \eqref{goodprof}, our
techniques in this paper provide us with the following sharper profile,
valid on a larger region:
\begin{prop}[Intermediate profile] \label{propintprof} Consider $w_0(y,s)$ the solution of
  equation \eqref{eqw} constructed in Proposition \ref{proptrap0}. For
  any $K>0$, it holds that
   \[
    \sup_{e^{-s}y_1^2y_2^2+\delta e^{-2s}(y_1^6+y_2^6) <K}
    \left|w_0(y,s)-\Phi(y,s)\right|\to 0 \mbox{ as } t\to T,
    \]
    where $\Phi(y,s)$ is defined in \eqref{defPhi}.
   \end{prop}
  \begin{nb}
$\Phi(y,s)$ is referred to as the ``intermediate profile'', since it
holds in the region $\{s\ge s_0=-\log T\}$, which corresponds to the
region $\{0\le t<T\}$ in the original variables defined by
\eqref{defw}. Later in Proposition \ref{propfinal}, we will define the
``final profile'', as a limit of $u(a,t)$ as $t\to T$ when $a\neq 0$,
and holds in some sense for $t=T$, justifying its character as ``final''.
  \end{nb}
  
\begin{proof}
Consider $w_0$ the solution of equation \eqref{eqw} constructed in
Proposition \ref{proptrap0} and defined for all $(y,s)\in \m R^2 \times
[s_0,\infty)$. Given $K_0>0$ and $\epsilon>0$, we look
for $\Sun=\Sun(K_0,\epsilon)$ such that for any $S\ge \Sun$ and $Y\in \m
R^2$ such that  
\begin{equation}\label{R3y}
e^{-S}Y_1^2Y_2^2+\delta e^{-2S}(Y_1^6+Y_2^6) <K_0,
\end{equation}
it holds that
\begin{equation}\label{hadaf0}
 \left|w_0(Y,s)- \Phi(Y,s)\right |
  \le C \epsilon,
\end{equation}
for some universal $C>0$, where $\Phi(y,s)$ is defined in \eqref{defPhi}.
Consider then $S\ge s_0$ and $Y\in \m R^2$ such that \eqref{R3y}
holds. Since $w_0(y,s)$ is symmetric with respect to the axes and the
bissectrices (see the remark following Proposition \ref{proptrap0}), we
may assume that
\begin{equation}\label{Y1Y2}
0\le Y_1\le Y_2.
\end{equation}
Following our strategy in Section \ref{seclargekl},
we introduce $a=a(Y,S)$
such that 
\begin{equation}\label{defa2}
ae^{\frac S2} = Y.
\end{equation}
Since
\begin{equation}\label{wwaS}
w_0(Y,S)=w_a(0,S)
\end{equation}
by \eqref{wawb},
our conclusion will follow from
the study of $w_a(y,s)$ for $s\in [s_0,S]$.
From \eqref{R3y} and \eqref{defa2}, we see that
\[
G_0(a) \le \frac {K_0(p-1)}\kappa e^{-S},
\]
where $G_0$ is defined in \eqref{defG0}. In particular,
\begin{equation}\label{boundai}
0\le a_1\le a_2\le \left(\frac{K_0(p-1)}{\kappa \delta}\right)^{\frac 16}
e^{-\frac S6}.
\end{equation}
Consider $\bar A>0$ to be fixed large enough later.
Taking $S\ge\Sdeux$ for some $\Sdeux(K_0,\bar A)\ge s_0$, we see that
\[
a_1+a_2 \le \bar A e^{-\frac{s_0}2}.
\]
Therefore, we may introduce $\sigma\ge s_0$, $K'$ and $L'$ such that 
\begin{equation}\label{defkls}
a=(K',L')e^{-\frac \sigma 2}\mbox{ with }0\le K'\le L' \mbox{ and } K'+L' =
\bar A.
\end{equation}
The conclusion will follow by applying Lemma \ref{cordyn'} and
Proposition \ref{proptrap2}, exactly as we did in Corollaries
\ref{corbR3'} and \ref{corbR3}. In order to apply those two
statements, let us assume that
$\bar A\ge \max(A,\bAsixp)$ 
and consider
$\eta^*\in (0, \etazerop]$,
where $A$ is fixed in Proposition \ref{proptrap0} and
$\bAsixp$ together with $\etazerop$
are introduced in Lemma \ref{cordyn'}.

\medskip

From \eqref{boundai} and \eqref{defkls}, we see that whenever $S\ge
\Strois$ for some $\Strois(K_0, \bar A,\bar\eta^*)\ge s_0$, it follows that
$\sigma \ge \sonzep(\bar A,\eta^*)$
defined in Lemma \ref{cordyn'}. Let us then assume that $S\ge \Strois$.

\medskip

Since $\bar A\ge A$, we see from Proposition \ref{proptrap0} and
Definition \ref{defvas} that $q(s)\in V_A(s) \subset V_{\bar A}(s)$
and $\nabla q(s)\in L^\infty$, for all $s\ge s_0$.

\medskip

Using \eqref{defkls}, we see that Lemma \ref{cordyn'} applies with any
$s_1\ge \sigma$. In particular, if we introduce $s^*$ such that
\begin{equation}\label{defs*'}
 e^{s^*-\sigma} \iota' =\eta^*
\end{equation}
  where $\iota'$ is defined in
\eqref{defiota'}, then we see that $s^*\ge \sigma$ and for all
$s\in [\sigma, s^*]$, 
\begin{equation}\label{vanice'}
\|w_a(s) - \left(\kappa
 - e^{s-\sigma}\iota'\right)\|_{L^2_\rho}
\le \bMsixp\left(\eta^*+\frac 1{\bar A}\right)\eta^*+\bMsixp
e^{-\frac\sigma 3},
\end{equation}
where the constant
$\bMsixp$
is defined in Lemma
\ref{cordyn'}.

\medskip

Assuming that $\eta^*<\kappa$, we may introduce $\sigma^*\in \m R$
such that
\begin{equation}\label{defsig*}
\psi(\sigma^*) = \kappa-\eta^*,
\end{equation}
where $\psi$ is defined in \eqref{defpsi}. Using \eqref{dlpsi},
\eqref{defkls} and \eqref{boundai}, we see that whenever $\bar A \ge
\Aquatre$, $\eta^*\le \equatre$ and $S\ge \Squatre$, for some
$\Aquatre \ge 1$, $\equatre >0$ and $\Squatre(K_0,\bar A, \eta^*)$, we
have
\begin{equation}\label{initwa}
  \|w_a(s^*)-\psi(\sunos)\|_{L^2_\rho}
 \le \bMsixp\left(\eta^*+\frac 1{\bar A}\right)\eta^*+\bMsixp
e^{-\frac\sigma 3}
   \le \frac {|{\psi'}(\sunos)|}{M_1},
 \end{equation}
where $M_1$ is introduced in Proposition \ref{proptrap2}.  Since
$\nabla w_a(s^*)\in L^\infty$ and $\|w_a(s)\|_{L^\infty}\le 2 \kappa$
for any $s\ge s^*$, thanks to Proposition \ref{proptrap0},
together with the definition \eqref{defq} of $q$ and the
transformation \eqref{wawb}, Proposition \ref{proptrap2} applies to
the shifted function $w_a(y,s+s^*)$ (with any $\sigma_1\ge 0$), and
we see that for all $s\ge s^*$, we have 
\begin{align*}
\|w_a(s)-\psi(s+\sunos-s^*)\|_{L^2_\rho}
 \le &\; M_1 \|w_a(s^*) -\psi(\sunos)\|_{L^2_\rho}
  \frac{|{\psi'}(s+\sunos)|}{|{\psi'}(\sunos)|}.
\end{align*}
Since $\psi'\in L^\infty$ by definition \eqref{defpsi} of $\psi$,
using again \eqref{initwa} and \eqref{dlpsi}, we see that
\begin{equation}\label{s1}
  \forall s\ge s^*,\;\;
  \|w_a(s)-\psi(s+\sunos-s^*)\|_{L^2_\rho} \le
2 M_1 \bMsixp
  \|\psi'\|_{L^\infty}
  \left(\eta^*+\frac 1{\bar A} + \frac{e^{-\frac \sigma3}}{\eta^*}\right),
\end{equation}
whenever $\eta^*\le \ecinq$, for some $\ecinq>0$.

\bigskip

Now, if $s\in [\sigma,s^*]$, using \eqref{calculrho}, we write
\begin{equation}\label{spetit}
 \|w_a(s)-\psi(s+\sunos-s^*)\|_{L^2_\rho} \le
 \|w_a(s)-  \left(\kappa - e^{s-\sigma}\iota'\right)\|_{L^2_\rho}
 +|\psi(s+\sunos-s^*)-\kappa|+e^{s-\sigma}\iota'.
\end{equation}
Since $s\le s^*$, hence $s+\sunos-s^*\le \sunos$ and $\psi$ is decreasing by definition
\eqref{defpsi}, we write from \eqref{defsig*} and \eqref{defs*'}
\begin{equation}\label{pequeno}
  |\psi(s+\sunos-s^*)-\kappa| \le |\psi(\sunos)-\kappa|= \eta^*
  \mbox{ and }e^{s-\sigma}\iota' \le e^{s^*-\sigma}\iota'=\eta^*.
\end{equation}
Therefore, using \eqref{spetit} together with \eqref{vanice'}, we see
that 
\begin{equation}\label{s2}
\forall s\in[\sigma,s^*],\;\;
  \|w_a(s)-\psi(s+\sunos-s^*)\|_{L^2_\rho} \le
\bMsixp\left(\eta^*+\frac 1{\bar A}\right)\eta^*+\bMsixp
e^{-\frac\sigma 3} +2\eta^*.
\end{equation}

Finally, if $s\in [s_0, \sigma]$, noting again that
$s+\sunos-s^*\le \sigma+\sunos - s^*\le \sunos$, we may use
\eqref{pequeno} and write similarly
\[
  \|w_a(s)-\psi(s+\sunos-s^*)\|_{L^2_\rho} \le
  \|w_a(s)-\kappa\|_{L^2_\rho} + \eta^*.
\]
Recalling that $\nabla w(s)\in L^\infty$ from the construction in Proposition
\ref{proptrap0}, we may use a Taylor expansion as in the proof of
Claim \ref{clreduc} and write
\begin{equation}\label{Taylor}
  \|w_a(s)-\kappa\|_{L^2_\rho} \le C|w_a(0,s)-\kappa|+
 C\|\nabla w_0(s)\|_{L^\infty}.
\end{equation}
Since $w_a(0,s) = w_0(ae^{\frac s2},s)$ by \eqref{wawb} and 
\[
  |ae^{\frac s2}|=|(K',L')|e^{\frac{s-\sigma}2}
  \le \sqrt{{K'}^2+{L'}^2}e^{\frac{s-\sigma}2} \le \bar A \sqrt 2
\]
by \eqref{defkls} and the fact that $s\le \sigma$, we may use again a
Taylor expansion for $w_0$ and write
\[
  |w_a(0,s)-\kappa|=|w_0(ae^{\frac s2},s)-\kappa| \le C \|w_0(s) -\kappa\|_{L^2_\rho} +
  C(1+\bar A)\|\nabla w_0(s)\|_{L^\infty}. 
\]
Furthermore, using  the construction in Proposition \ref{proptrap0}, together with
the definitions \eqref{defq} and \eqref{goodprof} of $q(y,s)$ and
$\varphi(y,s)$, and estimate \eqref{smallVA}, we write
\[
  \|w_0(s) - \kappa\|_{L^2_\rho}\le
  \|q(s) \|_{L^2_\rho} 
+ \|\varphi(s) -\kappa\|_{L^2_\rho} \le CAse^{-2s} + Ce^{-s},
\]
where $A$ is given in Proposition \ref{proptrap0}.
Collecting the previous estimates, we write
\begin{equation}\label{s3}
  \forall s\in [s_0,\sigma],\;\;
  \|w_a(s)-\psi(s+\sunos-s^*)\|_{L^2_\rho} \le
CAse^{-2s} + Ce^{-s}
  + C(1+\bar A)\|\nabla w(s)\|_{L^\infty} +   \eta^*.
\end{equation}

Using \eqref{defkls}
and \eqref{boundai} and the gradient estimate in Corollary \ref{corgra}, we see from \eqref{s1}, \eqref{s2} and \eqref{s3} that taking $\bar A\ge
\Asept$, $\eta^*\le \esept$ and $S\ge \Ssept$ for some
$\Asept(\epsilon)\ge 1$, $\esept(\epsilon)>0$ and $\Ssept(\epsilon, \eta^*,
K_0, \bar A)\ge s_0$, we get for all
$s\ge s_0$,
\[
 \|w_a(s)-\psi(s+\sunos-s^*)\|_{L^2_\rho} \le 4\epsilon.
\]
Taking $s=S$, using \eqref{defa2} and \eqref{wwaS}, then proceeding as
with \eqref{Taylor}, we write
\begin{align*}
 |w_0(Y,S) -&\psi(S+\sunos-s^*)|
  = |w_a(0,S) -\psi(S+\sunos-s^*)|\\
  \le &\;C \|w_a(S)-\psi(S+\sunos-s^*)\|_{L^2_\rho}
  +C\|\nabla w_0(S)\|_{L^\infty}
\le C\epsilon  +C\|\nabla w_0(S)\|_{L^\infty}.
\end{align*}
Taking $S$ larger and using again the gradient estimate of Corollary
\ref{corgra}, we see that 
\begin{equation}\label{z1}
 |w_0(Y,S) -\psi(S+\sunos-s^*)| \le 2C\epsilon.
\end{equation}

Since $e^{s^*-\sigma} \iota' =\eta^*$ from \eqref{defs*'}, $e^{\sunos}\sim \eta^*\frac{(p-1)}\kappa$ as $\eta^*\to 0$ from \eqref{dlpsi},
and $e^{S-\sigma}\iota' =e^{-S}Y_1^2Y_2^2+\delta e^{-2S}(Y_1^6+Y_2^6)$
by definitions \eqref{defa2}, \eqref{defkls} and \eqref{defiota'} of
$a$, $(K',L')$ and $\iota'$, we write 
\begin{align*}
 & e^{S+\sunos-s^*}= e^{\sunos}e^{-(s^*-\sigma)}e^{S-\sigma}
  =\left(\frac{p-1}\kappa +\bar \eta\right)
  \eta^*e^{-(s^*-\sigma)}e^{S-\sigma}
  = \left(\frac{p-1}\kappa +\bar \eta\right) \iota' e^{S-\sigma}\\
  =&\; \left(\frac{p-1}\kappa +\bar \eta\right)
  [e^{-S}Y_1^2Y_2^2+\delta e^{-2S}(Y_1^6+Y_2^6)],
\end{align*}
where $\bar \eta \to 0$ as $\eta^*\to 0$. Recalling condition
\eqref{R3y}, then making an expansion of $\psi$ defined in
\eqref{defpsi}, we see that for $\eta^*\le \ehuit$ for some
$\ehuit(K_0,\epsilon)>0$, we have
\begin{equation}\label{z2}
\left|\psi(S+\sunos-s^*)
- \left [p-1+\frac{(p-1)^2}\kappa\left( e^{-s}Y_1^2Y_2^2
  + \delta  e^{-2s}(Y_1^6+Y_2^6)\right)\right]^{-\frac 1{p-1}} \right|
\le \epsilon.
\end{equation}
Combining \eqref{z1}
and \eqref{z2} yields \eqref{hadaf0}.
This concludes the proof of Proposition \ref{propintprof}.

\end{proof}

\subsection{Derivation of the final profile}  

Throughout this section, we consider $u(x,t)$ the solution of equation
\eqref{equ} constructed in Proposition \ref{proptrap0} which blows up
at time $T>0$ only at the origin.

\medskip

Using Propositions \ref{proptrap0} and \ref{propintprof},  together
with the similarity variables' transformation \eqref{defw}, we derive the
following:
\begin{cor}[Intermediate profile for $u(x,t)$] \label{corintprof}
  For any $K>0$, it holds that
  \[
    \sup_{x_1^2x_2^2+\delta(x_1^6+x_2^6)\le K(T-t)}
   \left |(T-t)^{\frac 1{p-1}} u(x,t) - 
    \left[p-1+\frac{(p-1)^2}\kappa\frac{[x_1^2x_2^2+\delta(x_1^6+x_2^6)]}{T-t}\right]^{-\frac 1{p-1}} \right|
 \]
goes to $0$ as $t\to T$.
  \end{cor}
Using this result, we derive the following estimate for the final profile:
\begin{prop}[Final profile] \label{propfinal}
 For any $a\neq 0$, $u(a,t)$ has a limit as $t\to T$, denoted by
$u(a,T)$. Moreover,
\begin{equation}\label{defu*}
  u(a,T) \sim u^*(a)\mbox{ as } a\to 0,\mbox{ where }
  u^*(a) =\left[ \frac{(p-1)^2}\kappa
    (a_1^2a_2^2+\delta(a_1^6+a_2^6)) \right]^{-\frac 1{p-1}}.
\end{equation}
  \end{prop}
  \begin{nb}
Following the remark given after Proposition \ref{propintprof}, we
justify the name of $u^*$ as ``final profile''.
\end{nb}

This result is in fact a consequence of the
following ODE localization property of the PDE, proved in our earlier
paper \cite{MZcpam98}, and which is a direct consequence of the
Liouville Theorem stated in Proposition \ref{propliou}:
\begin{prop}[ODE localization for $u(x,t)$]\label{propode}
For any $\epsilon>0$, there exists $C_\epsilon>0$ such that for all
$(x,t)\in \m R^2 \times [0,T)$,
\[
  (1-\epsilon) u(x,t)^p - C_\epsilon \le \partial_t u(x,t)
  \le (1+\epsilon) u(x,t)^p + C_\epsilon.
\]
  \end{prop}
  \begin{nb}
Since initial data given in \eqref{defw0} are
nonnegative, the same holds for $w_0(y,s)$, solution of equation
\eqref{eqw}, and $u(x,t)$, solution of equation \eqref{equ}, both constructed in
Proposition \ref{proptrap0}. This justifies the notation $u(x,t)^p$
without absolute value.
\end{nb}
\begin{proof}
See Theorem 1.7 page 144 in \cite{MZcpam98}. The fact that initial
data given in \eqref{defw0} is in $W^{2,\infty}(\m R^2)$ is necessary
to have the estimate in $\m R^2 \times [0,T)$, otherwise, if initial
data in only in $L^\infty(\m R^2)$, we will have
$u(t_0)\in W^{2,\infty}(\m R^2)$ for any $t_0>0$ by parabolic
regularity, and the statement will hold uniformly in $\m
R^2\times[t_0, T)$ from any $t_0\in(0,T)$.
\end{proof}

Now, we are ready to prove Proposition \ref{propfinal}, thanks to
Corollary \ref{corintprof} and Proposition \ref{propode}. 
  \begin{proof}[Proof of Proposition \ref{propfinal}]
 The existence of the limiting profile $u(a,T)$ follows by compactness,
exactly as in Proposition 2.2 page 269 in Merle \cite{Mcpam92}.
It remains to prove \eqref{defu*}.

\medskip

Consider any $K_0>0$.
Given any small enough $a\neq 0$, we introduce the
time $t^*(a) \in [0,T)$ such that
\begin{equation}\label{deft*}
a_1^2a_2^2+\delta(a_1^6+a_2^6)= K_0\frac \kappa{p-1}(T-t^*(a)).
\end{equation}
Note that $t^*(a) \to T$ as $a\to 0$. The conclusion will follow from
the study of $u(a,t)$ on the time interval $[t^*(a),T)$. Consider first some arbitrary
$\epsilon>0$.\\
\textit{Step 1: Initialization}. From Corollary \ref{corintprof}, we
see that 
\begin{equation}\label{init*}
\left|(T-t^*(a))^{\frac 1{p-1}}u(a,t^*(a))- \kappa(1+K_0)^{-\frac
    1{p-1}}\right|
\le \epsilon \kappa(1+K_0)^{-\frac 1{p-1}},
\end{equation}
provided that $a$ is small enough.\\
\textit{Step 2: Dynamics for $t\in [t^*(a,T)$}. From Proposition
\ref{propode}, we see that
\begin{equation}\label{ode*}
  \forall t\in [t^*(a),T),\;\;
  (1-\epsilon) u(a,t)^p - C_\epsilon \le \partial_t u(x,t)
  \le (1+\epsilon) u(a,t)^p + C_\epsilon 
\end{equation}
for some $C_\epsilon>0$. Using \eqref{init*} and \eqref{ode*}, one
easily shows that for $a$ small enough,
\[
\forall t\in [t^*(a),T),\;\;C_\epsilon \le \epsilon u(a,t)^p.
\]
Therefore, we see from \eqref{ode*} that
\[
    \forall t\in [t^*(a),T),\;\;
(1-2\epsilon) u(a,t)^p \le \partial_t u(x,t)
  \le (1+2\epsilon) u(a,t)^p.
\]
Using again \eqref{init*}, we may explicitly integrate the 2 differential
inequalities we have just derived and write for all $t\in
[t^*(a),T)$,
\begin{align*}
  \kappa \left[(T-t^*(a))(1+K_0)(1-\epsilon)^{1-p}
  -(1-2\epsilon)(t-t^*(a))\right]^{-\frac 1{p-1}} \le u(a,t) \\
  \le
   \kappa \left[(T-t^*(a))(1+K_0)(1+\epsilon)^{1-p}
    -(1+2\epsilon)(t-t^*(a))\right]^{-\frac 1{p-1}}. 
\end{align*}
Making $t\to T$, we see that
\begin{align*}
  \kappa (T-t^*(a))^{-\frac 1{p-1}} \left[(1+K_0)(1-\epsilon)^{1-p}
  -(1-2\epsilon)\right]^{-\frac 1{p-1}} \le u(a,T) \\
  \le \kappa (T-t^*(a))^{-\frac 1{p-1}} \left[(1+K_0)(1+\epsilon)^{1-p}
  -(1+2\epsilon)\right]^{-\frac 1{p-1}}.
 \end{align*}
Taking $\epsilon$ small enough, we see that
\[
  \left| u(a,T)
    - \kappa (T-t^*(a))^{-\frac 1{p-1}} K_0^{-\frac 1{p-1}}\right|
  \le C(K_0) \epsilon  (T-t^*(a))^{-\frac 1{p-1}}.
\]
Using the definitions \eqref{deft*} and \eqref{defu*} of $t^*(a)$ and
$u^*(a)$, we see that
\[
|u(a,T) - u^*(a)|\le C'(K_0) \epsilon u^*(a).
\]
Since $\epsilon>0$ was arbitrarily chosen, this concludes the proof of
Proposition \ref{propfinal}.
  \end{proof}
  
\subsection{Conclusion of the proof of the main statement}
In this section, we gather the previous statements to derive Theorem
\ref{th1}.

\begin{proof}[Proof of Theorem \ref{th1}]
Let us first note that $\delta>0$ was chosen large enough in the
beginning of Section \ref{secformu}. In fact, as one can easily see
from the proof, our analysis holds for any $\delta \ge \delta_0$, for
some large enough $\delta_0$.

\medskip

  Using Proposition \ref{proptrap0}, we have the existence
  of a solution $u(x,t)$ of equation \eqref{equ} blowing up at some time
  $T>0$ only at the origin, such that $q(s) \in V_A(s)$ given in
  Definition \ref{defvas} for any $s\ge s_0 = -\log T$, where $q$ is
  defined in \eqref{defq}. In particular, $\|q(s)\|_{L^2_\rho}\le A
  se^{-2s}$ for all $s\ge s_0$, by \eqref{smallVA}. From the expansion \eqref{target} of the
  profile $\varphi$ defined in \eqref{goodprof}, we see that item (i)
  of Theorem \ref{th1} holds in $L^2_\rho(\m R^2)$. Using parabolic
  regularity, the convergence holds also uniformly in any compact set
  of $\m R^2$.

  \medskip

  As for items (ii) and (iii) of Theorem \ref{th1}, they directly
  follow from Propositions \ref{propintprof} and \ref{propfinal}.
This concludes the proof of Theorem \ref{th1}.
\end{proof}

\section{Proof of the technical details}\label{sectech}
In this section, we proceed in several subsections to justify all the
technical ingredients, which were stated without proofs in the
previous sections.
\subsection{Estimates of the remainder term $R(y,s)$ defined in \eqref{defR}}\label{secR}
In this section, we prove estimate \eqref{boundR} and Lemma
\ref{lemR}. Since the former is a direct consequence of the latter,
thanks to the decomposition \eqref{decomp}, we only prove the lemma.

\begin{proof}[Proof of Lemma \ref{lemR}]
Take $r\ge 2$.  In the following, all the expansions are valid in
  $L^r_\rho$ for $s\to \infty$. This is the case in particular for \eqref{target}.
 As one may convince himself by differentiating the expression
  \eqref{goodprof} then making an expansion as $s\to \infty$, we will
  find the same result as if we have directly differentiated the
  expression \eqref{target}. Using \eqref{Lhihj},
  it follows that
  \begin{align*}
\partial_s\varphi(y,s) =& \;e^{-s} h_2h_2 \nonumber\\
                                 &-2e^{-2s}\left\{-\delta (h_6h_0+h_0h_6)
                                    +\gamma (h_4h_2+h_2h_4)
                +\frac p{2\kappa} h_4h_4\right\}
    +O\left( e^{-3s}\right),\nonumber\\
(\q L-1)  \varphi(y,s) =& \; 2 e^{-s} h_2h_2 \\
                                 & +e^{-2s}\left\{3\delta (h_6h_0+h_0h_6)
                                    -3\gamma (h_4h_2+h_2h_4)
                -\frac {2p}\kappa h_4h_4\right\}
                 +O\left( e^{-3s}\right).\nonumber
  \end{align*}
 Since $\|\varphi(\cdot,s)\|_{L^\infty} \le \kappa+C_0$ for all $s\ge
 0$ (see item (ii) of Lemma \ref{lemphi}), using a Taylor expansion, we write
  \[
    \varphi^p = \kappa^p+p\kappa^{p-1}(\varphi-\kappa)
    +\frac{p(p-1)}2\kappa^{p-2}(\varphi-\kappa)^2 +O((\varphi-\kappa)^3).
  \]
  Since
  \[
    \varphi - \kappa = O(e^{-s})\mbox{ and }
    (\varphi - \kappa)^2 = e^{-2s}h_2(y_1)^2 h_2(y_2)^2+O(e^{-3s})
  \]
  from \eqref{target}, noting that
  \[
h_2(\xi)^2=h_4(\xi)+8h_2(\xi)+8h_0(\xi)
\]
by definition \eqref{defhj}, the result follows from the above
estimates, thanks to the definitions \eqref{defR} and \eqref{conv}
of $R(y,s)$ and $\kappa$.  This concludes the proof of Lemma \ref{lemR}.
 \end{proof}

\subsection{Dynamics of equation \eqref{eqq}} \label{sectechdyn}
 We prove Proposition \ref{propdyn} in this section.

 \begin{proof}[Proof of Proposition \ref{propdyn}]
    Consider $q(s)\in V_A(s)$ a solution of equation \eqref{eqq} on
    some time interval $[s_0,s_1]$ for some $A>0$. Note that
    $\|q(s)\|_{L^\infty}$ is uniformly bounded, hence, \eqref{boundB}
    holds. Assume that the initial condition \eqref{qs0}
    holds. Consider then $s\in[s_0,s_1]$.\\ 
    (i) Consider
$i\in \m N$
     and $0\le j \le i$. If we multiply equation
    \eqref{eqq} by $k_{i-j}(y_1)k_j(y_2)$ where the polynomial $k_n$
    is introduced in \eqref{defkn}, we get the conclusion,
    by definition
    \eqref{defvij} of $q_{i,j}$ together with \eqref{Lhihj},
    \eqref{boundV}, \eqref{smallVA}, \eqref{boundB} and H\"older's
    inequality.\\ 
     (ii) If $P_-$ is the $L^2_\rho$ projector on
\begin{equation}\label{defE-}
   E_-\equiv {\rm span
    }\{h_{i-j}h_j \;|\;i\ge 8\mbox{ and }0\le j\le i\},
\end{equation}
    then $q_-(s)=P_-(q(s))$ and $R_-(s) = P_-(R(s))$. Applying this
    projector to equation \eqref{eqq}, we get
\[
\partial_s q_- = \q L q_-+P_-(Vq+B)+R_-.
\]
Multiplying by $q_-\rho$ then integrating in space, we write
\[
  \frac 12\frac d{ds} \|q_-\|_{L^2_\rho}^2 =
 \int_{\m R^2} q_- \q L q_- \rho dy
  +\int_{\m R^2} q_- [P_-(Vq+B)+R_-] \rho dy.
\]
Since the highest eigenvalue of $\q L$ on $E_-$ is $\lambda = -3$ (see
\eqref{Lhihj}), it follows that
\[
\int_{\m R^2} q_- \q L q_- \rho dy \le -3\|q_-\|_{L^2_\rho}^2.
\]
Recalling that $|B|\le C|q|^{\bar p}$ where $\bar p = \min(p,2)>1$  thanks to \eqref{boundB}, then using the Cauchy-Schwarz inequality, we write
\begin{align*}
  \left|\int_{\m R^2} q_- [P_-(Vq+B)+R_-] \rho dy\right|
&\le \|q_-\|_{L^2_\rho}\left[\| P_-(Vq+B)\|_{L^2_\rho}+\|R_-\|_{L^2_\rho}\right]\\
&\le \|q_-\|_{L^2_\rho}\left[\| Vq \|_{L^2_\rho}+C\|B\|_{L^2_\rho}+\|R_-\|_{L^2_\rho}\right]\\
  &\le \|q_-\|_{L^2_\rho}\left[\|V\|_{L^4_\rho}\|q\|_{L^4_\rho}
    +C\|q\|_{L^{2\bar p}_\rho}^{\bar p}+\|R_-\|_{L^2_\rho}\right].
\end{align*}
Using the delay regularizing estimate given in Proposition
\ref{propdelay}, together with the estimate \eqref{boundV}
satisfied by $V$,
we conclude the proof of Proposition \ref{propdyn}.
    \end{proof}

\subsection{Parabolic regularity}\label{sectechdelay}
In this section, we prove Proposition \ref{propdelay}, which gives a
parabolic regularity estimate for equation \eqref{eqq}. Through a
Duhamel formulation for equation \eqref{eqq}, we reduce the question
to the linear level, where the estimate was already proved by Herrero
and Vel\'azquez \cite{HVihp93}, as we recall in the following, along
with another classical regularity estimate:
\begin{lem}(Regularizing effect of the operator $\q L$)\label{lemVel}
 $ $\\
(i) {\bf (Herrero and Vel\'azquez \cite{HVihp93})} 
For any $r>1$, $\bar r>1$, $v_0 \in L^r_\rho(\m R^N)$ and
$s>\max\left(0, -\log(\frac{r-1}{\bar r-1})\right)$, it holds that
\[
  \|e^{s\q L}v_0\|_{L^{\bar r}_\rho(\m R^N)}\le  \frac{C(r,\bar r)
e^s
  }{(1-e^{-s})^{\frac N{2r}}(r-1-e^{-s}(\bar r-1))^{\frac N{2\bar r}}} \|v_0\|_{L^r_\rho(\m R^N)}.
\]
(ii) Consider $r\ge 2$ and $v_0\in L^r_\rho(\m R^N)$
such that $|v_0(y)|+|\nabla v_0(y)|\le C(1+|y|^k)$ for some $k\in \m N$.
Then, for all $s\ge 0$, we have
$\|e^{s\q L }v_0\|_{L^r_\rho(\m R^N)}\le Ce^s\|v_0\|_{L^r_\rho(\m
  R^N)}$.\\
(iii) For any $v_0\in L^\infty(\m R^N)$ and $s\ge 0$, it holds that
$e^{s\q L} v_0\in L^\infty$ with
$\|e^{s\q  L}v_0\|_{L^\infty} \le Ce^s\|v_0\|_{L^\infty}$.\\
(iv) For any $v_0\in W^{1,\infty}(\m R^N)$ and $s>0$, it holds that
$e^{s\q L} \nabla v_0\in L^\infty$ with 
$\|e^{s\q  L}\nabla v_0\|_{L^\infty} \le \frac{Ce^s}{\sqrt{1-e^{-s}}}\|v_0\|_{L^\infty}$.
\end{lem}
\begin{nb}
  When $r=\bar r$ in item (i), then constant in the right-hand side blows
  up as $s\to 0$, which may seem surprising, since one expects some
  continuity of the norm in $L^r_\rho$. In fact, such a continuity is
  obtained in item (ii), thanks to an additional growth control of the gradient.
 \end{nb}
\begin{proof}
  $ $\\
(i) See Section 2 page 139 in \cite{HVihp93} where the one-dimensional
case is proved. The adaptation to higher dimensions is
straightforward.\\
(ii) See Lemma 2.1 in \cite{MZimrn21}.\\
(iii) and (iv) These estimates are straightforward from the definition
of the kernel:
 \begin{equation}\label{kernel}
e^{s\q L}(y,x)= \frac{e^s}{[4\pi(1-e^{-s})]^{N/2}}\exp\left[-
  \frac{|ye^{-\frac s2}-x|^2}{4(1-e^{-s})}\right].
\end{equation}
\end{proof}
Let us now give the proof of Proposition \ref{propdelay}.

\medskip

\begin{proof}[Proof of Proposition \ref{propdelay}]
Let us consider $A>0$ and $s_1\ge s_0 \ge 0$, and assume that $q(y,s)$
satisfies equation \eqref{eqq} for all $(y,s)\in \m R^2\times [s_0,s_1]$, with $q(s) \in
V_A(s)$ defined in \eqref{defvas}. Assume also that the initial
condition \eqref{qs0} holds. Note in particular that
$\|q(s)\|_{L^\infty}$ is uniformly bounded. Using \eqref{eqlin} together with
\eqref{boundVb}, we see that for some universal constant $C^*>0$ and for almost every
$(y,s)\in \m R^2 \times[s_0,s_1]$, we have 
\begin{equation}\label{difq}
\partial_s |q| \le (\q L +C^*)|q|+|R|.
\end{equation}
Consider now some $r\ge 2$ and introduce $s^* = -
\log(\frac{2-1}{r-1})\ge 0$, which is involved in the statement of
Lemma \ref{lemVel}. Consider then some $s\in [s_0,s_1]$, and introduce
\begin{equation}\label{defs'}
s'=\max[s_0,s-s^*].
\end{equation}
We then introduce the following Duhamel formulation of \eqref{difq}
on the interval $[s',s]$:
\[
|q(s)| \le e^{(\q L +C^*)(s-s')}|q(s')|+\int_{s'}^s  e^{(\q L +C^*)(s-\tau)}|R(\tau)|d\tau,
\]
which implies that
\begin{equation}\label{qij}
  \|q(s)\|_{L^r_\rho(\m R^2)} \le I+J
\end{equation}
where
\[
  I=\|e^{(\q L+C^*)(s-s')}|q(s')|\|_{L^r_\rho(\m R^2)}\mbox{ and }
J=\int_{s'}^s \| e^{(\q L +C^*)(s-\tau)}|R(\tau)|\|_{L^r_\rho(\m R^2)} d\tau.
\]
Let us first bound $J$ then $I$.\\
Since $R(y,s)$ and $\nabla R(y,s)$ are clearly bounded by a polynomial in $y$ by
definition \eqref{defR}, we write from item (ii) of Lemma \ref{lemVel}
and the bound \eqref{boundR} on $R$:
\begin{align*}
  J \le &\; C\int_{s'}^s e^{(1+C^*)(s-\tau)}\|R(\tau)\|_{L^r_\rho(\m R^2)}d\tau\\
  \le &\; C\int_{s'}^s e^{(1+C^*)(s-\tau)} e^{-2\tau}d\tau
  \le e^{(1+C^*)(s-s')}(s-s')e^{-2s'}d\tau.
\end{align*}
Since $s'\ge s-s^*$ from \eqref{defs'}, it follows that $s-s'\le s^*$
and $e^{-2s'}\le e^{2s^*-2s}$, hence
\begin{equation}\label{j}
J\le Cs^*e^{(3+C^*)s^*}e^{-2s}.
\end{equation}
In order to bound $I$, we consider 2 cases:\\
- If $s-s^*\ge s_0$, then $s'=s-s^*$ and $s-s'=s^*$. Using item (i) in
Lemma \ref{lemVel}, we write
\[
I\le C(r)\|q(s-s^*)\|_{L^2_\rho}.
\]
Since $q(s-s^*)\in V_A(s-s^*)$ by hypothesis, it follows from
\eqref{smallVA} that
\[
\|q(s-s^*)\|_{L^2_\rho} \le CA(s-s^*)e^{-2(s-s^*)}\le Ce^{2s^*}Ase^{-2s}.
\]
Therefore, we conclude that
\begin{equation}\label{i1}
I\le C(r) Ase^{-2s}.
\end{equation}
- If $s-s^*<s_0$, then $s'=s_0$. This time, from hypothesis
\eqref{qs0}, we see that we can apply item (ii) of Lemma \ref{lemVel}
to write
\[
I \le Ce^{(1+C^*)(s-s_0)}\|q(s_0)\|_{L^r_\rho} \le  C(r)e^{(1+C^*)s^*}As_0e^{-2s_0}.
\]
Since $s_0\le s <s_0+s^*$, it follows that
\begin{equation}\label{i2}
  I \le C(r) Ase^{-2s}.
\end{equation}
Combining \eqref{qij}, \eqref{j}, \eqref{i1} and \eqref{i2} concludes
the proof of Proposition \ref{propdelay}.
\end{proof}

\subsection{Position of the flow of \eqref{eqq} on the boundary of $V_A(s)$}\label{sectechflow}
We prove Proposition \ref{propflow} in this section.

\begin{proof}[Proof of Proposition \ref{propflow}]
 Consider $q(s)\in V_A(s)$ a solution of equation \eqref{eqq} on some time
    interval $[s_0,s_1]$ for some $A>0$, where $s_0\ge \sinit(A)$ is
    large enough so that both  Proposition \ref{propdyn} and Lemma
    \ref{lemR} apply. Assume that the initial
    condition \eqref{qs0} holds. We only prove items (i) and (iv),
    since items (ii) and (iii) follow in the same way.\\
  (i) Consider $i\le 4$, with $i$ and $j$ both even, and assume that
    \[
q_{i,j}(s_1) = \theta A e^{-2s_1},
    \]
    for some $\theta = \pm 1$. Since Lemma \ref{lemR} implies that
    $|R_{i,j}(s)|\le Ce^{-2s}$,
   by definition \eqref{defvij},
  we write from  item (i) of Proposition \ref{propdyn}, 
 \begin{equation}\label{un}
   \theta q_{i,j}'(s_1)
   \ge (1-\frac i2)A e^{-2s_1} - C_iAs_1e^{-3s_1}-Ce^{-2s_1}
\ge \frac{1-i}2 A e^{-2s_1},
\end{equation}
on the one hand, taking $A$ large enough,
then $s_0$ large enough. On the other hand, we have
\begin{equation}\label{deux}
\frac d{ds}
 A{e^{-2s}}_{|s=s_1}= - 2Ae^{-2s_1}.
\end{equation}
Since $\frac {1-i}2 \ge \frac {1-4}2 = -\frac 32 \ge -2$, the
conclusion follows.\\
(iv)
Let us assume
that $\|q_-(s_1)\|_{L^2_\rho} = A^2s_1^2
e^{-3s_1}$. Using Lemma \ref{lemR}, we see that
$\|R_-(s)\|_{L^2_\rho}\le Ce^{-3s}$, by definition \eqref{decomp}. Using 
item (ii) in Proposition \ref{propdyn}, it follows that
\begin{equation}\label{trois}
  \frac d{ds}\|q_-(s_1)\|_{L^2_\rho}
  \le -3 A^2s_1^2e^{-3s_1}+CAs_1e^{-3s_1}+Ce^{-3s_1}
  \le  -3 A^2s_1^2e^{-3s_1}+A^2s_1e^{-3s_1}
\end{equation}
on the one hand, taking $A$ large enough,
then $s_0$ large enough. On the other hand, we compute
\begin{equation}\label{quatre}
\frac d{ds} {A^2s^2 e^{-3s}}_{|s=s_1}= A^2  e^{-3s_1}(2s_1-3s_1^2),
\end{equation}
and the conclusion follows.
This concludes the proof of Proposition \ref{propflow}.
\end{proof}

\subsection{Details for the initialization} \label{sectechinit}
We prove Proposition \ref{propinit} here.

\begin{proof}[Proof of Proposition \ref{propinit}]
We will be using the notation  $d=(d_{0,0}, d_{2,0}, d_{4,0},
d_{4,2}, d_{6,0})$ for simplicity. Let us consider $A\ge 1$ and $s_0\ge 1$. The
first item (i) will be proved for any $d\in[-2,2]^5$.
The set $\q D$ will be introduced while proving item
(ii). Since the set $\q D$ we intend to construct
will be in  $[-2,2]^5$,
there will be no need to revisit the
proof of item (i) afterwards.
Note that all the expansions given below are valid
in $L^r_\rho$ for any $r\ge 2$.\\
(i) By Definition \eqref{defS} of $S(y)$ and $\bar S(y)$ and
Definition \eqref{defD} of $D$, together with \eqref{defNi}
and \eqref{tot2},  we write for $s_0$ large enough,
\begin{align}
\frac{Ae^{-2s_0}|S(y)|}D\le &\;\frac{CAe^{-2s_0}(1+|y|^4)}D\le
CAe^{-\frac{s_0}2},\nonumber\\
\frac{As_0^2e^{-3s_0}|\bar S(y)|}D\le & \;
CAs_0^2e^{-3s_0}\frac{1+|y|^6}{e^{-2s_0}+e^{-2s_0}|y|^6} \le C As_0^2e^{-s_0}. \label{boundS}
\end{align}
Using item (i) of Lemma \ref{lemphi}, the conclusion follows for $s_0$
large enough.\\
(ii) Since $D\ge p-1>0$ by definition \eqref{defD} of $D$, using the expressions \eqref{defw0} and \eqref{goodprof} of $w_0(y,s_0)$
and $\varphi$, together with item (ii) of Lemma \ref{lemphi} and \eqref{boundS}, we write
\begin{equation}\label{boundw0}
|w_0(y,s_0)|^{p-1} \le |\varphi(y,s_0)|^{p-1} +\frac {p-1}
{\kappa D}(Ae^{-2s_0}S(y) + As_0^2e^{-3s_0}\bar S(y))
\le \frac 1{p-1} +Ce^{-\frac{s_0}3},
\end{equation}
for $s_0$ large enough, and the bound on $\|w_0(s_0)\|_{L^\infty}$ follows.\\
Taking the logarithm then the gradient of $w_0(y,s_0)$ \eqref{defw0},
we write
\begin{equation}\label{gradw0}
  |\nabla w_0(y,s_0)|\le \frac{|w_0(y,s_0)|}{p-1}\left[ \frac{|\nabla {\bar
        E|}}{\bar E}+ \frac{|\nabla D|}D\right]
\end{equation}
where
\[
  \bar E = E + \frac {p-1}{\kappa D} (Ae^{-2s_0}S(y) +As_0^2 e^{-3s_0}\bar S(y)).
\]
Using \eqref{boundS}, \eqref{defE}, \eqref{lowP} and \eqref{lowQ}, we write
\begin{equation}\label{lowEb}
 \bar E\ge \frac {E_0}C \mbox{ where } E_0 =1+ e^{-s_0}(y_1^2+y_2^2)+e^{-2s_0}(y_1^4+y_2^4).
\end{equation}
Then, we write
\[
\nabla\left[\frac {S(y)}D\right]= \frac{\nabla S(y)}D-\frac{S(y)\nabla D}{D^2},
\]
hence, by definition \eqref{defS} of $S(y)$, we have
\[
  Ae^{-2s_0}\left|\nabla\left[\frac {S(y)}D\right]\right|
  \le  CAe^{-2s_0}\frac{(1+|y|^3)}D+Ae^{-2s_0}\frac{|S(y)|}D
  \frac{|\nabla D|}D
  \le CA e^{-\frac 23 s_0},
\]
thanks to \eqref{boundS} and \eqref{boundD}, together with the 
technique we used for \eqref{tot2}. Similarly, we derive that
\[
 As_0^2e^{-3s_0}\left|\nabla\left[\frac {\bar S(y)}D\right]\right|
  \le As_0^2 e^{-\frac 76 s_0}.
\]
Therefore, using \eqref{gradw0},
\eqref{lowEb} together with \eqref{boundED}, we write
\begin{equation}\label{gradEb}
  \frac{|\nabla \bar E|}{\bar E}
  \le C \frac{|\nabla E|}{E_0} + CA e^{-\frac 23 s_0} +C As_0^2 e^{-\frac 76 s_0}
  \le Ce^{-\frac {s_0}3}.
\end{equation}
Using \eqref{gradw0}, \eqref{boundw0}, \eqref{gradEb} together with
\eqref{boundD}, we obtain the following bound
\[
|\nabla w_0(y,s_0)|\le Ce^{-\frac {s_0}6}.
\]
Arguing as for \eqref{dlphi}, we show the following:
\begin{align}
             &w_0(y,s_0)=\kappa +e^{-s_0}\left(\frac\kappa{p-1} P(y)-y_1^2y_2^2\right)\nonumber\\
             & +e^{-2s_0}\left(\frac\kappa{p-1}Q(y) +AS(y) +\frac{\kappa(2-p)}{2(p-1)^2}P(y)^2
- \frac{P(y)}{p-1} y_1^2y_2^2-\delta y_1^6-\delta y_2^6+\frac
               p{2\kappa}y_1^4y_2^4\right)\nonumber\\
              &
                +As_0^2e^{-3s_0}\bar S(y) +O(Ae^{-3s_0}),\nonumber
\end{align}
uniformly for $d\in[-2,2]^5$.
Using again the expansion \eqref{dlphi} of
$\varphi$ together with definition \eqref{q0} of $q(y,s_0)$, we derive
that
\begin{equation}\label{expq}
  q(y,s_0) = Ae^{-2s_0}S(y) +As_0^2e^{-3s_0}\bar S(y)
    +O(Ae^{-3s_0})\mbox{ as }s_0\to \infty.
\end{equation}
By definition \eqref{defS} of $S(y)$ and $\bar S(y)$, together with
the definition \eqref{defvij} of the projections, we clearly see that for all  $d\in[-2,2]^4$
and $(i,j)\in I_0\equiv\{(0,0), (2,0), (4,0), (4,2)\}$,
\[
  q_{i,j}(s_0) =Ae^{-2s_0} d_{i,j} +O(Ae^{-3s_0})
  \mbox{ and }q_{6,0}(s_0) = As_0^2e^{-3s_0} d_{6,0} +O(Ae^{-3s_0})
\]
(please note that this identity holds after differentiation in $d$).
We also have $\|q_-(s_0)\|_{L^2_\rho}=O(Ae^{-3s_0})$ and
\[
  q_{i,j}(s_0) = O(Ae^{-3s_0})
  \mbox{ whenever } (i,j) \mbox{ and }(i,i-j) \mbox{ are not in }
  I_0\cup \{(6,0)\}.
\]
Recalling estimate \eqref{boundw0} and the definition \eqref{q0} of $q(y,s_0)$,
we see that
this clearly gives the existence of 
$\q D\subset [-2,2]^5$ such that
$q(s_0)\in V_A(s_0)$, \eqref{qs0} holds and
$|q_{6,2}(s_0)|+\|q_-(s_0)\|_{L^2_\rho}\le CAe^{-3s_0}$, whenever
$d\in \q D$,
with the function in \eqref{121} one-to-one,
provided that $s_0$ is large enough.
This concludes the proof of Proposition \ref{propinit}.

\end{proof}

\subsection
{Gradient estimate in the shrinking set}\label{secgrad}
This section is devoted to the proof of Proposition \ref{propgrad},
thanks to the Liouville theorem recalled in Proposition \ref{propliou}.

\begin{proof}[Proof of Proposition \ref{propgrad}]
Consider $A\ge 1$ and ${\delta_0}>0$. Proceeding by contradiction, we may
exhibit a sequence $q_n$ of solutions to equation  \eqref{eqq}
defined for all  $(y,s)\in \m R^2\times [s_{0,n}, s_{1,n}]$ for some
$s_{1,n}\ge s_{0,n}\ge n$ such that $q_n(s_{0,n})$ is given by \eqref{q0} for some parameter  $(d_{0,0,n},
  d_{2,0,n}, d_{4,0,n}, d_{4,2,n}, d_{6,0,n})\in \q D_n$ where $\q D_n=\q
  D(A,s_{0,n})$ is defined in Proposition \eqref{propinit}, with $q_n(s)\in
  V_A(s)$ for all $s\in [s_{0,n},s_{1,n}]$, and
  \begin{equation}\label{inf}
    \|\nabla q_n(s_{2,n})+\nabla \varphi(s_{2,n})\|_{L^\infty}>\delta_0 \mbox{ for some }
    s_{2,n}\in [s_{0,n}, s_{1,n}].
  \end{equation}
  Note that
  \begin{equation}\label{lims2n}
s_{2,n}\to \infty\mbox{ as } n\to \infty.
  \end{equation}
  We claim that it is enough to prove that
  \begin{equation}\label{goal}
s_{2,n}-s_{0,n}\to \infty \mbox{ as } n\to \infty
  \end{equation}
  in order to conclude. Indeed, if \eqref{goal} holds, then,
  introducing
  \begin{equation}\label{defwn}
w_n(y,s) = q_n(y,s+s_{2,n})+\varphi(y,s+s_{2,n}),
  \end{equation}
  we see from \eqref{defq} that $w_n$ is a solution of equation
  \eqref{eqw} defined for all $(y,s)\in \m R^2 \times
  [s_{0,n}-s_{2,n},0]$. Since $q_n(s) \in V_A(s)$ for all $s\in
  [s_{0,n},s_{2,n}]$, it follows by Definition \ref{defvas} that
$\|q_n(s)+\varphi(s)\|_{L^\infty}\le 2 \kappa$.
  Using this together with \eqref{inf}
   we see that for $n$ large enough, we have:
  \begin{align}
    \forall s\in [s_{0,n}-s_{2,n},0],\;\; \|w_n(s)\|_{L^\infty}&\le
2 \kappa,
\label{prop1}\\
   \|\nabla w_n(0)\|_{L^\infty} &\ge \delta_0.\label{prop2}
  \end{align}
 Applying a classical parabolic regularity technique to equation
 \eqref{eqw}, we obtain the following estimate:
 \begin{lem}[Parabolic regularity for equation \eqref{eqw}]\label{propara}
   For all $n\in \m N$, it holds that
 \begin{equation}\label{regul}
\|w_n\|_{C^{2,1,\alpha}(\m R^2\times(s_{0,n}-s_{2,n}+1,0))}\le C_0
\end{equation}
for some $C_0>0$, where $C^{2,1,\alpha}$ stands for the set of functions of space and time,
with 2 space derivatives and 1 time derivative which are $\alpha$-H\"older continuous
in both variables.
\end{lem}
\begin{proof}
Since this estimate is classical, we leave its justification to
Section \ref{appreg} in the Appendix.
\end{proof}
Recalling that $s_{0,n}-s_{2,n}\to -\infty$ from \eqref{goal}, we may use
the compactness provided by this lemma, combined to a diagonal
process, in order to extract a subsequence (still denoted by
$w_n$) such that $w_n \to w$ in $C^{2,1}$ of any compact subset of $\m
R^2\times(-\infty,0]$. Since $w_n$ is a solution of equation
\eqref{eqw}, the same holds for $w$. From Properties \eqref{prop1} and
\eqref{prop2}, it follows that
\begin{align*}
  \forall s\le 0,\;\; \|w(s)\|_{L^\infty}&\le 2\kappa,\\
   \|\nabla w(0)\|_{L^\infty} &\ge \delta_0.
\end{align*}
Using Proposition \ref{propliou}, we see that this is a
contradiction. Thus, it remains to prove \eqref{goal} in order to conclude.

\medskip

\textit{Proof of \eqref{goal}}: Since $w_n$ satisfies equation
\eqref{eqw}, differentiating this equation in space, we obtain the
following vector-valued equation on $\nabla w_n$, for all $(y,s)\in \m
R^2 \times [s_{0,n}-s_{2,n},0]$:
\[
\partial_s \nabla w_n = (\q L-\frac 32-\frac 1{p-1}) \nabla w_n+p|w_n|^{p-1}\nabla w_n,
\]
where the operator $\q L$ is defined in \eqref{defR}.\\
Using a Duhamel formulation based on the kernel \eqref{kernel},
together with the $L^\infty$ bound in \eqref{prop1}, we see
that for all $s\in  [ s_{0,n}-s_{2,n},0]$:
\begin{align*}
  \|\nabla w_n(s)\|_{L^\infty} \le
 & e^{-(\frac 1 2+\frac 1{p-1})(s+s_{2,n}-s_{0,n})}  \|\nabla w_n(s_{0,n}-s_{2,n})\|_{L^\infty}\\
&  +p(2\kappa)^{p-1}\int_{s_{0,n}-s_{2,n}}^0\|\nabla w_n(s')\|_{L^\infty} ds'.
\end{align*}
Since $\kappa^{p-1}=\frac 1{p-1}$ by definition \eqref{conv}
using Gronwall's lemma, we see that
\begin{equation}\label{gron}
  \|\nabla w_n(0)\|_{L^\infty}\le
  e^{(\frac{p 2^{p-1}-1}{p-1}-\frac 12)(s_{2,n}-s_{0,n})}  \|\nabla w_n(s_{0,n}-s_{2,n})\|_{L^\infty},
\end{equation}
on the one hand. On the other hand, recalling that by hypothesis, $q_n(s_{0,n})$ is
given by \eqref{q0} for some parameters
$(d_{0,0,n}, d_{2,0,n}, d_{4,0,n}, d_{4,2,n}, d_{6,0,n})\in \q D_n$  where
$\q D_n=\q  D(A,s_{0,n})$ is defined in Proposition \eqref{propinit},
we see from that proposition and definition \eqref{defwn} of $w_n$
that
\[
  \|\nabla w_n(s_{0,n}-s_{2,n})\|_{L^\infty}
= \|\nabla q_n(s_{0,n})+\nabla \varphi(s_{0,n}) \|_{L^\infty}
  \le Ce^{-\frac{s_{0,n}}6}.
\]
Using this together with \eqref{prop2} and \eqref{gron}, we write
\begin{align*}
  0<\delta_0 \le
  \|\nabla w_n(0)\|_{L^\infty}&\le
                                 e^{(\frac{p 2^{p-1}-1}{p-1}-\frac 12)(s_{2,n}-s_{0,n})}
                                 \|\nabla w_n(s_{0,n}-s_{2,n})\|_{L^\infty} \\
&\le Ce^{(\frac{p 2^{p-1}-1}{p-1}-\frac 12)(s_{2,n}-s_{0,n})} e^{-\frac{s_{0,n}}6}.
 \end{align*}
Since $s_{0,n} \ge n$ and $\frac{p 2^{p-1}-1}{p-1}-\frac 12>\frac 12>0$,
it follows that $s_{2,n}-s_{0,n}\to
\infty$, and \eqref{goal} holds. Since we have already showed that
\eqref{goal} yields a contradiction, this concludes the proof of
Proposition \ref{propgrad}.
 \end{proof}

\subsection{Size of the solution in the 3 regions}\label{secregions}
In this section, we prove Lemma \ref{lemregions}.

\begin{proof}[Proof of Lemma \ref{lemregions}]
  Consider
   $A\ge 1$,
$s_0\ge \sszeroun(A)$
  and $d=(d_{0,0}, d_{2,0},
  d_{4,0}, d_{4,2}, d_{6,0})\in \q D(A,s_0)$, where
 $\sszeroun(A)$
and $\q
  D(A,s_0)$ are defined in Proposition \ref{propinit}. Consider then
  $a\in \m R^2$ and introduce $y=ae^{\frac{s_0}2}$. By definition
  \eqref{defw0} and \eqref{goodprof} of initial data $w_0(y,s_0)$ and
  the profile $\varphi$, arguing as for \eqref{total}, we may improve
  that estimate and write
  \begin{equation}\label{1D}
\left|w_0(y,s_0)^{p-1} - \frac 1D \right|\le
C\left\{e^{-s_0}+\frac{N_2+N_3}D+ \frac{Ae^{-2s_0}S(y)}{D^2}
+\frac{As_0^2e^{-3s_0}\bar S(y)}{D^2}\right\}
\le C e^{-\frac{s_0}3}, 
  \end{equation}
  where $D$, $N_2$, $N_3$, $S(y)$ and $\bar S(y)$ are defined in \eqref{defD},
  \eqref{defNi} and \eqref{defS}, and where we have used the bounds \eqref{tot2},
  \eqref{tot3} and \eqref{boundS}.\\
By definitions \eqref{defD} and \eqref{defG0} of $D$ and
$G_0(a)$, we see that
\[
D = p-1+\frac{(p-1)^2}\kappa\left( e^{-s_0}y_1^2y_2^2 + \delta
  e^{-2s_0}(y_1^6+y_2^6)\right)
= (p-1)[1 + e^{s_0}G_0(a)]. 
\]
Consider now 2 nonnegative numbers $m$ and $M$
such that $0<m\le 1 \le M$.
If $a\in \q R_1$ (resp. $\q R_2$, resp. $\q R_3$) defined in \eqref{defqR}, we see by
definition \eqref{defqR} that $(p-1)(1+M)\le D$ (resp.
$(p-1)(1+m)\le D \le (p-1)(1+M)$, resp. $D \le (p-1)(1+m)$). Since
$0\le w_0(y,s_0)\le \kappa+Ce^{-\frac{s_0}3}$ by Proposition \ref{propinit}, combining
this with \eqref{1D} concludes the proof of  Lemma \ref{lemregions}.

\end{proof}

\subsection{Details for the control of $w_a(y,s)$ for $a$ in Region
  $\q R_3$}
\label{secwas}

In this section, we prove Lemmas \ref{propdeco0}, \ref{cordyn} and \ref{propdeco0'}.

\begin{proof}[Proof of  Lemma \ref{propdeco0}]
  Consider $A\ge 1$ and
$s_0\ge \sszeroun(A)$,
together with the parameter
  $d=(d_{0,0}, d_{2,0}, d_{4,0}, d_{4,2}, d_{6,0})\in \q D(A,s_0)$, where
$\sszeroun$
 and $\q D$ are defined in Proposition
  \ref{propinit}. Consider also $w_0(y,s_0)$ defined in
  \eqref{defw0}. Recalling that $\q D\subset[-2,2]^5$, we may write
  the following Taylor expansion:
  \begin{equation}\label{taylor}
    \left| w_0(y,s)-
    \kappa\left[1-\frac X{p-1}-\frac{e^{-s_0}P}{p-1}\right] \right|
    \le C\left\{  I
    +J^2
    +\m 1_{\{1<p<\frac 32\}}J^{\frac 1{p-1}} \right\}
  \end{equation}
  where
\begin{align*}
  X(y_1,y_2,s_0)=&\;\frac{p-1}\kappa\left[e^{-s_0}y_1^2y_2^2
     +\delta e^{-2s_0}(y_1^6+y_2^6)\right],\\
I= & \; Xe^{-s_0}|P|+e^{-2s_0}|Q|(1+X)
                                                +X^2(1+e^{-s_0}|P|+e^{-2s_0}|Q|)\\
  &+Ae^{-2s_0}(1+|y|^4)+As_0^2e^{-3s_0}(1+|y|^6),\\
 J=&\;X+e^{-s_0}|P|+I,
\end{align*}
  the polynomials $P$ and $Q$ are given in \eqref{defP} and
  \eqref{defQ} and have respectively 2 and 4 as degree (see
  \eqref{lowP} and \eqref{expQ}), and the constant $\delta\ge 1$ was
  already fixed large enough at the beginning of Section
  \ref{secformu}. 

  \medskip

  Consider now some $m\in (0,1)$ and $a\in \q R_3$ defined in
  \eqref{defqR}, with $a$ decomposed as in \eqref{defa}, for some
  $L\ge K\ge 0$, with $L+K\ge A$.
Given some $r\ge 2$, we may use the relation \eqref{waw0} together
with \eqref{taylor} to derive an expansion for $w_a(y,s_0)$, showing
error terms bounded by small terms in scales of $1/(L+K)$ and
$e^{-s_0}$.
In particular, the following expansions are useful:
\begin{align*}
  &P(y_1+K,y_2+L)= \frac{2(p-1)}\kappa
  \left[K^2+L^2+2+2Kh_1(y_1)+2Lh_1(y_2)+h_2(y_1)+h_2(y_2)\right],\\
  &X(y_1+K,y_2+L,s_0)=\;e^{-s_0}
  \left[y_1^2y_2^2+2Ly_1^2y_2 +2Ky_1y_2^2 +L^2y_1^2+4KLy_1y_2 +K^2y_2^2\right.\\
  &\left.+2KL^2y_1+2K^2Ly_2+K^2L^2\right]
+\delta e^{-2s_0}(K^6+L^6)+O((K^5+L^5)e^{-2s_0})
\end{align*}
in $L^r_\rho(\m R^2)$.
This latter estimate can be easily written in the
Hermite polynomials basis \eqref{defhj}.
Since by definition \eqref{defiota} of $\iota$ and \eqref{boundiota}, it follows that   
\[
  \iota \le \frac{m\kappa}{p-1}\le \frac \kappa{p-1},\;\;
  K+L\le 2 \iota^{\frac 16}e^{\frac{s_0}3}
  \mbox{ and } KL \le \sqrt\iota e^{\frac{s_0}2}, 
\]
one can easily bound all the error terms by $O\left(\frac\iota
  A\right)$ and $O\left(\iota^2\right)$, as required by the statement
of the lemma.
This concludes the proof of Lemma \ref{propdeco0}.
\end{proof}

\begin{proof}[Proof of Lemma \ref{cordyn}]
  Take
   $A\ge 1$ and
$s_0\ge \max[\sszeroun(A),\squatre(A)]$
where
$\sszeroun$
and $\squatre$ are defined in Proposition \ref{propinit} and Lemma
  \ref{propdeco0}. Consider then
 $(d_{0,0},d_{2,0}, d_{4,0}, d_{4,2}, d_{6,0})\in \q D$ defined in Proposition
\ref{propinit} and initial data $w(y,s_0)$ defined in
\eqref{defw0}. Consider also some $\eta^*\le \frac \kappa{p-1}$,
$0<m <\frac{\eta^*(p-1)}\kappa\le 1$
and $a\in \q R_3$ \eqref{defqR} given
by \eqref{defa} for some $L\ge K\ge 0$ such that $L+K\ge A$.
In this case, Proposition \ref{propinit} applies, and so does Lemma
\ref{propdeco0}. In particular, the expansion given there holds for $w_a(y,s_0)$.

\medskip

(i) Note that condition \eqref{condeta*} holds from the choice of
$\eta^*$ and $m$. Therefore, using \eqref{boundiota}, we see that
$\iota \le \frac{m\kappa}{p-1} \le \eta^*$. By definition \eqref{defs*}
of $s^*$, it follows that $ s^*= s_0+\log\frac{\eta^*}{\iota} \ge s_0$.

\medskip

(ii) Assume now that
\begin{equation}\label{walinf}
\forall s\in [s_0,s_1],\;\; \|w_a(s)\|_{L^\infty}\le 2 \kappa,
\end{equation}
for some $s_1\ge s_0$.
Introducing
\begin{equation}\label{defva}
  v_a = w_a-\kappa
  \mbox{ and }\bar s = \min(s^*,s_1),
\end{equation}
we work in the following in the interval $[s_0,\bar s]$, and proceed in 3 steps
in order to give the proof:\\  
- In Step 1, we write an equation satisfied by $v_a$ and project it on
the various components $v_{a,i,j}$ defined in \eqref{defvij}.\\
- In Step 2, we integrate those equations.\\
- In Step 3, we collect the previous information to conclude the proof.

    \bigskip

    \textbf{Step 1: Dynamics for $v_a$}

    Since $w_a$ satisfies equation \eqref{eqw}, by definition \eqref{defva}, it follows
  that $v_a$ satisfies the following equation:
\begin{equation}\label{eqva}
\forall s\in [s_0,\bar s],\;\;\partial_s v_a = \q L v_a + \bar B(v_a),
\end{equation}
where the linear operator $\q L$ is introduced in \eqref{defR} and
\begin{equation}\label{bva}
\bar B(v_a) = |\kappa+v_a|^{p-1}(\kappa+v_a) - \kappa^p - p\kappa^{p-1}v_a.
\end{equation}
In this step, we project equation
    \eqref{eqva} in order to write differential inequalities satisfied
    by the various components $v_{a,i,j}$ defined in
    \eqref{defvij} as well as $\bar P(v_a)$ where $\bar P$ is the
    $L^2_\rho$ orthogonal projector on
 \begin{equation}\label{Ebar}
\bar E={\rm span} \{h_ih_j\;|\; (i,j) \not \in \{(0,0), (1,0),(0,1),(2,0),
(1,1), (0,2), (1,2), (2,1), (2,2)\}\;\},
\end{equation}
    the orthogonal supplement of the directions appearing in the
    expansion of Lemma \ref{propdeco0}. This is our statement:
    \begin{lem}[Projections of equation \eqref{eqva}] \label{lemproj}
      Under the hypotheses of Lemma \ref{cordyn},
      for all $s\in [s_0, \bar s]$,
      for all $i\in \m N$ and $j=0,\dots,i$, it holds that
\begin{equation}\label{eqvaij}
\left|v_{a,i,j}'(s)-\left(1-\frac i2\right)v_{a,i,j}(s)\right|\le C(i)\|v_a(s)\|_{L^2_\rho}^2.
\end{equation}
In addition,
\begin{align}
  \frac d{ds}\|\bar P(v_a(s))\|_{L^2_\rho}\le&
 -\frac 12 \|\bar P(v_a(s))\|_{L^2_\rho}\label{eqpbar}\\
 &+ \m 1_{\{s_0\le s\le  s_0+2\}}\Chuit\|v_a(s_0)\|_{L^4_\rho}^2
 +\m 1_{\{s_0+2\le s \le \bar s\}}\Chuit\|v_a(s-2)\|_{L^2_\rho}^2,\nonumber
\end{align}
for some universal constant $\Chuit$,  where $\bar s$ is defined in
\eqref{defva} and the projector $\bar P$ is defined right before
\eqref{Ebar}. 
\end{lem}
\begin{proof}
  Note first from \eqref{walinf}, \eqref{defva} and \eqref{bva} that
\begin{equation}\label{bvabound}
 \forall s\in [s_0,\bar s],
  \;\;|\bar B(v_a)|\le C|v_a|^2.
\end{equation}
This way, the proof follows as for Proposition \ref{propdyn} proved in Section \ref{sectechdyn}.
More precisely, identity \eqref{eqvaij} follows from equation
\eqref{eqva} and the quadratic estimate \eqref{bvabound} exactly as for
item (i) of that proposition. As for \eqref{eqpbar}, arguing as for
item (ii) of the same proposition, we write
  \[
    \frac d{ds}\|\bar P(v_a(s))\|_{L^2_\rho}\le
    -\frac 12 \|\bar P(v_a(s))\|_{L^2_\rho} +
    C\|v_a(s)^2\|_{L^2_\rho},
   \]
   since $\lambda = -\frac 12$ is the largest eigenvalue of $\q L$
   corresponding to the components spanning $\bar E$ \eqref{Ebar}, the
   image of the projector $\bar P$
   (see \eqref{Lhihj}). Note that
   $\lambda =-\frac 12$ corresponds to the eigenfunctions $h_3h_0$
   and $h_0h_3$.\\
   The question then reduces to the control of
   $\|v_a(s)^2\|_{L^2_\rho}=\|v_a(s)\|_{L^4_\rho}^2$.
   Arguing as we did  in Section \ref{sectechdelay} for the proof of Proposition
   \ref{propdelay}, we can apply here items (i) and (ii) of Lemma
   \ref{lemVel} (with a delay time equal to $2$). Indeed, $v_a(s)$, $v_a(s_0)$ and $\nabla
   v_a(s_0)$ are in $L^\infty(\m R^2)$, thanks to \eqref{walinf} and
   Proposition \ref{propinit}, through the transformations
   \eqref{defva} and \eqref{waw0}.
  This concludes the proof of Lemma \ref{lemproj}.
\end{proof}
   
\bigskip

\textbf{Step 2: Integration of the equations of Lemma \ref{lemproj}}

We claim the following:
 \begin{lem}[Integration of equation \eqref{eqva}]\label{lemdyn} There
   exist $\Msix$,
$\bAsixs\ge 1$
 and $\etazero>0$
such that  under the hypotheses of Lemma \ref{cordyn},
 if $A\ge \bAsixs$
and $\eta^*\le \etazero$,
 then for all $s\in [s_0,\bar s]$, it holds that   
 \begin{align*}
    \forall i=0,\dots,4\mbox{ and }j=0,\dots,i,\;\;
    |v_{a,i,j}(s)- e^{(1-\frac i2)\tau}v_{a,i,j}(s_0)|\le
    &\;\Msix(\eta^*+A^{-2})e^\tau \iota,\\
  \|\bar P(v_a(s))\|_{L^2_\rho} \le
  e^{-\frac \tau 2}  \|\bar P(v_a(s_0))\|_{L^2_\rho}
  +  &\;\Msix(\eta^*+A^{-2})e^\tau \iota,
 \end{align*}
 provided that $A$ and $s_0$ are large enough and $\eta^*$ is small enough,
  where $\bar s$, $\iota$ and $\bar P$ are defined in \eqref{defva},
  \eqref{defiota} and right before \eqref{Ebar}, with $\tau=s-s_0$.
\end{lem}
\begin{proof}
  This integration is at the heart of our argument, as it was already
  the case in our previous work \cite{MZimrn21} dedicated to the
  classification of the possible behaviors near non isolated blow-up
  points for equation \eqref{equ} (see in particular Section 4.4 in
  that paper). Looking in that paper
  is certainly convincing for the expert reader. To be nice to all
  readers, we summarize the integration argument below.

  \medskip

Consider $\Msix>0$ to be fixed later.
For $s_0$ large enough, we note that both identities in Lemma \ref{lemdyn} are true at
$s=s_0$. Therefore, we may proceed by contradiction and assume that one of these
identities doesn't hold at some time in the interval 
$[s_0, \bar s]$ where $\bar s$ is defined in \eqref{defva}. 
If $\ssix$ is the infimum of such times, then we 
see from continuity that both identities hold on the interval $[s_0,\ssix]$ and that one
of them has an equality case at $s=\ssix$. In the following, we will
prove that no equality case occurs, yielding a contradiction.

\medskip

Since both identities hold for all $s\in [s_0,\ssix]$,
using the estimates at initial time $s=s_0$ given in Lemma
\ref{propdeco0} and using the definition \eqref{defva} of $v_a$,
we derive the following bounds for $s_0$ large enough and for
all $s\in[s_0,\ssix]$ and $r\in \{2,4\}$:
\begin{align}
\|v_a(s_0)\|_{L^r_\rho}\le &\;\Msixp (\iota+J)
 \mbox{ where }  J=e^{-s_0}(K^2+L^2),\label{defJ}\\
  \|v_a(s)\|_{L^2_\rho}\le &  \;\Msixp (e^\tau \iota+J)
+\Msixp\Msix(\eta^*+A^{-2}) e^\tau \iota,\nonumber 
 \end{align}
 for some universal constant $\Msixp>0$,
 where $\tau=s-s_0$ hereafter.
Taking $\eta^*>0$ small enough and $A$ large enough, we write
for all $s\in[s_0,\ssix]$,
\begin{equation}\label{vas}
  \|v_a(s)\|_{L^2_\rho}\le   2\Msixp (e^\tau \iota+J).  
\end{equation}
Using Lemma \ref{lemproj}
together with \eqref{vas}, we write for $s_0$
large enough, for all  $i=0,\dots,4$,
$j=0,\dots,i$ and $s\in[s_0,\ssix]$,
\begin{align*}
  \left|v_{a,i,j}'(s)-\left(1-\frac i2\right)v_{a,i,j}(s)\right|
  \le &\;\Chuitp {\Msixp}^2 (e^{2\tau} \iota^2+J^2),\\
\frac d{ds}\|\bar P(v_a(s))\|_{L^2_\rho}\le
 -\frac 12 \|\bar P(v_a(s))\|_{L^2_\rho}
+&\;\Chuitp{\Msixp}^2(e^{2\tau}\iota^2+J^2),
\end{align*}
for some universal constant $\Chuitp>0$.\\
Integrating the first equation, we see that
for all $s\in[s_0,\ssix]$, 
\begin{equation}\label{i}
  |v_{a,i,j}(s)- e^{(1-\frac i2)\tau} v_{a,i,j}(s_0)|
  \le   \Chuitp {\Msixp}^2 (e^{2\tau} \iota^2+2e^\tau J^2).                                     
 \end{equation}
Integrating the second inequality, we see that  for all
$s\in[s_0,\ssix]$,
\begin{align}
  \|\bar P(v_a(s))\|_{L^2_\rho}
  \le \;& e^{-\frac \tau 2}\|\bar P(v_a(s_0))\|_{L^2_\rho}
        +\Chuitp{\Msixp}^2 \left(\frac 25 e^{2\tau}\iota^2+2J^2\right).\label{pb}
 \end{align}
Recalling that $\tau= s-s_0 \le \ssix - s_0 \le \bar s -s_0\le s^*-s_0$, we see by
definition \eqref{defs*} that
\begin{equation}\label{eti}
  e^\tau \iota \le \eta^*.
  \end{equation}
Moreover, recalling that $K+L\ge A$ and $\delta\ge 1$ (see the
beginning of Section \ref{secformu}), we write by definitions
\eqref{defJ} and \eqref{defiota} of $J$ and $\iota$:
\[
  J^2 \le e^{-2s_0}\frac{(K+L)^2}{A^2}(K^2+L^2)^2
  \le  e^{-2s_0}\frac{16}{A^2}(K^6+L^6) \le \frac{16\iota}{\delta A^2}
  \le \frac{16\iota}{A^2}.
\]
  Using this together with \eqref{eti}, \eqref{i} and \eqref{pb}, we see that
for $s_0$ large enough, we have  for all $s\in[s_0,\ssix]$,
 \begin{align*}
    \forall i=0,\dots,4\mbox{ and }j=0,\dots,i,\;\;
    |v_{a,i,j}(s)- e^{(1-\frac i2)\tau}v_{a,i,j}(s_0)|\le
 & \;32\Chuitp{\Msixp}^2\left(\eta^*+ A^{-2}\right)e^\tau \iota,\\
   \|\bar P(v_a(s))\|_{L^2_\rho} \le
  e^{-\frac \tau 2}\|\bar P(v_a(s_0))\|_{L^2_\rho}
   \;+&\;32\Chuitp{\Msixp}^2\left(\eta^*+A^{-2}\right)e^\tau \iota.
 \end{align*}
  Fixing
   \[
\Msix = 33\Chuitp{\Msixp}^2, 
   \]
   we see that no equality case occurs in both identities shown in
   Lemma \ref{lemdyn}. A contradiction follows from the beginning of
   the proof. This concludes the proof of Lemma \ref{lemdyn}.
 \end{proof}

\bigskip

\textbf{Step 3: Conclusion of the proof of Lemma \ref{cordyn}}

Recalling the transformation \eqref{defva}, then using Lemma
\ref{lemdyn} together with Lemma \ref{propdeco0} and the
Cauchy-Schwarz inequality, we see that for all
$s\in [s_0,\bar s]$,
\[
  \left\|w_a(\cdot,s) -\left(\kappa-e^\tau \iota\right)\right\|_{L^2_\rho}
  \le \uMsix \left\{ \Msix (\eta^*+A^{-2}) e^\tau \iota
      + e^\tau\left(\frac \iota A + \iota^2\right)
       +(L^2+K^2) e^{-s_0}\right\},
\]
for some universal constant
$\uMsix>0$.
Using the bounds \eqref{boundJ} and \eqref{boundiota}, then recalling
that $\delta\ge 1$ from the beginning of Section \ref{secformu}, we
see that
\[
  e^{-s_0}(L^2+K^2) \le 2e^{-\frac{s_0}3}
  \left(\frac \kappa{p-1}\right)^{\frac 13}.
\]
This concludes the proof of Lemma \ref{cordyn}.
\end{proof}

\medskip

\begin{proof}[Proof of Lemma \ref{propdeco0'}]
Since $w_0(y,\sigma) = \varphi(y,\sigma)+q(y,\sigma)$ by the relation
\eqref{defq}, proceeding as in \eqref{waw0}, we may introduce
\begin{align}
  \varphi_a(y,\sigma)=&\; \varphi(y+ae^{\frac \sigma 2},\sigma)
                        =\varphi(y_1+K',y_2+L',\sigma), \nonumber\\
q_a(y,\sigma) = &\; q(y+ae^{\frac \sigma 2}, \sigma)
                    = q(y_1+K', y_2+L',\sigma)\label{defqa}
\end{align}
(use \eqref{defk'l'}). This way, we write $w_a(y,\sigma) =
\varphi_a(y,\sigma) + q_a(y,\sigma)$, and the proof of \eqref{expwas}
follows by adding the expansions of $\varphi_a(y,\sigma)$ and
$q_a(y,\sigma)$, performed in 2 steps.

\medskip

\textbf{Step 1: The expansion of $\varphi_a(y,\sigma)$}

We claim that the expansion of $\varphi_a(y,\sigma)$ follows
from Lemma \ref{propdeco0}. Indeed, the input in that lemma is initial
data $w_0(y,s_0)$ \eqref{defw0}, and if one takes the parameter
$(d_{0,0}, d_{2,0}, d_{4,0},d_{4,2},d_{6,0})=(0,0,0,0,0)$ and formally
replaces $s_0$ by $\sigma$ in the definition \eqref{defw0} of initial
data $w_0(y,s_0)$, then, we recover $\varphi(y,\sigma)$ defined in
\eqref{goodprof}. In addition, the point $a$ we consider is given by
\eqref{defk'l'} with $K'+L'=A$, which falls in the framework
considered in Section \ref{seclargekl}. Therefore, Lemma
\ref{propdeco0} applies and we see that
 \begin{align}
    \varphi_a(y,\sigma)=\kappa-\iota'
     -e^{-\sigma}&\left\{2K'{L'}^2h_1h_0+2{K'}^2L'h_1h_2+{L'}^2h_2h_0+4K'L'h_1h_1
      +{K'}^2h_0h_2\right. \nonumber\\
    &+\left.2L'h_2h_1+2K'h_1h_2+h_2h_2\right\}
        +O\left(\frac {\iota'} A\right)+O({\iota'}^2), \label{expfa}
  \end{align}
in $L^r_\rho$ for any $r\ge 2$,  where $\iota'$ is given in \eqref{defiota'}.

\bigskip

\textbf{Step 2: The expansion of $q_a(y,\sigma)$}

Take $r\ge 2$. Using the decomposition \eqref{decomp} for $q(y,\sigma)$,
we write from \eqref{defqa}:  
\begin{equation}\label{decomqa}
  q_a(y,\sigma) = \bar q_a(y,\sigma) +\barbelow q_a(y,\sigma)
  \end{equation}
  where
  \begin{equation}\label{defqb}
  \bar q_a(y,\sigma) =
  \sum_{i=0}^7\sum_{j=0}^iq_{i,j}(\sigma) h_{i-j}(y_1+K')h_j(y_2+L'),
\;\; \barbelow   q_a(y,\sigma) =q_-(y_1+K',y_2+L',\sigma).
\end{equation}
Concerning $\bar  q_a(y,\sigma)$,
recalling that $q(\sigma)\in V_A(\sigma)$ defined in Definition
\ref{defvas}, then proceeding as for the proof of Lemma
\ref{propdeco0}
given at the beginning of this subsection,
we derive that  
\begin{align}
\bar q_a(y,\sigma) =   q_{6,2}(\sigma)& \left\{
                   {L'}^4h_2h_0 +{K'}^4 h_0h_2
                   +4{L'}^3 h_2h_1 +4{K'}^3 h_1h_2
                   +6({K'}^2+{L'}^2)h_2h_2\right.\nonumber\\
                   &\left.+4K'h_3h_2+4L'h_2h_3
                     +h_4h_2+h_2h_4\right\}
+\left(\frac {\iota'} A\right) \label{q1}
\end{align}
in $L^r_\rho$.
As for $\barbelow  q_a(y,\sigma)$, we can bound it thanks to the following parabolic
regularity estimate on $q_-(\sigma)$:
\begin{lem}[Parabolic regularity for $q_-(\sigma)$]\label{lemq-}
  Under the hypotheses of Lemma \ref{propdeco0'}, it holds that for
  all $r'\ge 2$, 
\[
  \forall s\in [s_0, \sigma],\;\;
  \|q_-(s)\|_{L^{r'}_\rho} \le C(r')A^2 s^2e^{-3s}. 
\]
\end{lem}
Indeed, using \eqref{defqb}, we write
\begin{equation}\label{q2}
  \int |\barbelow q_a(y,\sigma)|^r \rho(y) dy
 = \int |q_-(z,\sigma)|^r \rho(z_1-K',z_2-L') dz.
\end{equation}
Then, by definition \eqref{defro}, we write
\[
\rho(z_1-K',z_2-L') = \rho(z) e^{\frac{K'z_1+L'z_2}2}e^{\frac{{K'}^2+{L'}^2}4}.
\]
Using the Cauchy-Schwarz inequality together with Lemma \ref{lemq-},
we write
\begin{align}
  &\int |q_-(z,\sigma)|^r \rho(z_1-K',z_2-L') dz\nonumber\\
  \le\;& e^{\frac{{K'}^2+{L'}^2}4}  \left(\int |q_-(z,\sigma)|^{2r} \rho(z) dz\right)^{\frac 12}
  \left(\int e^{K'z_1+L'z_2} \rho(z) dz\right)^{\frac 12}\nonumber\\
 \le\;& e^{\frac 34 ({K'}^2+{L'}^2)}  \left(\int |q_-(z,\sigma)|^{2r}
   \rho(z) dz\right)^{\frac 12}
  \le C  e^{\frac 34 ({K'}^2+{L'}^2)}  (A^2\sigma^2 e^{-3 \sigma})^r.\label{q3}
\end{align}
Using \eqref{k'l'} and recalling the definition \eqref{defiota'} of
$\iota'$, then taking  $s_0$ large enough (remember that $\sigma \ge s_0$), we
see that 
\begin{equation}\label{q4}
   e^{\frac 3{4r} ({K'}^2+{L'}^2)}  A^2\sigma^2 e^{-3 \sigma}
  \le e^{-2\sigma}\frac{A^6}A
  =e^{-2\sigma}\frac{(K'+L')^6}A
  \le \frac {2^6} A e^{-2\sigma} ({K'}^6+{L'}^6)
  \le \frac{2^6\iota'}{A\delta}.
\end{equation}
Using \eqref{q1}, \eqref{q2}, \eqref{q3} and \eqref{q4}, we obtain an expansion
for $q_a(z,\sigma)$, by \eqref{decomqa}. Adding the expansion
\eqref{expfa}, we obtain the desired expansion \eqref{expwas} in Lemma
\ref{propdeco0'}.  It remains to justify Lemma \ref{lemq-}.

\medskip

\begin{proof}[Proof of Lemma \ref{lemq-}]
  Assume here that $s_0\ge \sinit(A)$ defined in Proposition
  \ref{propdyn}, so that Proposition \ref{propdelay} applies. Consider
  then some $r'\ge 2$. Proceeding as for item (ii) of Proposition
  \ref{propdyn}, we project equation \eqref{eqq} for all
  $s\in [s_0, \sigma]$ as follows
  \[
    \partial_s q_- = \q L q_-+G \mbox{ where }
      G = P_-(Vq+B+R) 
  \]
    and $P_-$ is the $L^2_\rho$ projector on the subspace $E_-$ \eqref{defE-}.
  Given some $\sigma_0\in [s_0, \sigma]$, we may write a Duhamel
  formulation based on the kernel given in \eqref{kernel}:
  \[
    q_-(\sigma) = e^{(\sigma-\sigma_0)\q L}q(\sigma_0)
    +\int_{\sigma_0}^\sigma e^{(\sigma-\sigma')\q L}G(\sigma')d\sigma'.
  \]
  Taking the $L^{r'}_\rho$ norm, we write
  \begin{equation}\label{q-ode}
    \|q_-(\sigma)\|_{L^{r'}_\rho} \le I + II \equiv
    \|e^{(\sigma-\sigma_0)\q L}q(\sigma_0) \|_{L^{r'}_\rho} 
    +\int_{\sigma_0}^\sigma\| e^{(\sigma-\sigma')\q L}G(\sigma') \|_{L^{r'}_\rho}  d\sigma'.
  \end{equation}
  We start by bounding $II$. We claim that for any $\sigma'\in [s_0,
  \sigma]$, $G(\sigma')$ and its gradient have polynomial
  growth in $y$, allowing the application of item (ii) of Lemma
  \ref{lemVel}. Indeed, by definition \eqref{defR}, together with
  the $L^\infty$ bound on $q(s)$ from Definition \ref{defvas} and the
  gradient estimate of Proposition \ref{propgrad}, we see that
  $Vq+B(q)+R$ and its gradient have polynomial growth in $y$. By definition
  of the $P_-$ operator (see \eqref{defE-} and \eqref{decomp}), so has
  $G$. Applying item (ii) of Lemma \ref{lemVel}, we see that
  \begin{equation}\label{bound2}
    |II|\le \int_{\sigma_0}^\sigma \|G(\sigma')\|_{L^{r'}_\rho} d\sigma'.
  \end{equation}
  Since
  \[
    \|P_-(g)\|_{L^{r'}_\rho}\le C(r')\|g\|_{L^{r'}_\rho}
    \mbox{ for any }g\in L^{r'}_\rho
  \]
   (see again \eqref{defE-} and \eqref{decomp}), we write
  \[
    \|G(\sigma')\|_{L^{r'}_\rho} \le \|Vq(\sigma')\|_{L^{r'}_\rho} +
    \|B(q(\sigma'))\|_{L^{r'}_\rho} + \|R_-(\sigma')\|_{L^{r'}_\rho}.
  \]
  Using Proposition \ref{propdelay} and proceeding as for item (ii) of
  Proposition \ref{propdyn}, we see that
  \begin{equation}\label{boundG}
    \|G(\sigma')\|_{L^{r'}_\rho} \le  C(r')A \sigma' e^{-3 \sigma'}.
   \end{equation}
  As for the term $I$, introducing $\sigma^*(r')>0$ and $C^*(r')>0$
  such that the following delay regularizing effect holds for any
  $v\in L^2_\rho$ (see item (i) of Lemma \ref{lemVel}): 
 \begin{equation}\label{estvel1}
\|e^{\sigma^*(r')\q L}(v)\|_{L^{r'}_\rho} \le C^*(r')\|v\|_{L^2_\rho},
 \end{equation}
 we distinguish two cases in the following:\\
 \textbf{Case 1}: $\sigma\ge s_0+\sigma^*$. Fixing $\sigma_0 =
 \sigma-\sigma^*$, we see that $\sigma_0\ge s_0$. Using
 \eqref{estvel1}, we see by definition \eqref{q-ode} of $I$ and
 Definition \ref{defvas} of $V_A(s)$ that
 \[
   |I|\le C^*(r')\|q_-(\sigma-\sigma^*)\|_{L^2_\rho}
   \le C^*(r') A^2 (\sigma-\sigma^*)^2 e^{-3(\sigma -\sigma^*)}.
  \]
 Using \eqref{q-ode}, \eqref{bound2} and \eqref{boundG}, we see that
 \[
   \|q_-(\sigma)\|_{L^{r'}_\rho}
   \le C^*(r') A^2 \sigma^2 e^{-3(\sigma-\sigma^*)}
   +\sigma^* C(r') A (\sigma-\sigma^*) e^{-3(\sigma-\sigma^*)}
   \le \bar C(r') A^2\sigma^2 e^{-3\sigma}
 \]
 and the conclusion of Lemma \ref{lemq-} follows. \\
 \textbf{Case 2}:  $s_0\le\sigma\le s_0+\sigma^*$. Fixing $\sigma_0 = s_0$,
 and noting that $q(s_0)$ and $\nabla q(s_0)$ are bounded (see the
 hypotheses of Lemma \ref{propdeco0'} and Definition \ref{defvas} of
 $V_A(s_0)$), we can apply item (ii) of Lemma \ref{lemVel} and write
 by definition \eqref{q-ode} of $I$ and Definition \ref{defvas}:
\[
|I|\le C e^{\sigma-s_0}\|q_-(s_0)\|_{L^2_\rho} 
\le C e^{\sigma-s_0} A^2s_0^2 e^{-3s_0}
\le  C e^{\sigma^*} A^2 \sigma^2 e^{-3(\sigma-\sigma^*)},
\]
and the conclusion follows as in Case 1.
  This concludes the proof of Lemma \ref{lemq-}.
\end{proof}
This concludes the proof of Lemma \ref{propdeco0'} too.

\end{proof}

  \appendix

\section{A classical parabolic regularity estimate for equation \eqref{eqw}}\label{appreg}
  We prove Lemma \ref{propara} here. Since the argument was
  extensively used in our earlier papers, we won't give 
  details.
   \begin{proof}[Proof of Lemma \ref{propara}]
    We proceed in 2 steps: we first justify the estimate locally in space, then
    we extend it to the whole space.\\
   \textit{Step 1: Proof of a local version, where $\m R^2$ is
  replaced by $B(0,1)$, the unit ball of $\m R^2$}. 
  When restricting to $B(0,1)$, one can use the similarity variables'
  transformation to translate the problem into a regularity question
  for equation \eqref{equ}. Using the technique of Step 2 page 1060 of
  \cite{MZgafa98}, which relies on Theorem 3 page 406 of Friedman
  \cite{Fjmm58}, we get the result.\\
  \textit{Step 2: Extension to the whole space $\m R^2$}.
Introducing for any $a\in \m R^2$,
\begin{equation}\label{wanwn}
w_{a,n}(y,s) = w_n(y+ae^{\frac s2},s),
\end{equation}
we see from the similarity variables transformation \eqref{defw} that
$w_{a,n}$ is also a solution of \eqref{eqw} defined for all $s\in
  [s_{0,n},s_{2,n}]$ and satisfying the uniform bound
  \eqref{prop1}. Applying the same local regularity technique on
  $w_{a,n}$ as in Step 1,
  we show that \eqref{regul} holds also for $w_{a,n}$ with $\m R^2$
  replaced by $B(0,1)$, uniformly in $a\in \m R^2$. Using
\eqref{wanwn} and varying $a$ in the whole space $\m R^2$, we recover
the full estimate \eqref{regul} (on $\m R^2$) for $w_n$.
This concludes the proof of Lemma \ref{propara}.
\end{proof}

\section{Stability results for equation \eqref{eqw}}\label{sectrap}
This section is devoted to the proof of Propositions \ref{proptrap} and
\ref{proptrap2}.

\medskip

\begin{proof}[Proof of Proposition \ref{proptrap}]
  Consider $w$ a solution of equation \eqref{eqw} defined for all $(y,s)\in \m
  R^2\times [\suno,\sdue]$ for some $\sdue\ge \suno$, with
\begin{equation}\label{normeinf}  
  |w(y,s)|\le 2 \kappa \mbox{ and }
\nabla w(0)(1+|y|)^{-k}\in L^\infty
\end{equation}
 for some $k\in \m N$.
We aim at proving that
\begin{equation}\label{hadaf}
  \forall s\in [\suno,\sdue],\;\; \|w(s)\|_{L^2_\rho}
  \le M_0 \|w(\suno)\|_{L^2_\rho}
 e^{-\frac{s}{p-1}},
\end{equation}
provided that
\begin{equation}\label{cond}
  \|w(\suno)\|_{L^2_\rho}\le \epsilon_0,
  \end{equation}
for some large $M_0\ge 1$ and small $\epsilon_0>0$.\\
We will assume that
\begin{equation}\label{assume}
\|w(\suno)\|_{L^2_\rho}>0,
\end{equation}
  otherwise $w\equiv 0$ and \eqref{hadaf} is trivial.\\
Since $\Delta w - \frac 12 y \cdot \nabla w = \frac 1 \rho\nabla \cdot
\left(\rho \nabla w\right)$ by definition \eqref{defro} of $\rho$, multiplying equation \eqref{eqw} by
$w\rho$, integrating in space, then using an integration by parts, we write for all $s\in [\suno,\sdue]$:
\begin{equation}\label{ode0}
\frac 12  \frac d{ds} \int w(y,s)^2 \rho(y) dy \le
-\frac 1{p-1} \int w(y,s)^2\rho(y) dy+\int |w(y,s)|^{p+1} \rho(y) dy. 
\end{equation}
Note that the fact that $w(0)\in L^\infty$ and $\nabla
w(0)(1+|y|)^{-k}\in L^\infty$ for some $k\in \m N$ (see
\eqref{normeinf}) is
important to justify this integration by parts,
as it is the case in item (ii) of Lemma \ref{lemVel}.
The conclusion will follow from 2 arguments: a rough estimate for
general data, then a delicate estimate for small data. In the final
step, we combine both arguments to conclude.

\medskip

\textbf{Step 1: A rough estimate for general data}

Using \eqref{ode0} together with \eqref{normeinf} and the definition
\eqref{conv}
of $\kappa$, we write  for all $s\in [\suno,\sdue]$,
\[
\frac 12  \frac d{ds} \int w(y,s)^2 \rho(y) dy \le
  \frac {2^{p-1}-1}{p-1}\int w(y,s)^2\rho(y) dy,
\]
hence,
\begin{equation}\label{rough}
\|w(s)\|_{L^2_\rho}\le e^{\frac{(2^{p-1}-1)}{p-1}s}\|w(\suno)\|_{L^2_\rho}.
\end{equation}

\medskip

\textbf{Step 2: A delicate estimate for small data and large $\sdue$}

Using again equation \eqref{eqw}, together with 
\eqref{normeinf} and the definitions
\eqref{conv}
and \eqref{defR}
of $\kappa$ and $\q L$, we write for almost every $(y,s) \in \m R^2 \times
[\suno,\sdue]$,
\[
\partial_s|w|\le (\q L-1+\frac{(2^{p-1}-1)}{p-1})|w|.
\]
Using the regularizing effect of Lemma \ref{lemVel}, we derive the
existence of $\stella>0$ and $C^*>0$ such that
if $\sdue\ge \stella$,
then
for all $s\in[\stella,\sdue]$,
we have
\begin{equation}\label{reg}
\|w(s)\|_{L^{p+1}_\rho}\le C^*\|w(s-\stella)\|_{L^2_\rho}.
\end{equation}
Now, if $s\in[\suno,\min(\stella, \sdue)]$,
we use the $L^\infty$ bound
\eqref{normeinf} and the definition
\eqref{conv}
of $\kappa$ to derive the
following rough control of the nonlinear term:
\[
\|w(s)\|_{L^{p+1}_\rho}^{p+1}\le (2\kappa)^{p-1}\|w(s)\|_{L^2_\rho}^2
= \frac{2^{p-1}}{p-1}\|w(s)\|_{L^2_\rho}^2.
\]
Using \eqref{ode0} with these 2 controls of the nonlinear term, we write for all
$s\in[\suno,\sdue]$,
the following delay differential inequality:
\begin{align}
\frac 12  \frac d{ds} \|w(s)\|_{L^2_\rho}^2 \le &- \frac
                                                  1{p-1}\|w(s)\|_{L^2_\rho}^2 \label{delay1}\\
  &+ \1_{\{\suno\le s\le  \stella\}}
\frac{2^{p-1}}{p-1}\|w(s)\|_{L^2_\rho}^2+
    \1_{\{\stella\le s\le \sdue\}}
     (C^*)^{p+1}\|w(s-\stella)\|_{L^2_\rho}^{p+1}.\nonumber
\end{align}

\medskip

\textbf{Step 3: Conclusion of the proof}

Fixing $M_0>0$ and $\epsilon_0$ such that
\begin{equation}\label{defM0}
 M_0 =2\max\left(1, e^{\frac{2^{p-1}}{p-1}\stella}\right)
\mbox{ and }
 4\left(C^*M_0\right)^{p+1}\epsilon_0^{p-1}
    e^{\frac{2\stella}{p-1}}= M_0^2 
\end{equation}
we are ready to finish the proof of \eqref{hadaf}, if \eqref{cond} and
\eqref{assume} hold. By continuity of the $L^2_\rho$ norm
of $w(s)$
\footnote{This is a consequence of the continuity in
$L^\infty$ for equation \eqref{equ}, through the transformation
\eqref{defw}}
and noting that $M_0>1$, if we proceed by contradiction and
assume that identity \eqref{hadaf} fails, then we may introduce $\bar
s\in (\suno,\sdue]$ such that
\begin{align}
  \forall s\in [\suno,\bar s],\;\;
  \|w(s)\|_{L^2_\rho}
\le M_0 \|w(\suno)\|_{L^2_\rho}
  e^{-\frac s{p-1}},
  \label{hypo}\\
   \|w(\bar s)\|_{L^2_\rho}=
 M_0 \|w(\suno)\|_{L^2_\rho}
e^{-\frac{\bar s}{p-1}}.
\label{contra}
\end{align}
We will reach a contradiction in each of the 2 cases we consider in the following.\\
\textit{Case 1:
$\bar s \le \stella$}.
In this case, using \eqref{rough}, the assumption \eqref{assume} and
the choice of $M_0$ in \eqref{defM0}, we write
\begin{align*}
\|w(\bar s)\|_{L^2_\rho}&\le
e^{\frac{(2^{p-1}-1)}{p-1}\bar s}\|w(\suno)\|_{L^2_\rho}
\le e^{\frac{(2^{p-1}-1)}{p-1}\stella}\|w(\suno)\|_{L^2_\rho}\\
                        &\le \frac{M_0}2\|w(\suno)\|_{L^2_\rho}
 e^{-\frac{\stella}{p-1}}
<
                          M_0 \|w(\suno)\|_{L^2_\rho}
 e^{-\frac{\bar s}{p-1}},
\end{align*}
and a contradiction follows by \eqref{contra}.\\
\textit{Case 2: $\stella \le \bar s \le \sdue$}. In this case, identity
\eqref{delay1} holds for any $s\in [\suno,\bar s]$. Using \eqref{rough}
when $\suno\le s \le \stella$ and \eqref{hypo} when $\stella \le s \le
\bar s$, we write for all $s\in [\suno,\bar s]$:
\begin{align*}
&  \frac 12 \frac d{ds}  \|w(s)\|_{L^2_\rho}^2\le
  -\frac 1{p-1}  \|w(s)\|_{L^2_\rho}^2\\
  &+
    \1_{\{\suno\le s\le  \stella\}}  \frac{2^{p-1}}{p-1}e^{\frac{2(2^{p-1}-1)}{p-1}s}
    \|w(\suno)\|_{L^2_\rho}^2\\
  &+    \1_{\{\stella\le s\le \bar s\}}  (C^*M_0 \|w(\suno)\|_{L^2_\rho})^{p+1}
 e^{-\frac{(p+1)(s-\stella)}{p-1}}.
\end{align*}
Integrating this equation, we see that
\begin{align*}
  \|w(\bar s)\|_{L^2_\rho}^2 &\le
  e^{-\frac{2\bar s}{p-1}}
  \{\|w(\suno)\|_{L^2_\rho}^2
    +e^{\frac{2\cdot 2^{p-1}\stella}{p-1}}\|w(\suno)\|_{L^2_\rho}^2
    +2\left(C^*M_0 \|w(\suno)\|_{L^2_\rho}\right)^{p+1}
    e^{\frac{2\stella}{p-1}}\}\\
&\le e^{-\frac{2\bar s }{p-1}}
                                      \left\{ \left(\frac{M_0}2\right)^2  \|w(\suno)\|_{L^2_\rho}^2
                                      +\left(\frac{M_0}2\right)^2
                                      \|w(\suno)\|_{L^2_\rho}^2
                                      +\frac{M_0^2}2
                                      \|w(\suno)\|_{L^2_\rho}^2\right\}\\
                                      &< \frac 34 M_0^2e^{-\frac{2\bar s}{p-1}} \|w(\suno)\|_{L^2_\rho}^2
\end{align*}
thanks to the definition \eqref{defM0} of $M_0$ and $\epsilon_0$,
together with \eqref{assume}. Thus,
a contradiction follows from \eqref{contra}. This concludes the proof
of Proposition \ref{proptrap}.
  \end{proof}

  \bigskip

  Now, we give the proof of Proposition \ref{proptrap2}.
  \begin{proof}[Proof of Proposition \ref{proptrap2}]
Consider $w$ a solution of equation of \eqref{eqw}
  defined for all $(y,s)\in \m R^2\times [\suno,\sdue]$ for some 
  $\sdue\ge \suno$, with
  \begin{equation}\label{w2k}
    \nabla w(0)(1+|y|)^{-k}\in L^\infty\mbox{ and }
    |w(y,s)|\le 2 \kappa,
   \end{equation}
 for some
  $k\in \m N$. We will prove that for some universal
 constant $M_1\ge 1$, if
\begin{equation}\label{init}
  \|w(\suno)-\psi(\sunos)\|_{L^2_\rho}\equiv \epsilon_1
\le \frac {|{\psi'}(\sunos)|}{M_1}
\end{equation}
  for some $\sunos\in \m R$, where
   $\psi$ is defined in \eqref{defpsi}, then, it holds that
\begin{equation}\label{finalg}
  \forall s\in [\suno,\sdue],\;\; \|w(s)-\psi(s+\sunos)\|_{L^2_\rho}
  \le M_1 \|w(\suno) -\psi(\sunos)\|_{L^{\bar p+1}_\rho}
\frac{|{\psi'}(s+\sunos)|}{|{\psi'}(\sunos)|}.
\end{equation}
We may assume that
\begin{equation}\label{assume1}
  \epsilon_1>0,
\end{equation}
otherwise $w(y,s)=\psi(s+\sunos)$ for any $s\ge \suno$,
from the uniqueness of solutions to equation \eqref{eqw}, and \eqref{finalg}
is trivial.

\medskip

Since $w$ and $\psi$ are both solutions of \eqref{eqw}, introducing
\begin{equation}\label{defv}
\bar \psi(s) = \psi(s+\sunos)\mbox{ and }
v(y,s) = w(y,s)- \bar\psi(s),
\end{equation}
we write the following PDE satisfied by $v$, for all
$(y,s)\in \m R^2\times  [\suno, \sdue]$:
\begin{equation}\label{eqv}
\partial_s v = \Delta v - \frac 12y\cdot \nabla v - \frac v{p-1}
+p|\tilde w|^{p-1}v,
\end{equation}
where
\begin{equation}\label{deftw}
  \tilde w(y,s)\in [w(y,s), \bar \psi(s)].
  \end{equation}
Arguing as for \eqref{ode0}, we derive the following identity from \eqref{eqv},
for all $s\in [\suno, \sdue]$:
\begin{equation}\label{ode1}
\frac 12 z'(s) \le - \frac{z(s)}{p-1}
+p\int |\tilde w(y,s)|^{p-1}v(y,s)^2 \rho(y) dy,
\mbox{ where }z(s) = \int v(y,s)^2 \rho(y) dy. 
\end{equation}
The fact that $\nabla w(0)(1+|y|)^{-k}\in L^\infty$ is useful to
justify the integration by parts in \eqref{ode1} and elsewhere.
We proceed in 2 steps, first deriving a differential inequality for
$z(s)$, then using a Gronwall argument to conclude.

\medskip

\textbf{Step 1: A differential inequality on $z(s)$}.

Since $\tilde w$ and $\bar \psi$ are bounded by \eqref{defpsi}, 
\eqref{deftw} and \eqref{w2k}, 
using the definitions \eqref{deftw} and \eqref{defv} of
$\tilde w$ and $v$, we write by continuity:
\[
  \left||\tilde w(y,s)|^{p-1}-\bar \psi^{p-1} \right|
  \le C_0 |\tilde w(y,s) - \bar \psi(s)|^{\bar p -1}
  \le C_0|v(y,s)|^{\bar p -1}
\]
for some $C_0>0$, where $\bar p = \min(p,2)$.
Plugging this in \eqref{ode1}, we write
\begin{equation}\label{ode2}
  \frac 12 z'(s) \le [- \frac 1{p-1} +p \bar \psi(s)^{p-1}]z(s)
  + C_0 \int|v(y,s)|^{\bar p+1}\rho(y) dy.
\end{equation}
Let us now bound $\|v(s)\|_{L^{\bar p+1}_\rho}$.
Using equation \eqref{eqv}, the bound \eqref{w2k}, the definitions
\eqref{deftw} and \eqref{defpsi} of $\tilde w$ and $\psi$, together with 
the definitions
\eqref{conv}
and \eqref{defR}
of $\kappa$ and $\q L$, we write for almost every $(y,s) \in \m R^2 \times
[\suno,\sdue]$,
\begin{equation}\label{eqbv}
\partial_s|v|\le (\q L-1+\frac{(2^{p-1}p-1)}{p-1})|v|.
\end{equation}
Arguing as for \eqref{reg}, we see that if $\sdue\ge \stella$, 
then, we have for all $s\in[\stella,\sdue]$,
\begin{equation}\label{after}
  \|v(s)\|_{L^{\bar p+1}_\rho}\le \bar C\|v(s-\stella)\|_{L^2_\rho}
  =\bar Cz(s-\stella)^{\frac 12},
\end{equation}
for some possibly different $\stella(p)>0$ and $\bar C>0$.
Now, if $s\in[\suno,\min(\stella, \sdue)]$, using \eqref{w2k} and the
definition \eqref{defpsi} of $\psi$, we see by definition \eqref{defv}
of $v$ that $|v|\le 3\kappa$ and $\nabla v(0)(1+|y|)^{-k}\in
L^\infty$. Therefore,
\begin{equation}\label{funes}
  \int|v(y,s)|^{\bar p+1}\rho(y) dy
  \le (3\kappa)^{\bar p -1}\int v(y,s)^2\rho(y) dy.
\end{equation}
In addition, using \eqref{eqbv}, we see that  we can apply item (ii) of
Lemma \ref{lemVel} and get from \eqref{init}
\[
\forall s\in[\suno,\min(\stella, \sdue)],\;\;
  \|v(s)\|_{L^2_\rho}\le C^*\|v(\suno)\|_{L^2_\rho}
  = C^*\epsilon_1.
\]
Using this together with \eqref{funes}, \eqref{ode2} and \eqref{after}, we see that
for all $s\in [\suno, \sdue]$, 
\begin{align}
  \frac 12 z'(s) \le & [-\frac 1{p-1}+p\bar \psi(s)^{p-1}]z(s) \label{delay2}\\
&+C_1\1_{\{\suno\le s\le  \stella\}} \epsilon_1^2
+ C_1\1_{\{\stella\le s\le \sdue\}} z(s-\stella)^{\frac{\bar p+1}2},\nonumber
\end{align}
for some $C_1>0$.

\bigskip

\textbf{Step 2: A Gronwall estimate}

Let us define
\begin{equation}\label{defza}
\bar z_p(s) = z_p(s+\sunos) \mbox{ where }
z_p(s) = \frac{e^{-\frac{2s}{p-1}}}{(1+e^{-s})^{\frac{2p}{p-1}}}
  =\left[\frac{p-1}\kappa \psi'(s)\right]^2
\end{equation}
and $\psi$ is defined in \eqref{defpsi}. Since $\psi(s)$ satisfies
equation \eqref{eqw}, it follows that $\bar z_p(s)$ 
is a solution of the linear part of \eqref{delay2},
namely
\begin{equation}\label{eqza}
\bar z_p'(s) = 2\left(-\frac 1{p-1}+p\bar \psi(s)^{p-1}\right)\bar z_p(s).
\end{equation}
Then, we introduce the following barrier
\begin{equation}\label{defbz}
\bar z(s) =\frac{M_1'\epsilon_1^2}{z_p(\sunos)}\bar z_p(s),
\end{equation}
where $M_1'>1$ will be fixed large enough later.
With this definition and recalling the definition \eqref{ode1} of
$z(s)$, we suggest to prove that 
\begin{equation}\label{finalg'}
  \forall s\in [\suno,\sdue],\;\;  z(s) \le \bar z(s),
\end{equation}
if $\epsilon_1$ defined in \eqref{init} is small enough, which clearly
implies \eqref{finalg},
by definition \eqref{defza} of $z_p$.
We proceed by contradiction and assume that identity \eqref{finalg'} fails. Since
\begin{equation}\label{pos}
  0<z(\suno)= \epsilon_1^2< M_1'\epsilon_1^2= \bar z(0),
\end{equation}
by \eqref{ode1}, \eqref{defv}, \eqref{init},
\eqref{assume1}, \eqref{defbz} and \eqref{defza},
using the continuity in time of the $L^2_\rho$ norm of $v(s)$ solution of
equation \eqref{eqv} (which is a consequence of the continuity in
$L^\infty$ for equation \eqref{equ}, through the transformations
\eqref{defw} and \eqref{defv}),
we see that \eqref{finalg'} holds at least on a
small interval to the right of $\suno$. Hence,
we may introduce $\bar s\in (\suno,\sdue]$ such that
\begin{align}
  \forall s\in [\suno,\bar s],\;\;
  z(s) \le \bar z(s),\label{hypo1}\\
  z(\bar s) =\bar z(\bar s). \label{contra1}
 \end{align}
 Using the differential inequality \eqref{delay2} together with the
 auxiliary function $\bar z_p$ which satisfies equation \eqref{eqza}, \eqref{pos}
 and \eqref{hypo1}, we write the following Gronwall estimate:
  \[
   z(\bar s) \le \bar z_p(\bar s)
   \left\{
     \frac{\epsilon_1^2}{\bar z_p(\suno)}
    +C_1 \epsilon_1^2J_1
     +C_1\left(\frac{M_1'\epsilon_1^2}{z_p(\sunos)}\right)^{\frac{\bar
         p+1}2}J_2
     \right\}
 \]
 where
 \[
   J_1 = \int_{\suno}^{\stella}\frac{d\sigma}{\bar z_p(\sigma)}
   \mbox{ and } J_2  =\int_{\stella}^{\bar s}
   \frac{\bar z_p(\sigma-\stella)^{\frac{\bar p+1}2}}{\bar z_p(\sigma)}d\sigma.
 \]
 By definition \eqref{defza}, we see that for all $s'\in[0,\stella]$, we
 have $z_p(\sunos+s') \ge
 z_p(\sunos)e^{-2\frac{(p+1)}{p-1}\stella}$, hence
 \[
  J_1 = \int_{\sunos}^{\sunos+\stella} \frac{d\sigma'}{z_p(\sigma')}
   \le \frac{\stella e^{2\frac{(p+1)}{p-1}\stella}}{z_p(\sunos)}.
 \]
 We also have
 \[
 J_2  =\int_{\sunos+\stella}^{\sunos+\bar s}
         \frac{z_p(\sigma'-\stella)^{\frac{\bar  p+1}2}}
         {z_p(\sigma')}d\sigma'
         \le \int_{-\infty}^\infty
         \frac{z_p(\sigma'-\stella)^{\frac{\bar  p+1}2}}
         {z_p(\sigma')}d\sigma' \equiv C_2(s^*).
       \]
       Imposing that
       \begin{equation}\label{defM1}
         \epsilon_1^2\le \frac{z_p(\sunos)}{M_1'(C_1C_2(\stella))^{\frac 2{\bar p-1}}},
          \end{equation}
         we see that
         \[
           z(\bar s)
           \le \frac{\epsilon_1^2}{z_p(\sunos)}\bar  z_p(\bar s)
           [1+C_1 \stella  e^{2\frac{(p+1)}{p-1}\stella} +1].
           \]
           Fixing
           \[
M_1'=3+C_1 \stella  e^{2\frac{(p+1)}{p-1}\stella},
\]
we see that a contradiction follows from \eqref{contra1},
         \eqref{hypo1} and \eqref{defbz} (remember that $\bar z(\bar
         s)>0$ by definition \eqref{defbz}, together with
         \eqref{defza} and \eqref{assume1}). Thus, \eqref{finalg'}
         holds. Since $z_p(s) = C|{\psi'}(s)|^2$ from \eqref{defza},
        using the definitions \eqref{ode1}  and
         \eqref{defbz} of $z(s)$ and $\bar z(s)$, together with the
         condition \eqref{defM1}, we conclude the proof of Proposition
         \ref{proptrap2}.  
  \end{proof}

  \def\cprime{$'$} \def\cprime{$'$}


\noindent{\bf Address}:\\
CY Cergy Paris Universit\'e, D\'epartement de math\'ematiques, 
2 avenue Adolphe Chauvin, BP 222, 95302 Cergy Pontoise cedex, France.\\
\vspace{-7mm}
\begin{verbatim}
e-mail: merle@math.u-cergy.fr
\end{verbatim}
Universit\'e Sorbonne Paris Nord, Institut Galil\'ee, 
Laboratoire Analyse, G\'eom\'etrie et Applications, CNRS UMR 7539,
99 avenue J.B. Cl\'ement, 93430 Villetaneuse, France.\\
\vspace{-7mm}
\begin{verbatim}
e-mail: hatem.zaag@math.cnrs.fr
\end{verbatim}

\end{document}